\documentclass{amsart}
\usepackage{amsfonts}
\usepackage{amsmath}
\usepackage{amssymb}
\usepackage{mathrsfs}
\usepackage{color}
\usepackage{hyperref}

\newcommand\N{{\mathbb N}}
\newcommand\R{{\mathbb R}}
\newcommand\T{{\mathbb T}}
\newcommand\C{{\mathbb C}}

\def\AA{{\mathcal A}}
\def\BB{{\mathcal B}}
\def\CC{{\mathcal C}}
\def\DD{{\mathcal D}}
\def\EE{{\mathcal E}}
\def\FF{{\mathcal F}}
\def\GG{{\mathcal G}}

\def\LL{{\mathcal L}}

\def\SS{{\mathcal S}}
\def\TT{{\mathcal T}}

\def\BBB{{\mathscr B}}
\def\CCC{{\mathscr C}}

\def\GGG{{\mathscr G}}

\newcommand{\la}{\lesssim}

\newcommand\dt{{\frac{\mathrm d}{\mathrm dt}}}

\newcommand{\dd}{{\, \mathrm d}}

\newcommand{\n}[1]{{\left\| #1 \right\|}}

\newcommand{\fa}{\forall \,}
\newcommand{\sk}{\smallskip}
\newcommand{\mk}{\medskip}

\newcommand{\var}{\varepsilon}
\newcommand{\Id}{\hbox{Id}}

\let\oldmarginpar\marginpar
\renewcommand\marginpar[1]{\-\oldmarginpar[\raggedleft\footnotesize #1]%
{\raggedright\footnotesize #1}}
\newcommand{\Black}{\color{black}}

\def\eps{{\varepsilon}}

\newtheorem{theo}{Theorem}

\newtheorem{lem}[theo]{Lemma}
\newtheorem{cor}[theo]{Corollary}

\theoremstyle{definition}
\newtheorem{defin}[theo]{Definition}

\theoremstyle{remark}
\newtheorem{rem}[theo]{Remark}
\newtheorem{rems}[theo]{Remarks}

\newcommand{\beqn}{\begin{equation}}
\newcommand{\eeqn}{\end{equation}}
\newcommand{\bear}{\begin{eqnarray}}
\newcommand{\eear}{\end{eqnarray}}
\newcommand{\bean}{\begin{eqnarray*}}
\newcommand{\eean}{\end{eqnarray*}}

\newcommand{\ds}{\displaystyle}

 \def\signsm{\bigskip \begin{center} {\sc St\'ephane
      Mischler\par\vspace{1mm} Universit\'e Paris-Dauphine \& IUF\par
      CEREMADE, UMR CNRS 7534\par
      Place du Mar\'echal de Lattre de Tassigny \par
      75775 Paris Cedex 16 FRANCE\par\vspace{1mm} e-mail:}
    \tt{mischler@ceremade.dauphine.fr} \end{center}}

\def\signcm{\bigskip \begin{center} {\sc Cl\'ement
      Mouhot\par\vspace{1mm}
      University of Cambridge\par
      DPMMS, Centre for Mathematical Sciences\par
      Wilberforce road, Cambridge CB3 0WA, UK
      \par\vspace{1mm} e-mail:}
    \tt{C.Mouhot@dpmms.cam.ac.uk} \end{center}}

\author{S. Mischler, C. Mouhot}
\begin{document}

\title[Stability of slowly decaying solutions\dots]
{Exponential Stability of slowly decaying solutions to the
  Kinetic-Fokker-Planck equation}
%\title[Changing functional spaces for the Fokker-Planck equation]{Semigroup decay estimates when changing functional spaces
%with applications to Fokker-Planck equations}

% \footnotetext[1]{Department of Mathematics, The University of Texas
%   at Austin, 1 University Station C1200, Texas 78712, USA.  E-mail:
%   \texttt{gualdani@math.utexas.edu}}

% \footnotetext[2]{CEREMADE, Universit\'e Paris IX-Dauphine, Place du
%   Mar\'echal de Lattre de Tassigny, 75775 Paris, France.  E-mail:
%   \texttt{mischler@ceremade.dauphine.fr}}

% \footnotetext[3]{CEREMADE, Universit\'e Paris IX-Dauphine, Place du
%   Mar\'echal de Lattre de Tassigny, 75775 Paris, France.  E-mail:
%   \texttt{cmouhot@ceremade.dauphine.fr}}

\begin{abstract} 
  The aim of the present paper is twofold:
  
  (1) We carry on with developing an abstract method for deriving
  decay estimates on the semigroup associated to non-symmetric
  operators in Banach spaces as introduced in \cite{GMM}.  We extend
  the method so as to consider the \emph{shrinkage} of the functional
  space. Roughly speaking, we consider a class of operators writing as
  a dissipative part plus a mild perturbation, and we prove that if
  the associated semigroup satisfies a decay estimate in some
  reference space then it satisfies the same decay estimate in
  another---smaller or larger---Banach space under the condition
  that a certain iterate of the ``mild perturbation'' part of the
  operator combined with the dissipative part of the semigroup maps
  the larger space to the smaller space in a bounded way. The
  cornerstone of our approach is a factorization argument, reminiscent
  of the Dyson series.
  
  (2) We apply this method to the kinetic Fokker-Planck equation when
  the spatial domain is either the torus with periodic boundary
  conditions, or the whole space with a confinement potential.  We
  then obtain spectral gap estimates for the associated semigroup for
  various metrics, including Lebesgue norms, negative Sobolev norms,
  and the Monge-Kantorovich-Wasserstein distance $W_1$.
% \textcolor{red}{cm~: mise en avant du traitement de la geometrie du
%   potentiel de confinement, et de l'hypocoercivite sur distance faible
% de proba sur la loi} {\Red cm:
% revenir a la fin pour avoir une vue d'ensemble}
%  \medskip
%  \tiny   (3) and we prove an almost equivalence between the Poincar\'e inequality
%  and the decay estimates in large spaces.
%    
%
%  We apply  this approach to  the Fokker-Planck equation in the spatially
%  inhomogeneous case in torus and in the all space {\Red with confinement potential}. 
% We also give a comprehensible proof ``\`a la Nash" of the smoothing effect of the 
%  kinetic  Fokker-Planck equation. 
%
%  
  \end{abstract}

\maketitle 

\bigskip

One-parameter semigroups and linear evolution equations [See also
34G10, 34K30], 35P15 Estimation of eigenvalues, upper and lower
bounds, 47H20 Semigroups of nonlinear operators [See also 37L05,
47J35, 54H15, 58D07], 35Q84 Fokker-Planck equations, 76P05 Rarefied
gas flows, Boltzmann equation [See also 82B40, 82C40, 82D05].
\medskip


\bigskip

\begin{center} {\bf Preliminary version of \today}
\end{center}

\tableofcontents

%%%%%%%%%%%%%%%%%%%% Introduction %%%%%%%%%%%%%%%%%%%%%%%%%%%%%%%%%%%

\section{Introduction}
\label{sec:intro}
\setcounter{equation}{0}
\setcounter{theo}{0}

\subsection{The question at hand}
\label{sec:question-at-hand}

\smallskip This paper deals with the study of decay properties of
linear semigroups and their link with spectral properties as well as
some applications to the Fokker-Planck equations with various types of
confinement. It continues the program of research \cite{Mcmp,GMM}
where quantitative methods for enlarging the functional space of
spectral gap estimates were developed with application to kinetic
equations; specifically in \cite{GMM} spectral gap estimates were
obtained in Lebesgue spaces for Boltzmann and Fokker-Planck equations
in the spatially homogeneous and spatially periodic frameworks.

\smallskip Our approach is based on the following abstract question:
consider two Banach spaces $E \subset \EE$ with $E$ is dense in $\EE$,
and two unbounded closed linear operators $L$ and $\LL$ respectively
on $E$ and $\EE$ with spectrum $\Sigma(L), \Sigma(\LL) \subset \C$,
which are assumed to generate $C_0$-semigroups $(\mathscr S_L(t))_{t \ge 0}$
on $E$ and $(\mathscr S_\LL(t))_{t \ge 0}$ on $\EE$ respectively and
so that $\LL_{|E} = L$; can one deduce quantitative informations on
$\Sigma(\LL)$ and $\mathscr S_\LL(t)$ in terms of informations on
$\Sigma(L)$ and $\mathscr S_L(t)$ (\emph{enlargement} issue), or can
one deduce quantitative informations on $\Sigma(L)$ and
$\mathscr S_L(t)$ in terms of informations on $\Sigma(\LL)$ and
$\mathscr S_\LL(t)$ (\emph{shrinkage} issue)?

%The abstract question we address is the following: consider
%two Banach spaces $E \subset \EE$, and two unbounded closed linear
%operators $L$ and $\LL$ resp. on $E$ and $\EE$ with spectrum
%$\Sigma(L), \Sigma(\LL) \subsetneq \C$ and so that $\LL_{|E} = L$, and
%$E$ is dense in $\EE$, then can one deduce quantitative informations
%on $\Sigma(\LL)$ and $\SS(t) {\Blue =S_\LL(t)}$ in terms of informations on $\Sigma(L)$
%and $S(t){\Blue =S_L(t)}$ (enlargement issue), or can one deduce quantitative informations
%on $\Sigma(L)$ and $S(t)$ in terms of informations on $\Sigma(\LL)$
%and $\SS(t)$ (shrinkage issue)?

% Assuming that
% some informations on the localization of the spectrum $\Sigma(L)$ and
% on the asymptotic behavior of the semigroup $S(t)$ are known, can we
% obtain similar informations on the localization of the spectrum
% $\Sigma(\LL)$ and on the asymptotic behavior of the semigroup
% $\SS(t)$? In other words, can we carry over the larger space $\HH$
% some information we have on $\LL$ in the smaller space $H$?

We prove, under some assumptions discussed below, (i) that the
spectral gap property of $L$ in $E$ (resp. of $\LL$ in $\EE$) can be
shown to hold for $\LL$ in the space $\EE$ (resp. for $L$ in $E$) and
(ii) explicit estimates on the rate of decay of the semigroup
$\mathscr S_\LL(t)$ (resp. the semigroup $\mathscr S_L(t)$) can be
computed from the ones on $\mathscr S_L(t)$ (resp.
$\mathscr S_\LL(t)$). This holds for a class of operators $\LL$ which
split as $\LL = \AA + \BB$, where $\AA$ is bounded, $\BB$'s spectrum
is well localized and some appropriate combination of $\AA$ and the
semigroup $\mathscr S_\BB(t)$ of $\BB$ has some regularising
properties. This last ``semigroup commutator condition'' is
reminiscent of H\"ormander's commutator conditions~\cite{MR0222474}.

The Fokker-Planck equations we consider are then shown to belong to
this general class of operators and, as a consequence, we extend the
hypocoercivity results---usually obtained in $L^2$ or $H^1$ spaces
with inverse Gaussian type tail and endowed with convenient twisted
scalar product---into sharp exponential decay estimates on the
semigroup in many larger Lebesgue and Sobolev spaces.

\subsection{The abstract result}
\label{sec:main-abstract-result}

We denote $\mathscr C(E)$ the set of closed operators on a Banach
space $E$, $\mathscr B(E)$ the set of bounded operators on $E$, and
$\mathscr B(E,\EE)$ the set of bounded operators between two Banach
spaces. We say that $P \in \mathscr C(E)$ is \emph{hypodissipative} if
it is \emph{dissipative} for some norm equivalent to the canonical
norm of $E$ and we say that $P$ is dissipative for the norm $\| \cdot
\|$ on $E$ if 
\begin{equation*}
  \fa f \in \mbox{Domain}(P), \ \fa f^* \in E^* \, \mbox{ s.t. } \, \langle f, f^*
  \rangle = \n{f}_E^2 = \n{f^*}_{E^*} ^2, \quad \Re e \, \langle Pf, f^*
  \rangle \le 0
\end{equation*}
where the $\langle \cdot, \cdot \rangle$ denotes the duality bracket
between $E$ and its dual $E^*$. Finally we denote
$\Delta_a := \{ z \in \C;$ $\Re e \, z > a \}$.

\begin{theo}[Change of the functional space of the semigroup
  decay] \label{theo:Extension} Given $E$, $\EE$, $L$, $\LL$ defined
  as above, assume that there are $A, B \in \mathscr C(E)$, $\AA, \BB
  \in \mathscr{C}(\EE)$ so that
$$
\LL = \AA + \BB, \,\,\,  L = A + B, \,\,\, A = \AA_{|E}, \,\,\, B = \BB_{|E},
$$
and a real number $a\in \R$ such that
\begin{itemize}

\item[{\bf (i)}]   $(B-a)$ is  hypodissipative on $E$, $(\BB-a)$ is  hypodissipative on $\EE$;
\item[{\bf (ii)}] $A \in \BBB(E)$, $\AA \in \BBB( \EE)$;
\item[{\bf (iii)}] there is $n \ge 1$ and $C_a>0$ such that (semigroup
  commutator condition)
\[
\big\|  (\AA \mathscr S_\BB)^{(*n)}(t)  \big\|_{\BBB(\EE,E)} + \big\|
  (\mathscr S_\BB \AA)^{(*n)}(t)  \big\|_{\BBB(\EE,E)}  \le  C_a \, e^{at}.
\]
\end{itemize}
   
 \smallskip\noindent
 Then the following two properties are equivalent:
 \begin{itemize}
 \item[{\bf (1)}] There are distinct $\xi_1, \dots, \xi_k \in
   \Delta_a$ and finite rank projectors $\Pi_{j,L} \in \BBB(E)$, $1
   \le j \le k$, which commute with $L$ and satisfy
   $\Sigma(L_{|\Pi_{j,L}}) = \{ \xi_j \}$, so that the semigroup
   $\mathscr S_L(t)$ satisfies for any $a'>a$
 \begin{equation}\label{eq:estimSGL} 
   \forall \, t \ge 0, \quad \left\|  \mathscr S_L(t) - \sum_{j=1}^k 
     \mathscr S_L(t) \, \Pi_{j,L} \right\|_{\BBB(E)}
   \le C_{L,a'} \, e^{a' \, t}  
 \end{equation}
 with some constant $C_{L, a'}>0$.

\item[{\bf (2)}] There are distinct $\xi_1, \dots, \xi_k \in \Delta_a$
  and finite rank projectors $\Pi_{j,\LL} \in \BBB(\EE)$, $1 \le j \le
  k$, which commute with $\LL$ and satisfy $\Sigma(\LL_{|\Pi_{j,\LL}})
  = \{ \xi_j \}$, so that the semigroup $\mathscr S_\LL(t)$ satisfies
  for any $a'>a$
 \begin{equation}\label{eq:estimSGLL} 
   \forall \, t \ge 0, \quad \left\|  \mathscr S_\LL(t) - \sum_{j=1}^k 
     \mathscr S_\LL(t) \, \Pi_{j,\LL} \right\|_{\BBB(\EE)}
   \le C_{\LL, a'} \, e^{a' \, t}  
 \end{equation}
 with some constant $C_{\LL, a'} >0$.
\end{itemize}
  \end{theo}
   
  \begin{rems} \label{rem:theoExtension}
   \begin{itemize}
   \item[{\bf (a)}] The constants in this statement can be estimated
     explicitly from the proof. 
   \item[{\bf (b)}] The same result holds in the case $\{ \xi_1,
     \dots, \xi_k \}= \emptyset$, that we denote as a convention
     as the case $k = 0$.
   \item[{\bf (c)}] The condition ``$E \subset \EE$'' can be replaced by
     ``$E\cap\EE$ is dense in $E$ and $\EE$ with continuous embedding''.
   \item[{\bf (d)}] Note that in the LHS of condition {\bf (iii)}, any
     of the two terms (at order $n$) can be deduced from the other (at
     order $n+1$) with the help of assumptions {\bf (i) } and {\bf
       (ii)}.
   \end{itemize}
 \end{rems}
 
\subsection{The main PDE results}
\label{sec:results-fokk-planck}

Let us briefly present the evolution PDEs of Fokker-Planck types on
which we are able to make use of Theorem~\ref{theo:Extension} in order
to establish exponential asymptotic stability of their equilibria.
% We then show that the spacial homogeneous Fokker-Planck equations and
% the kinetic (or spatially inhomogeneous) Fokker-Planck equation in the
% case of the torus and of the whole space with confinement potential
% fall in that class.
\mbox{ }

\sk

\noindent
\emph{(a) ``Flat'' confinement.} The model is the kinetic
Fokker-Planck equation
\begin{equation}\label{eq:FP1}
\partial_t f + v \cdot \nabla_x f= \nabla_v \cdot \left( \nabla_v f + \nabla_v \Phi \,
  f\right), 
\end{equation} on the density $f=f(t,x,v)$, $t
\ge 0$, $x \in \T^d$ the flat $d$-dimensional torus, $v \in \R^d$, for
a friction potential $\Phi=\Phi(v)$ satisfying $ \Phi \approx |v|^\gamma$, $\gamma \ge 1$, for
large velocities. 

\begin{rem}
  Observe that this model contains as a subcase the (spatially
  homogeneous) Fokker-Planck equation
  \begin{equation}\label{eq:FP1bis}
    \partial_t f = \Delta_v f + \hbox{div}_v (\nabla_v \Phi \,
    f), \quad f = f(t,v), \ t \ge 0, \ v \in \R^d,
  \end{equation} 
  when the probability density $f=f(t,v)$ is independent of the space
  variable.
  % , $t \ge 0$, $v \in \R^d$, for a friction potential
  % $\Phi=\Phi(v)$ satisfying the same assumptions as above.
\end{rem}

\sk 

\noindent
\emph{(b) Confinement by a potential.} The model is the kinetic
Fokker-Planck equation in the whole space with a space confinement
potential
\begin{equation}\label{eq:FP2}
  \partial_t f + v \cdot \nabla_x f -
  \nabla_x \Psi \cdot \nabla_{v}f= \nabla_v
  \cdot \left( \nabla_v f + v \, f\right), 
\end{equation}
on the density $f=f(t,x,v)$, $t \ge 0$, $x \in \R^d$, $v \in \R^d$,
for a confinement potential $\Psi=\Psi(x)$ on the space variable which
behaves like $|x|^\beta$, $\beta \ge 1$, for large values of the
vector position.

%\smallskip For these three models, we will prove some semigroup
%spectral gap estimate in a large class of weight Sobolev spaces
%$W^{\sigma,p}(m)$, $\sigma \in \{-1,0,1\}$, $p \in [1,\infty]$ and $m$
%weight function which increases rapidly enough (typically $m$ is a
%polynomial function or a stretch exponential function).  
%More
%precisely, we prove that there exist $a < 0$ and $C_a > 0$ so that
%\begin{equation}\label{eq:intro:decay}
%\| e^{t\LL} f - \Pi_{1} f \|_{W^{\sigma,p}(m)}\le C_{a} \, e^{at} \, \|  f - \Pi_{1} f \|_{W^{\sigma,p}(m)}
%\qquad \forall \, f \in W^{\sigma,p}(m),
%\end{equation}
%where $\Pi_1$ denotes the projection on the first eigenspace $\R \mu$  (associated to the eigenvalue $0$ which is indeed the largest eigenvalue)
%and $\mu$ denotes  the unique (positive and normalized) equilibrium of the equation. 
%A similar estimate in  Monge-Kantorovish-Wasserstein distance $W_1$ is also established in some particular cases. 
%
%More
%precisely, we prove that there exist $a < 0$ and $C_a > 0$ so that
%\begin{equation}\label{eq:intro:decay}
%\| e^{t\LL} f - \Pi_{1} f \|_{W^{\sigma,p}(m)}\le C_{a} \, e^{at} \, \|  f - \Pi_{1} f \|_{W^{\sigma,p}(m)}
%\qquad \forall \, f \in W^{\sigma,p}(m),
%\end{equation}
%where $\Pi_1$ denotes the projection on the first eigenspace $\R \mu$  (associated to the eigenvalue $0$ which is indeed the largest eigenvalue)
%and $\mu$ denotes  the unique (positive and normalized) equilibrium of the equation. 
%A similar estimate in  Monge-Kantorovish-Wasserstein distance $W_1$
%is also established in some particular cases. 
\mk

For these models, we prove semigroup exponential decay estimates in
weighted Sobolev spaces with weight function increasing like
polynomial function or a stretch exponential function, so much slower
than the usual inverse of the Gaussian equilibrium.

\begin{theo}[] \label{theo:Gal} Consider $\LL$ the Fokker-Planck
  operator as defined above in (a) or (b), and $\mu$ the unique
  positive associated equilibrium with mass $1$.  Consider the
  weighted Sobolev space $\EE := W^{\sigma,p}(m)$ with
  $\sigma \in \{-1,0,1\}$ and $p \in [1,\infty]$, where the precise
  conditions on the weight $m$ (so that it is confining enough) are
  given in Theorems~\ref{theo:vFP} and
  \ref{theo:KFPrdhyp1}. % In the last case (c) above for
  % $\LL$ we assume that $\sigma=0$ (weighted Lebesgue space only).

  Then there exist $a < 0$ and $C_a > 0$ so that
  \begin{equation}\label{eq:intro:decay}
    \forall \, f_0, \, g_0 \in \EE, \quad 
    \left\| \mathscr S_\LL(t) f_0 - \mathscr S_\LL(t) g_0 \right\|_{\EE}\le 
    C_{a} \, e^{at} \, \left\|  f_0 - g_0 \right\|_{\EE}
  \end{equation}
  between two solutions with same mass; this implies the exponential
  convergence towards the projection on the first eigenspace $\R \mu$
  (associated with the eigenvalue $0$).

  In the case (a) (periodic confinement), we also establish a similar
  decay estimate in Monge-Kantorovich-Wasserstein distance: for all
  $f_0$, $g_0$ probability measures with first moment bounded
  \[
    %\forall \, f_0, \, g_0, \quad 
    W_1\left(\mathscr S_{\LL}(t) f_{0} , \mathscr S_{\LL}(t) g_{0}
  \right) \le C_a \, e^{a \, t} \, W_1( f_{0} , g_0).
  \]
  % between two solutions with same mass.
\end{theo}

\smallskip This theorem is proved by combining:
\begin{itemize}
\item the spectral gap property of the Fokker-Planck semigroup which
  is classically known in the space of self-adjointness
  $L^2(\mu^{-1/2})$ in a spatially homogeneous setting (Poincar\'e
  inequality) and has been recently proved in a series of works about
  {\it ``hypocoercivity''} in the spaces $L^2(\mu^{-1/2})$ or
  $H^1(\mu^{-1/2})$ for the kinetic Fokker-Planck semigroup with
  periodic or potential confinements \cite{MR2034753,MR2562709,DMS};
\item an appropriate decomposition of the operator with:
  \begin{itemize}
  \item some dissipativity estimates in the ``target'' functional
    spaces, for the ``dissipative part'' of the decomposition (this is
    the main difficulty in the case of confinement by a potential and
    we introduce specifically weight multipliers inspired from
    commutator conditions on derivatives);
  \item some regularisation estimates adapted on the semigroup
    inspired from ultracontractivity estimates in the spirit of Nash's
    regularity estimate \cite{MR0100158} in the spatially homogeneous
    case and H\'erau-Villani's quantitative global hypoellipticity
    estimate \cite{Herau2007}-\cite[section A.21.3]{MR2562709} in the
    spatially inhomogeneous case;
  \item the application of Theorem~\ref{theo:Extension} (whose
    assumptions are established by the previous items) which
    establishes the decay estimates of the semigroup in the target
    functional space;
  \item finally the $W_1$ estimate is obtained by some additional
    technical efforts in estimating the decay in weighted $W^{-1,1}$
    type spaces.
  \end{itemize}
\end{itemize}

\begin{rem}
  Some decay estimates for kinetic Fokker-Planck semigroups with flat
  confinement had been already established in \cite{GMM}. In this
  setting this new paper improves on these previous paper as follows:
  we use new integral identity in order to deal with any integrability
  exponent $p \in [1,\infty]$ and we introduce an appropriate duality
  argument in order to deal with the regularity exponent $\sigma =
  -1$.
\end{rem}

\begin{rem} 
  Let us mention that there is an important literature in the
  probability community, see for instance \cite{MR1807683,MR2381160},
  that deals with exponential relaxation to equilibrium for stochastic
  processes whose law follows kinetic Fokker-Planck equations. From an
  analysis viewpoint, these results typically correspond to the
  exponential decay of solutions to the PDE in weighted total
  variation norms, assuming higher moments or integrability on the
  initial data, and without quantitative estimate on the rate. It is
  worth pointing out that our dissipativity estimates on the
  dissipative part of the decomposition of the operator, discussed
  above, are reminiscent of the so-called ``Lyapunov condition'' at
  the core of these probability works (on the whole operator). Their
  uses and the results obtained are different though as we aim at
  semigroup decay estimates rather than functional inequalities, and
  take advantage of an already existing decay estimate in a
  (typically) smaller functional space.
\end{rem}

\subsection{Plan of the paper}
\label{sec:plan-paper}

The outline of the paper is as follows. We prove the main abstract
theorem in Section~\ref{sec:factorization}. We prove the decay
estimates on kinetic Fokker-Planck semigroups with periodic (or
spatially homogeneous) confinements in
Section~\ref{sec:FPhomo}. Finally we prove the decay estimates on
kinetic Fokker-Planck semigroups with confinement by a potential in
Section~\ref{sec:KFPrd}.
\medskip

\noindent
{\bf Acknowledgements.}   
We thank Jos\'e Alfr\'edo Ca\~nizo and Maria P. Gualdani for fruitful
comments and discussions. We thank the anonymous referees for many
useful comments. The first author's work is supported by the french
``ANR blanche'' project Stab: ANR-12-BS01-0019.  The second author's
work is supported by the ERC starting grant MATKIT.

 %%%%%%%%%%%%%%%%%%%%%%%%%%%%%%%%%%%%%%%%%%%%
 \bigskip\bigskip 
 \section{Factorisation of semigroups in Banach spaces and
   applications}
\label{sec:factorization}
\setcounter{equation}{0}
\setcounter{theo}{0}
 
The section is devoted to the proof of Theorem~\ref{theo:Extension}.
After having recalled some notation, we present the proof of
Theorem~\ref{theo:Extension} that we split into two steps, namely the
analysis of the spectral problem and the semigroup decay.

\subsection{Notations and definitions}
\label{sec:notation-definitions}

We denote by $\GGG(E) \subset \CCC(E)$ the space of semigroup
generators and for $\Lambda \in \GGG(E)$ we denote by $\mathscr
S_\Lambda(t) = e^{\Lambda t}$, $t \ge 0$, its semigroup, by
$\mathscr{D}(\Lambda)$ its domain, by $\mathscr{N}(\Lambda)$ its null
space, by
$$
\mathscr M(\Lambda) = \cup_{\alpha \ge 1} \mathscr N(\Lambda^\alpha)
$$
its algebraic null space, and by $\mathscr{R}(\Lambda)$ its range. 
%The spaces $\textrm{D}(\Lambda)$ and $\textrm{D}(\Lambda^2)$ are Banach spaces when endowed
%with the norms
%$$
%\|f \|_{\textrm{D}(\Lambda)} = \|f \|_E + \|\Lambda f \|_E, \quad
%\|f \|_{\textrm{D}(\Lambda^2)} = \|f \|_E + \|\Lambda f \|_E +  \|\Lambda^2 f \|_E.
%$$
We also denote by $\Sigma(\Lambda)$ its spectrum, so that for any $\xi
\in \C \backslash \Sigma(\Lambda)$ the operator $\Lambda - \xi$ is
invertible and the resolvent operator
$$
R_\Lambda(\xi) := (\Lambda -\xi)^{-1}
$$
is well-defined, belongs to $\mathscr{B}(E)$ and has range equal to
$\mathscr D(\Lambda)$.

We recall that $\xi \in \Sigma(\Lambda)$ is said to be an \emph{eigenvalue}
if $\mathscr N(\Lambda - \xi) \neq \{ 0 \}$. Moreover an eigenvalue $\xi \in
\Sigma(\Lambda)$ is said to be \emph{isolated} if
\[
\Sigma(\Lambda) \cap \left\{ z \in \C, \,\, |z - \xi| < r \right\} =
\{ \xi \} \ \mbox{ for some } r >0.
\]
In the case when $\xi$ is an isolated eigenvalue we may define
$\Pi_{\Lambda,\xi} \in \mathscr{B}(E)$ the spectral projector by 
\begin{equation}\label{def:SpectralProjection} 
\Pi_{\Lambda,\xi} := {i \over
  2\pi} \int_{ |z - \xi| = r' } (\Lambda-z)^{-1} \, dz
\end{equation}
with $0<r'<r$. Note that this definition is independent of the value
of $r'$ by Cauchy's theorem as the application
\[
\C \setminus \Sigma(\Lambda) \to \mathscr{B}(E), \quad z \mapsto
R_{\Lambda}(z)\] is holomorphic in $B(z,r)$. It is well-known
\cite[III-(6.19)]{Kato} that $\Pi_{\Lambda,\xi}^2=\Pi_{\Lambda,\xi}$
is a projector, and its range $\mathscr{R}(\Pi_{\Lambda,\xi})$ is the
closure of the algebraic eigenspace associated to $\xi$.  Moreover the
range of the spectral projector is finite-dimensional if and only if
there exists $\alpha_0 \in \N^*$ such that
\[
\hbox{dim} \, \mathscr N(\Lambda -\xi)^{\alpha_0} < \infty, \quad
\mathscr N(\Lambda -\xi)^\alpha = \mathscr N(\Lambda -\xi)^{\alpha_0} \ \mbox{ for
  any } \ \alpha \ge \alpha_0,
\]
so that 
\[
\overline{\mathscr M(\Lambda-\xi)} = \mathscr M(\Lambda - \xi) =
\mathscr N((\Lambda -\xi)^{\alpha_0}).
\]
In that case, we say that $\xi$ is a \emph{discrete eigenvalue}, written as
$\xi \in \Sigma_d(\Lambda)$. Observe that $R_\Lambda$ is meromorphic
on $(\C \setminus \Sigma(\Lambda)) \cup \Sigma_d(\Lambda)$ (with
non-removable finite-order poles). 
%\textcolor{red}{cm~: je ne comprend cette derniere caracteristation du
%  fait que la valeur propre soit discrete: l'identite verifie
%  $N(\Id-1)=E$ et pourtant $N((\Id-1)^k)=N(\Id-1)$ pour tout $k \ge
%  1$, et le projecteur spectral n'est pas de dimension finie.}
Finally for any $a \in \R$ such that $\Sigma(\Lambda) \cap \Delta_{a }
= \left\{ \xi_1, \dots, \xi_k\right\}$ where $\xi_1, \dots, \xi_k$ are
distinct discrete eigenvalues, we define without any risk of ambiguity
\[
\Pi_{\Lambda,a} := \Pi_{\Lambda,\xi_1} + \dots + \Pi_{\Lambda,\xi_k}.
\]

We need the following definition on the convolution of semigroup
(corresponding to composition at the level of the resolvent
operators).

\begin{defin}[Convolution of semigroups] 
  Consider some Banach spaces $X_1$, $X_2$, $X_3$. For two given
  functions
  \[
  \mathscr S_1 \in L^1(\R_+; \BBB(X_1,X_2)) \ \mbox{ and } \ \mathscr S_2 \in L^1(\R_+;
  \BBB(X_2,X_3)),
  \]
  we define the convolution $\mathscr S_2 \ast \mathscr S_1 \in
  L^1(\R_+; \BBB(X_1,X_3))$ by
  $$
  \forall \, t \ge 0, \quad (\mathscr S_2 * \mathscr S_1)(t) :=
  \int_0^t \mathscr S_2(s) \, \mathscr S_1 (t-s) \dd s.
  $$
  When $\mathscr S = \mathscr S_1= \mathscr S_2$ and $X_1=X_2=X_3$, we
  define inductively $\mathscr S^{(*1)} = \mathscr S$ and $\mathscr
  S^{(*\ell)} = \mathscr S \ast \mathscr S^{(*(\ell-1))} $ for any
  $\ell \ge 2$.
\end{defin}

\subsection{Factorization and spectral analysis when changing space}
\label{sec:spectr-analys-fact}

% We prove in the subsection Theorem~\ref{theo:Extension}, which asserts
% that the localization of the principal part of the spectrum does not
% change when the functional space is changed by enlargement and
% shrinking. 

\begin{theo}\label{theo:factor:equiv} 
  Consider $E,\EE, L, A, B, \LL, \AA, \BB$ as above and assume
  that
\begin{itemize}
\item[{\bf (i$'$)}]   $\Sigma(B) \cap \Delta_a = \Sigma(\BB) \cap \Delta_a = \emptyset$    for some  $a \in \R$;
  \item[{\bf (ii)}]  $A \in \BBB(E)$ and   $\AA \in \BBB(\EE)$;
  \item[{\bf (iii$'$)}] there is $n \ge 1$ such that for any $\xi \in
    \Delta_a$, the operators $(\AA \, R_\BB(\xi))^n$ and $(R_\BB(\xi)
    \, \AA )^n$ are bounded from $\EE$ to $E$.
   \end{itemize}

\medskip
Then the following two properties are equivalent, with the same family of distinct complex numbers and the convention $\{\xi_1, \dots,
    \xi_k \}= \emptyset$ if $k=0$:
\smallskip
  \begin{itemize}
  \item[{\bf (1)}] $\Sigma(L) \cap \Delta_a = \left\{ \xi_1, \dots,
      \xi_k \right\} \subset \Sigma_d(L)$ (distinct discrete
    eigenvalues). 
\smallskip

\item[{\bf (2)}] $\Sigma(\LL) \cap \Delta_a = \left\{ \xi_1, \dots,
    \xi_k \right\} \subset \Sigma_d(\LL)$ (distinct discrete
  eigenvalues).
       \end{itemize}
\smallskip

Moreover, in both cases, there hold
    \begin{itemize}
  \item[{\bf (3)}]   For any $z \in \Delta_a \setminus \{ \xi_1, \dots, \xi_k
  \}$ the resolvent operators $R_L$ and $R_\LL$  satisfy: 
 \begin{align}
    \label{eq:factor-enlarg}
    R_{\LL}(z) & = \sum_{\ell=0} ^{n-1} (-1)^\ell R_{\BB}(z) \left( \AA
      R_{\BB}(z) \right)^\ell + (-1)^n R_{L}(z) \left( \AA
      R_{\BB}(z) \right)^n \\
    \label{eq:factor-reduc}
    R_{L}(z)  &= \sum_{\ell=0} ^{n-1} (-1)^\ell \left( R_{B}(z) A
    \right)^\ell R_{B}(z) + (-1)^n \left( R_{\BB}(z) \AA
    \right)^n R_{\LL}(z).
  \end{align}
 
  \smallskip

\item[{\bf (4)}]  For any $j=1, \dots,
  k$, we have 
\begin{equation*}
\left\{
\begin{array}{lcl}
  \mathscr N(L -\xi_j)^\alpha  & = & \mathscr N(\LL -\xi_j)^\alpha, 
  \quad \forall \, \alpha
  \ge 1
  \\[2mm]
  \mathscr M(L -\xi_j) & = & \mathscr M(\LL -\xi_j) \\[2mm]
  (\Pi_{\LL,\xi_j})_{|E}  & = &\Pi_{L,\xi_j}  \\[2mm]
  \mathscr S_{\LL,\xi_j}(t)  & = &\mathscr S_{\LL}(t) \Pi_{\LL,\xi_j} =
  \mathscr S_{L}(t) \Pi_{\LL,\xi_j}.
\end{array}
\right.
\end{equation*}
\end{itemize}  
 \end{theo}

\begin{rems}\label{rem:factorization} 
  \begin{enumerate}
  \item In this theorem, the implication \textbf{(1)} $\Rightarrow$
    \textbf{(2)} has been established in \cite[Theorem 2.1]{GMM};
    since $E \subset \EE$, this is a recipe for \emph{enlarging} the
    functional space where a property of localization of the discrete
    spectrum holds.  The implication \textbf{(2)} $\Rightarrow$
    \textbf{(1)} is a recipe for \emph{shrinking} the functional space
    where a property of localization of the discrete spectrum holds.
   
  \item In the simplest case where $\AA \in \mathscr{B}(\EE,E)$, the
    assumption {\bf (iii$'$)} is satisfied with $n=1$.
    
  \item The hypothesis (i)-(ii)-(iii) (for some $a \in \R$) in
    Theorem~\ref{theo:Extension} imply the hypothesis
    (i$'$)-(ii)-(iii$'$) above, for any $a' > a$.

  \item A similar result holds when we replace the assumption $E
    \subset \EE$ by the assumption that $E \cap \EE$ is dense in both
    $E$ and $\EE$.
  \end{enumerate}
\end{rems}

\begin{proof}[Proof of Theorem~\ref{theo:factor:equiv}.]
  Because of Remark~\ref{rem:factorization}-(1), we only have to prove
  the implication \textbf{(2)} $\Rightarrow$ \textbf{(1)}. Let us
  denote $\Omega := \Delta_a \setminus \{ \xi_1, \dots, \xi_k \}$ and
  define for $z \in \Omega$
\[
U(z) := \sum_{\ell=0} ^{n-1} (-1)^\ell \left( R_{B}(z) A
    \right)^\ell R_{B}(z) + (-1)^n \left( R_{\BB}(z) \AA
    \right)^n R_{\LL}(z).
\]
Observe that thanks to the assumptions {\bf (i$'$)-(ii)-(iii$'$)} and
{\bf (2)}, the operator $U(z)$ is well-defined and bounded on $E$.

\smallskip
\noindent {\it Step 1. $U(z)$ is a left-inverse of $(L-z)$ on
  $\Omega$.}  
For any $z \in \Omega$, we compute
 \begin{eqnarray*} 
   U(z) (L-z) 
   &=& \sum_{\ell = 0}^{n-1} (-1)^\ell \, \left(R_B(z)
     A\right)^\ell \, R_B(z)  \,  (A+(B-z )) \\
   &&
   \qquad + (-1)^n \,  \left(R_\BB(z) \AA\right)^n \, R_\LL(z) \, (L-z)  \\
   &=&  \sum_{\ell = 0}^{n-1} (-1)^\ell \, \left( R_B(z) A
   \right)^{\ell+1} 
   +  \sum_{\ell = 0}^{n-1} (-1)^\ell \, \left( R_B(z) A \right)^\ell \\
   &&
   \qquad + (-1)^n \,  \left(R_B(z) A \right)^n  =  \mbox{Id}_{E}. 
\end{eqnarray*} 

 \smallskip
\noindent
{\it Step 2. $(L-z)$ is invertible on $\Omega$.}  
Consider $z_0 \in \Omega$. First observe that if the
operator $(L - z_0)$ is bijective, then composing to the right the
equation 
\[
U(z_0) (L-z_0)= \mbox{Id}_{E}
\]
by $(L-z_0)^{-1}=R_L(z_0)$ yields $R_L(z_0) = U(z_0)$ and we deduce that the inverse map is bounded (i.e. $(L-z_0)$ is an invertible operator in $E$) together with the desired formula for the
resolvent. 

Since $(L-z_0)$ has a left-inverse it is injective.  Let us prove that
it is surjective. Consider $g \in E$. Since $\LL-z_0$ is invertible
and therefore bijective there is $f \in \EE$ so that
\[
(\LL -z_0)f = g \ \mbox{ and thus } \ \mbox{Id}+ R_\BB(z_0) \AA f =
  R_\BB(z_0) g = R_B(z_0) g. 
\]
We denote $\bar g := R_B(z_0)g \in E$ and $\GG(z_0) := R_\BB(z_0) \AA$
and write
\[
f = \bar g - \GG(z_0)f = \sum_{\ell=0} ^{n-1} (-1)^\ell
\GG(z_0)^\ell \bar g + (-1)^n \GG(z_0)^n f. 
\]
Because of {\bf (i$'$)-(ii)-(iii$'$)},  it implies that $f \in E$, and in fact since
$\mathscr D(B)=\mathscr D(L)$, we further have $f  \in
\mathscr D(L) \subset E$. We conclude that $(L-z_0)f = g$ in $E$,
and the proof of this step is complete.

\smallskip
\noindent
{\it Step 3. Spectrum, eigenspaces and spectral projectors.} 
On the one hand, we have
\[
\mathscr N (L-\xi_j)^\alpha \subset \mathscr N(\LL-\xi_j)^\alpha,
\quad j = 1, \dots, k, \quad \alpha \in \N,
\]
so that $\Sigma(L) \cap \Delta_a \subset \{ \xi_1, \dots, \xi_k \}$.
On the other hand, consider $\xi_j \in \Sigma(\LL) \cap \Delta_a$,
$\alpha \in \N^*$ and $f \in \mathscr N(\LL-\xi_j)^\alpha$:
\[
\left( \LL - \xi_j \right)^\alpha f = 0.
\]
Denote $g_\beta := \left( \LL - \xi_j \right)^\beta f$, $\beta = 0,
\dots, \alpha$ and argue by induction on $\beta$ decreasingly to prove
that $g_\beta \in E$. The initialisation $\beta = \alpha$ is
clear. Assume $g_{\beta+1} \in E$ and write $\left( \LL - \xi_j
\right) g_\beta = g_{\beta+1}.$ Using $\LL = \AA + \BB$ and composing
to the left by $R_\BB(\xi_j)$, we get
\[
\left( \GG(\xi_j) +\mbox{Id} \right) g_\beta = R_B(\xi_j) g_{\beta+1}
\in E \ \mbox{ with } \
\GG(\xi_j) := R_\BB(\xi_j) \AA. 
\]
We deduce that 
\[
g_\beta = (-1)^n \GG(\xi_j)^n g_{\beta+1} + \sum_{k=0} ^{n-1}
\GG(\xi_j)^k R_B(\xi_j) g_{\beta+1}. 
\]
Since $\GG(\xi_j)^n$ is bounded from $\EE$ to $E$, and $\GG(\xi_j)$ is
bounded from $E$ to $E$, with in each the range included in $\mathscr
D(B) = \mathscr D(L)$, we deduce that $g_\beta \in \mathscr D(L)
\subset E$, and the proof of the induction is complete. Finally $g_0 =
f \in \mathscr D(L) \subset E$. Since the eigenvalues are discrete,
this completes the proof of {\bf (1)}.
% Hence we conclude that $\Sigma(L) \cap \Delta = \{ \xi_1, \dots,
% \xi_k \}$ with 
% \[
% N (L-\xi_j)^\alpha = N(\LL-\xi_j)^\alpha, 
% \ \alpha \in \N \ \mbox{ and thus } \ M( \LL - \xi_j) =
% M(L-\xi_j), \quad j = 1, \dots, k.
% \]

Finally, the fact that $\Pi_{\LL,\xi_j| E} = \Pi_{L,\xi_j}$ is a
straightforward consequence of $R_\LL(z) f = R_L(z)f$ when $f \in E$
and the formula \eqref{def:SpectralProjection} for the projection
operator.  This concludes the proof of {\bf (3)-(4)}.
\end{proof}

%{\Red cm: legere modification dans le step 3 ci-dessus il y avait une
%  petite erreur quand on factorisait la puissance $\alpha$-eme de $\LL
%  - \xi_j$ due au fait que $R_B$ et $A$ ne commutent pas en
%  general. Il doit y avoir la meme erreur dans le papier soumis a jams
%  qu'il faudra corriger.}
% We state the the following result variant of \cite[Lemma 2.15]{GMM}
% and which makes possible to establish assumption (iii) in
% Theorem~\ref{theo:Extension} in the present hypoelliptic context.

%%%%%%%%%%%%%%%%%%%%%%%%%%%%%%%%%%%%%%%%%%%%%%%%%%%%%%%%%%%%%%%%

\subsection{Factorization and semigroup decay when changing spaces}
\label{sec:fact-decay}

We now prove Theorem~\ref{theo:Extension}. First we notice that the
assumptions of Theorem~\ref{theo:factor:equiv} are met since {\bf
  (i$'$)} follows from {\bf (i)} and {\bf (ii)}-{\bf (iii)} imply {\bf
  (iii$'$)}. Because of Theorem~\ref{theo:factor:equiv} we know that
$\mathscr R (\Pi_{\LL,a}) = \mathscr R (\Pi_{L,a}) \subset E$, and then
for any $f_0 \in \mathscr R (\Pi_{L,a})$, there holds
$$
\mathscr S_\LL(t) \, f_0 = \mathscr S_L(t) \, f_0 = \sum_{j=1}^k
e^{L_j t}\, \Pi_{L,\xi_j} f_0,
$$
where $L_j := L_{|X_j}$, $X_j := \mathscr
R(\Pi_{L,\xi_j})$. By linearity, it is enough to prove the equivalent
estimates \eqref{eq:estimSGL} and \eqref{eq:estimSGLL} in the
supplementary space of the subspace $\mathscr R(\Pi_{L,a})$. We split the proof
into two steps.

\medskip\noindent {\sl Step 1. Enlargement of the functional space. }
We give here an alternative presentation of the proof of {\bf (1)}
$\Rightarrow$ {\bf (2)} in Theorem~\ref{theo:Extension} which is in
the spirit of \cite{MR946973} while the original (but similar) proof
in \cite{GMM} uses an iterate Duhamel formula.  We assume
\eqref{eq:estimSGL} and denote $f_t := S_\LL(t) f_0$ the solution to
the evolution equation $\partial_t f = \LL f$.  We decompose
\begin{align*}
  f & = \Pi_{L,a} f_t + g^1 + g^2 + \dots + g^{n+1},\\
  \partial_t g^1 &= \BB g^1, \quad g^1_0 = f_0-\Pi_{L,a} f_0,
  \\
  \partial_t g^k &= \BB g^k + \AA g^{k-1}, \quad g^k_0
  = 0, \quad 2 \le k \le n,
  \\
  \partial_t g^{n+1} &= \LL g^{n+1} + \AA g^{n}, \quad
  g^{n+1} _0 = 0,
  % \\
  % \partial_t g^{n+2} &= \LL g^{n+2} + \Pi_{\LL,a} ( \AA g^1+ \dots +
  % \AA g^{n}) , \quad g^{n+2}_0 =0,
\end{align*}
and we remark that this system of equations on $g_k$, $1 \le k \le
n+1$, is compatible with the equation satisfied by $f$. Moreover, by
induction
\begin{align*} 
  \AA  g_k (t) %= (\AA \SS_\BB* \AA g_{n-1})(t) % + (\AA \SS_B
  % * \Pi_{\LL,a} \AA g_{n-1})(t)  \\
  = (\AA \SS_\BB)^{(*k)} (t)(f_0-\Pi_{L,a} f_0), \quad 1 \le k \le n,% +
  % \sum_{k=1}^{n-1}(\AA  \SS_B * \AA g_{k})(t),
\end{align*}
so that $\AA g_n (t) \in E$, because of assumption {\bf (iii)}, and
thus the equation on $g_{n+1}$ is set in $E$ and writes
\begin{equation*}
  \partial_t g_{n+1} =  L g_n  + \AA g_{n}, \quad g_{n+1} (0) = 0.
\end{equation*}
We deduce successively the estimates (for $a'>a$)
\begin{align*}
  \| g_k(t) \|_{\EE} & \lesssim t^k e^{at} \| f_0 - \Pi_{L,a} f_0
  \|_\EE, \quad 1 \le k \le n, \\
  \| g_k(t) \|_{\EE} & \lesssim_{a'} e^{a't} \| f_0 - \Pi_{L,a} f_0
  \|_\EE, \quad 1 \le k \le n, \\
  \|\AA g_n(t) \|_{E} & \lesssim t^n e^{at} \| f_0 - \Pi_{L,a} f_0
  \|_\EE,\\
  % \|g_k(t) \|_{\EE} &\le \int_0^t \|S_\BB (t-s) \AA
  % g_{k-1}(s) \|_\EE \dd s \\
  % &\le \int_0^t e^{as} \, C_{\Pi_{\LL,a} ,\AA} \, \| g_{k-1}(s) \|_\EE
  % \dd s \\
  % &\le C_k \, t^k \, e^{at} \, \| f_0 - \Pi_{\LL,a} f_0 \|_\EE.
  \| (\Id - \Pi_{L,a}) g_{n+1}(t) \|_{\EE} & \lesssim \| (\Id - \Pi_{L,a})
  g_{n+1}(t) \|_{E} \la_{a'} e^{a't} \| f_0 - \Pi_{L,a} f_0 \|_\EE,
\end{align*}
and since, from the definition of the decomposition,
\begin{align*}
  \Pi_{L,a} g_{n+1} = - \Pi_{L,a} g_1 - \dots - \Pi_{L,a} g_n
\end{align*}
we have, using the previous decay estimates,
\begin{align*}
  \left\| \Pi_{L,a} g_{n+1}(t) \right\|_{\EE} \le \sum_{k=1} ^n \left\| \Pi_{L,a}
  g_k(t) \right\|_{\EE} \lesssim_{a'} e^{a't} \left\| f_0 - \Pi_{\LL,a} f_0 \right\|_\EE,
\end{align*}
which concludes the proof of \eqref{eq:estimSGLL} by piling up these
estimates on $f$.

\medskip\noindent {\sl Step 2. Shrinkage of the functional space. } We
assume \eqref{eq:estimSGLL} and $f_0 \in E$ and write the
following family of operators depending on time on $E$ through a
factorization formula: 
\begin{multline*}
  \mathscr S_*(t) = S_L(t) \Pi_{L,a} + \sum_{\ell=0} ^{n-1} (-1)^\ell \left( \mathscr S_{B}(t) A
    \right)^{(* \ell)} \mathscr S_{B}(t)(\Id - \Pi_{L,a}) \\ + (-1)^n \left( \mathscr S_{\BB}(t) \AA
    \right)^{(* n)} \mathscr S_{\LL}(t) (\Id - \Pi_{L,a}).
\end{multline*}
Using the assumptions and \eqref{eq:estimSGLL} one gets 
\begin{align*}
 \left\|  \mathscr S_{\LL}(t) (\Id - \Pi_{L,a}) \right\|_{\BBB(E,E)} \lesssim_{a'}
 e^{a't} \\
  \left\|  \left( \mathscr S_{\BB}(t) \AA \right)^{(* n)}
  \right\|_{\BBB(\EE,E)} \lesssim_{a'}
 e^{a't} \\
\left\|  \mathscr S_{B}(t) A \right\|_{\BBB(E,E)} \lesssim_{a'}
 e^{a't} 
\end{align*}
which proves that 
\begin{equation*}
  \left\|  \mathscr S_*(z) -  S_L(t) \Pi_{L,a} \right\|_{\BBB(E,E)} \lesssim_{a'}
 e^{a't}.
\end{equation*}
Therefore the Laplace transform $U_*(z)$ of $t \mapsto (\mathscr S_*(t) - S_L(t)
\Pi_{L,a})$ is well-defined on $\Re e \, z > a'$, and is 
\begin{equation*}
  U_*(z) = \sum_{\ell=0} ^{n-1} (-1)^\ell \left( R_{B}(z) A
    \right)^\ell R_{B}(z) (\Id - \Pi_{L,a}) + (-1)^n \left( R_{\BB}(z) \AA
    \right)^n R_{\LL}(z) (\Id - \Pi_{L,a})
\end{equation*}
which is exactly $U_*(z) = R_L(z) (\Id - \Pi_{L,a})$ from
Theorem~\ref{theo:factor:equiv}. By uniqueness of the Laplace
transform we deduce that $\mathscr S_*(t) = \mathscr S_L(t)$, and this
proves the decay \eqref{eq:estimSGL}. \qed

%\textcolor{red}{cm: j'ai reecrit la preuve ici, il y avait une erreur
%  il me semblait, la decomposition en semigroupe ne marchait pas dans
%  ce cas, ou en tout cas je ne trouve pas comment ecrire sous forme de
%  decomposition de semigroupe la formule de factorisation, je suis
%  revenu a la formule de factorisation et je mentionne brievement
%  l'argument par identification de la transformee de Laplace.}
% $$
% F_t = e^{\BB t } F_0 + \int_0^t e^{\BB(t-s)}\, \AA F_s \, ds,
% $$
% we have 
% $$
% f_t = \sum_{k=0}^{n} \SS_B * ( A \SS_B )^{(*k)}(t) f_0 + (\SS_B * (
% \AA \SS_\BB )^{(*n)} \AA * \SS_\LL) (t) f_0.
% $$
% Proceeding similarly as in step 1, using the decay estimate on
% $\SS_B$ in order to deal with the first term and estimate
% \eqref{eq:estimSGLL} in order to deal with the second term, we find
% \begin{eqnarray*} \|F_t \|_E &\le& \sum_{k=0}^n C_{A,B,k} \, t^k \, e^{a \, t} \,
% \| F_0 \|_E + C_B \, C_{\AA,\BB} \, C_\AA \, C_\LL \, e^{at} \, \|F_0
% \|_\EE
% \\
% &\le& C \, (1+t^n) \, e^{at} \, \|F_0 \|_E .  \end{eqnarray*} This concludes the
% proof of \eqref{eq:estimSGL} thanks to the continuous embedding $E
% \subset \EE$. 

\subsection{A practical criterion}
\label{sec:practical-criterion}

We finally prove a criterion implying both {\bf (iii$'$)} in
Theorem~\ref{theo:factor:equiv} and {\bf (iii)} in
Theorem~\ref{theo:Extension}.

\begin{lem} \label{lem:Tn} Consider two Banach spaces $E$ and 
	$\EE$  such that $E \cap \EE$ is dense into $E$ and $\EE$ with 
	continuous embedding.  Consider $\LL$ an operator on $E+\EE$ so that there
  exist some operators $\AA$ and $\BB$ on $E+\EE$ such that $\LL$
  splits as $\LL = \AA + \BB$.  Denoting with the same letter $\AA$,
  $\BB$ and $\LL$ the restriction of these operators on $E$ and $\EE$,
  we assume that there hold:

\begin{itemize}
\item[{\bf (a)}]   $(\BB-a)$ is  hypodissipative in $E$ and $\EE$  for some  $a \in \R$;
  \item[{\bf (b)}]  $\AA \in \BBB( E)$  and $\AA \in \BBB(\EE)$;
  \item[{\bf (c)}] for some $b \in \R$ and $\Theta \ge 0$ there holds
    $\| \AA \mathscr S_{\BB}(t) \|_{\BBB(\EE,E) } \le C e^{b t}
    \,t^{-\Theta}$ and $\| \mathscr S_\BB(t) \AA \|_{\BBB(\EE,E)} \le
    C e^{bt} t^{-\Theta}$.
   \end{itemize}

\smallskip\noindent

Then for any $a' > a$, there is some constructive $n \in
\N$, $C_{a'} \ge 1$ such that
$$
\forall \, t \ge 0, \quad \| (\AA \mathscr S_\BB)^{(*n)}(t) \|_{\BBB(\EE,E)} +
\| (\mathscr S_\BB \AA)^{(*n)}(t) \|_{\BBB(\EE,E)} \le C_{a'} e^{ a'
  t}.
$$
As a consequence, $(\AA \, R_\BB(z))^n$ and $(R_\BB \AA)^n$ are
bounded from $\EE$ to $E$ for any $z \in \Delta_a$.
\end{lem}

\begin{rem}
  It is necessary to include the non-integrable time factor in (c)
  for later application since (c) will be proved by hypoelliptic
  regularity which has this possibly non-integrable behavior at time
  zero.
\end{rem}

\begin{proof}[Proof of Lemma~\ref{lem:Tn}.] When $\Theta \ge 1$, we
  denote by $J$ the integer such that $\Theta < J \le \Theta+1$ and we
  set $\theta := \Theta/J \in [0,1)$.  We define the family of
  intermediate complex interpolation spaces $\EE_j = [E,\EE]_{j/J}$.
  Thanks to the Riesz-Thorin interpolation theorem, we have
$$
\AA \mathscr S_\BB(t)\ : \ \EE_{\delta,\delta'} := [[\EE_0,\EE_1]_\delta,
\EE_0]_{\delta'} \to \EE^{\delta,\delta'} := [[\EE_0,\EE_1]_\delta,
\EE_1]_{\delta'}
$$
with the following estimate on the operator norm
$$
\|\AA \mathscr S_\BB(t) \|_{ \EE_{\delta,\delta'} \to \EE^{\delta,\delta'} } \le
\|\AA \mathscr S_\BB(t) \|^{(1-\delta)(1-\delta')}_{ \EE_0 \to \EE_0 } \,
\|\AA \mathscr S_\BB(t) \|^{\delta (1-\delta')}_{ \EE_1 \to \EE_1} \,
\|\AA \mathscr S_\BB(t)
\|^{\delta'}_{ \EE_0 \to \EE_1 } .
$$
Since 
\begin{align*}
 \EE_{\delta,\delta'} & = [[\EE_0,\EE_1]_\delta, [\EE_0,\EE_1]_0]_{\delta'} 
= [\EE_0,\EE_1]_{(1-\delta')\delta}\\
\EE^{\delta,\delta'} & =  [[\EE_0,\EE_1]_\delta, [\EE_0,\EE_1]_1]_{\delta'}  
= [\EE_0,\EE_1]_{(1-\delta')\delta+\delta'},
\end{align*}
by taking $\delta' = 1/J$ and $\delta = j/(J-1)$, we get 
\begin{align*}
\|\AA \mathscr S_\BB(t) \|_{ \EE_j \to  \EE_{j+1} }
&\le \|\AA \mathscr S_\BB(t) \|^{1-(j+1)/J}_{ \BBB(E) } \, \|\mathscr S_\BB(t) \|^{j/J}_{ \BBB(\EE) } 
\,  \|\AA \mathscr S_\BB(t) \|_{\BBB(\EE,E)} ^{1/J } 
\\
&\lesssim {1 \over t^\theta} \, e^{[(1-1/J) a + b/J] t}.
\end{align*}
We define now $n := \ell \, J$ so that $(\AA \mathscr S_\BB)^{(*n)} =
\mathscr T_J^{(*\ell)}$ with $\mathscr T_J := (\AA \mathscr
S_\BB)^{(*J)}$. From the assumptions and the previous estimate, for
any $a'' > a$
$$
\|\mathscr T_J(t) \|_{E\to E} \lesssim_{a''} e^{a't}, \quad
\|\mathscr T_J(t) \|_{\EE \to E} \lesssim e^{bt}, \quad
\|\mathscr T_J(t) \|_{\EE\to \EE} \lesssim_{a''} \, e^{a't}.
$$
As a consequence, we obtain
$$
\|(\AA \mathscr S_\BB)^{(*n)} (t) \|_{\EE \to E} \lesssim_{a'} e^{[(1-1/\ell) a' + b/\ell] \, t },
$$
which concludes the proof by fixing $\ell$ large enough so
that $(1-1/\ell) a'' + b'/\ell < a'$. The estimate on $(\mathscr S_\BB
\AA)^{(*n)}$ is proved by the same argument. 
\end{proof}
 
 %%%%%%%%%%%%%%%%%%%%%%%%%%%%%%%%%%%%%%%%%%%%%%%%%%%%%%%%%%%%%%%%
%
%  FOKKER - PLANCK HOMOGENEOUS
%
%%%%%%%%%%%%%%%%%%%%%%%%%%%%%%%%%%%%%%%%%%%%%%%%%%%%%%%%%%%%%%%%

 \bigskip\bigskip
 \section{The kinetic Fokker-Planck equation with flat confinement}
\label{sec:FPhomo}
\setcounter{equation}{0}
\setcounter{theo}{0}

This section is dedicated to the proof of semigroup decay estimates
for the kinetic Fokker-Planck equation~\eqref{eq:FP1} where the
confinement is ensured by spatial periodicity, for initial data in a
large class of Banach spaces, including the case of negative Sobolev
spaces, and with slow decay at large velocities. We deduce decay
estimates in Wasserstein distance as well. Our results apply to the
simpler case where the solution is spatially homogeneous and
solves~\eqref{eq:FP1bis}.

\subsection{Main result }
\label{sec:FP-MainR}

Consider the Fokker-Planck equation
\begin{equation}\label{eq:FPhomog}
  \partial_t f = \mathcal L f :=  \nabla_v \cdot \left( \nabla_v f +  F  f \right) - v
  \cdot \nabla_x f,
\end{equation}
on the density $f=f(t,x,v)$, $t \ge 0$, $x \in \T^d$ (the torus'
volume is normalised to one), $v \in \R^d$, where the (exterior) force
field $F = F(v) \in \R^d$ takes the form
\begin{equation}\label{def:forceF}
  F = \nabla_v \Phi \quad \mbox{with} \quad \fa |v| \ge R_0, 
  \ \Phi (v) = {1 \over \gamma} \,  \langle v \rangle^\gamma + \Phi_0 
\end{equation}
for some constants $R_0 \ge 0$ and $\gamma \ge 1$. Here and below, we
denote $\langle v \rangle := (1 + |v|^2)^{1/2}$.  We define $\mu (v)
:= e^{- \Phi (v)}$ with $\Phi_0 \in \R$ such that $\mu$ is a
probability measure. Observe that $\mu$ is a steady state for the
evolution equation \eqref{eq:FPhomog}. We shall consider separately
along this section the case where $f$ does not depend on $x$,
commenting on the simpler proofs and sharper estimates in this case.

Let us now introduce the key assumptions:

\medskip

\begin{center} {\bf Assumptions on the functional spaces} \end{center}
\smallskip

\noindent {\em \underline{Polynomial weights}}: For any
  $\gamma \ge 2$, $\sigma \in \{-1,0,1\}$ and $p \in [1,\infty]$, we
  introduce the weight functions
\begin{equation}\label{eq:def-wmpoly}
  m:= \langle v \rangle^k, \quad 
 k > |\sigma| + |\sigma| \sqrt{d} (\gamma-2) + \left(1-{1 \over p}\right) (d + \gamma-2) \
\end{equation}
 and the abscissa 
\begin{equation*}
  a_\sigma(p,m) := 
  \left\{ 
    \begin{array}{ll}
      |\sigma| + (1-1/p) d - k  & \hbox{if} \ \gamma = 2,\\[0.3cm] 
      -\infty  & \hbox{if} \ \gamma > 2.
    \end{array}
  \right.
\end{equation*}
\smallskip

\noindent
{\em \underline{Stretched exponential weights}}: For any $\gamma \ge
  1$, $\sigma \in \{-1,0,1\}$ and $p \in [1,\infty]$, we introduce the
  weight functions
\begin{align}\label{eq:def-wmexpo}
  m:= e^{\kappa \, \langle v \rangle^s} \,\,\,\, & \hbox{with} \  s \in
  [2-\gamma,\gamma), \,\, \kappa > 0, \,\, s > 0, \\ \nonumber &
  \hbox{or with}  \  s=\gamma, \,\, \kappa \in (0,1/\gamma),
\end{align}
 and the abscissa 
\begin{align*}
a_\sigma(p,m) &:=  
\left\{ 
    \begin{array}{ll}
      \kappa^2-\kappa & \hbox{if} \ \gamma=s=1,\\[0.3cm] 
      - \kappa s & \hbox{if} \ \gamma+s=2, \ s <\gamma, \\[0.3cm]
      -\infty  & \hbox{in the other cases}.
    \end{array}
  \right.
\end{align*}
 \smallskip

\noindent
{\em \underline{Definition of the spaces}}: For any weight function
  $m$, we define $L^p(m)$, $1 \le p \le \infty$, as the Lebesgue
  weighted space associated to the norm
$$
 \|f \|_{L^p(m)}:= \|f \, m \|_{L^p},
$$
and $W^{1,p}(m)$, $1 \le p \le \infty$, as the Sobolev weighted space
associated to the norm
$$
 \|f \|_{W^{1,p}(m)}:= \left(\| m \, f  \|_{L^p}^p + \| m \, \nabla f  \|_{L^p}^p\right)^{1/p}
 $$
 when $p \in [1,\infty)$ and 
 $$
 \|f \|_{W^{1,\infty}(m)}:= \max \left\{ \| m \, f \|_{L^\infty} , \|
   m \, \nabla f\|_{L^\infty} \right\}.
 $$
We also define $W^{-1,p}(m)$, $p \in [1,\infty]$, as the %(non conventional) 
weighted  negative Sobolev   space associated to the dual norm
\begin{equation}\label{def:W-1p}
  \|f \|_{W^{-1,p}(m)}:=  \|f \, m \|_{W^{-1,p}}:=  
  \sup_{ \|\phi \|_{W^{1,p'} } \le 1}  \langle f, \phi \, m \rangle,
  \quad p' := \frac{p}{p-1},
\end{equation}
where it is worth insisting that in this last equation the condition
$\|\phi \|_{W^{1,p'} } \le 1$ refers to the standard Sobolev space
$W^{1,p'} $ (without weight).

\medskip

Observe that for $\sigma\in \{-1,0,1\}$, $p \in [1,\infty]$ and $m$
satisfying \eqref{eq:def-wmpoly} or \eqref{eq:def-wmexpo}, the Sobolev
space $W^{\sigma,p}(m)$ defined as above is such that $1 \in
W^{{\sigma',p'}}(m^{-1})$ with $\sigma' = - \sigma \in \{-1,0,1\}$ and
$p' := p/(p-1) \in [1,\infty]$.  As a consequence, for any $f \in
W^{\sigma,p}(m)$, we may define the ``{\it mass of $f$}'' by
$$
\langle \langle f \rangle \rangle := \langle f, 1
\rangle_{W^{\sigma,p}(m),W^{{\sigma',p'}}(m^{-1})} = \left\langle m \,
  f, \frac1m \right\rangle_{W^{\sigma,p},W^{\sigma',p'}}
 $$
 where the double bracket recalls that there are two variables $x$ and
 $v$. In the case $\sigma=0,1$, there holds $W^{\sigma,p}(m) \subset
 L^{1}$ (here $L^1$ denotes the usual Lebesgue space without weight)
 and therefore the {\it ``mass of $f$''} corresponds to the usual
 definition
\[
\langle \langle f \rangle \rangle := \int_{\T^d \times \R^d} f(x,v) \dd
x \dd v
\]
and else this is the mass of the associated measure. Observe also that
when $f$ does not depend on $x$, this reduces thanks to the
normalisation of the torus volume to
\[
\langle \langle f \rangle \rangle = \langle f \rangle := \int_{\R^d}
f(v) \dd v.
\]

We finally define the projector $\Pi_1^\perp$ on the orthogonal
supplementary of the first eigenspace:
$$
\forall \, f \in W^{\sigma,p}(m), \quad \Pi_1^\perp f := f - \langle
\langle f \rangle \rangle \, \mu.
$$

\medskip

\begin{theo}\label{theo:vFP} 
  Consider $\sigma \in \{-1,0,1\}$ and $m$, $p \in [1,+\infty]$ that
  satisfy conditions \eqref{eq:def-wmpoly} or \eqref{eq:def-wmexpo}
  above (this implies $a_\sigma(p,m) < 0$). For any
  $a > \max\left\{a_\sigma(p,m),-\lambda \right\}$, there exists
  $C_a = C_a(\sigma,p,m)$ such that for any
  $f_0,g_0 \in W^{\sigma,p}(m)$ with the same mass, there holds
  \begin{equation} \label{eq:theovFPprelim} \left\| \mathscr
      S_{\LL}(t) f_{0} - \mathscr S_{\LL}(t) g_{0}\right\|_{
      W^{\sigma,p}(m)} \le C_{a} \, e^{at} \, \left\| f_{0} - g_0
    \right\|_{W^{\sigma,p}(m)},
\end{equation}
which implies in particular the relaxation to equilibrium 
  \begin{equation} \label{eq:theovFP} \left\| \mathscr S_{\LL}(t) f_{0} -
    \langle f_0 \rangle \mu \right\|_{ W^{\sigma,p}(m)} \le C_{a} \, e^{at}
    \, \left\| f_{0} - \langle f_0 \rangle \mu \right\|_{W^{\sigma,p}(m)},
\end{equation}
where $\lambda := \lambda(d,\sigma,p,m) > 0$ is constructive from the
proof. 

Moreover, when $\gamma \in [2,2 + 1/(d-1))$, there exists $\tilde
a(\gamma) < 0$ and for any $a > \tilde a(\gamma)$ there exists $C_a
\in (0,\infty)$ so that for any probability measures $f_0$, $g_0$ with
bounded first moments, there holds
\begin{equation}\label{eq:theovFPW1prelim}
   W_1\left(\mathscr S_{\LL}(t) f_{0} , \mathscr S_{\LL}(t) g_{0} \right)
   \le  C_a \, e^{a \, t} \, W_1( f_{0} , g_0)
\end{equation}
which implies the relaxation to equilibrium
\begin{equation}\label{eq:theovFPW1}
W_1\left(\mathscr S_{\LL}(t) f_{0} , \langle f_0 \rangle \mu \right)
   \le C_a \, e^{a \, t} \,
  W_1( f_{0} , \langle f_0 \rangle  \mu ).
\end{equation}
\end{theo}
%\textcolor{red}{CM~: J'ai ajout\'e la propri\'et\'e plus forte de
%  d\'ecroissance entre deux solutions g\'en\'erales car on peut
%  l'obtenir avec notre m\'ethode (remarque faite par Karl-Theodor
%  Sturm apres un expose)}

\begin{rems}  We first list the remarks in the spatially homogeneous case.

\begin{enumerate}

\item For $m = \mu^{-1/2}$, $p=2$ and $\sigma = 0$, \eqref{eq:theovFP}
  reduces to the classical spectral gap inequality for the
  Fokker-Planck semigroup in $L^2(\mu^{-1/2})$. 
  In that case the  semigroup spectral gap is equivalent to the Poincar\'e inequality.
  Denoting as $\lambda_P$ the best constant  in the Poincar\'e inequality, the estimate
  \eqref{eq:theovFP} holds with $a = - \lambda_{P}$ and $C_a = 1$.
  
\item Our proof in the general case is  based on the above mentioned semigroup spectral gap estimate in
  $L^2(\mu^{-1/2})$ and on the abstract extension Theorem~\ref{theo:Extension}.  
 More precisely, our approach allows one to prove an equivalence between
  Poincar\'e's inequality and semigroup decay of the Fokker-Planck
  equation in Banach spaces, including the case of negative Sobolev
  spaces. The meaning of the sentence is that the functional
  inequality \beqn\label{eq:PoincarIneq} \forall \, a' > a, \quad (\LL
  f , f)_{L^2(\mu^{-1/2})} \le a' \, \| \Pi_1^\perp f
  \|^2_{L^2(\mu^{-1/2})} \eeqn is equivalent to the semigroup decay
  estimate \eqref{eq:theovFP} for a large class of weight function
  $m$.

\item  For $\gamma \ge 2$, it has been proved recently in \cite{BGG}
  that a semigroup decay estimate similar to \eqref{eq:theovFP}
  holds for the Monge-Kantorovich-Wasserstein distance $W_{2}$, or in
  other words that for any probability measure $f_{0}$ with bounded
  second moment, there holds
\begin{equation}\label{eq:theovFPW2}
 W_2(\mathscr S_\LL(t) f_{0} , \langle f_0 \rangle  \mu ) \le C \, e^{\alpha \, t} \, 
  W_2( f_{0} , \langle f_0 \rangle  \mu ).  
  \end{equation}
  In the above inequality $C=1$ and $-\alpha$ is the optimal constant
  in the {\it ``$W\!J$ inequality''} (introduced in \cite[Definition
  3.1]{BGG}), which corresponds to the optimal constant in the {\it
    ``log-Sobolev inequality''} for convex potential and in particular
  $-\alpha$ is smaller than the optimal constant $\lambda_P$ in the
  Poincar\'e inequality \eqref{eq:PoincarIneq}.  Our estimate
  \eqref{eq:theovFPW1} can be compared to
  \eqref{eq:theovFPW2}. However we note that it had not been proved
  yet in the probability literature, even in the spatially
  homogeneous case, that the Poincar\'e inequality implies the
  convergence in $W_1$ distance (whereas the converse is known, see
  for instance \cite{MR3262508}).

\item It is worth emphasizing that in Theorem~\ref{theo:vFP}, the
  function space can be chosen smaller in term of tail decay than the
  space of  self-adjointness $L^2(\mu^{-1/2})$: one can choose for
  instance $L^2(\mu^{-\theta/2})$ with $\theta \in (1,2)$.

\item Note that this statement implies in particular that for a strong
  enough weight function, so that the essential spectrum move far
  enough to the left, there holds
\[
\Sigma\left(\LL\right) \subset \{ z\in \C \; | \; \Re e (z) \le
 -\lambda_{P}  \}\cup \{ 0\} 
\]
and that the null space of $\LL$ is exactly $\R \mu$.

\item Moreover, thanks to Weyl's Theorem, we know that in the
  $L^2(\mu^{-1/2})$ space the spectrum is constituted of discrete
  eigenvalues denotes as $\xi_\ell$, $ \ell \in \N$, with $\ell
  \mapsto \Re e \xi_\ell$ decreasing.  In any Banach space
  $W^{\sigma,p}(m)$, exactly the same proof as for
  Theorem~\ref{theo:vFP} (same splitting $\LL=\AA+\BB$ and same
  application of the abstract extension Theorem~\ref{theo:Extension})
  yields to the more accurate description of the spectrum
\[
\Sigma\left(\LL\right) \cap \Delta_{a_\sigma(p,m)} =  \{ \xi_\ell;  \; \Re e (\xi_\ell) > a_\sigma(p,m) \}  
\]
as well as the more accurate estimate \eqref{eq:estimSGLL} for any $a
> a_\sigma(p,m)$ and with $k$ defined by $k = \sup \{ \ell; \;\Re e
(\xi_\ell) > a_\sigma(p,m) \}$.

\item As a consequence of the preceding point, we may improve the
  intermediate asymptotic for the heat equation established in
  \cite{MR2727993}. Consider $g$ the solution to the heat equation
$$
\partial_t g = \Delta_v g, \quad g(0) = g_0,
$$
with $g_0 \in L^p(m)$, $m = \langle v \rangle^k$, $k > d/p' + n - 1$,
$n \in \N^*$. Assume furthermore that
$$
\forall \, \ell \in \N^d, \,\, |\ell| \le  n-1, \quad \int_{\R^d} g_0 \, H_\ell \dd x = 0,
$$
where $(H_\ell)$ stands for the family of Hermite polynomials (see
\cite{MR2727993} and the references therein).  In particular $\langle
g_0 \rangle = 0$ since $H_0 = 1$.  We observe that the function $f$
defined thanks to
$$
g(t,x) = R^{-d} \, f(\log R, v/R), \quad R = R(t) = \sqrt{1+2t}, 
$$
is a solution to the harmonic Fokker Planck equation 
$$
\partial_t f = \LL f =  \Delta_v f + \hbox{div}_v (v f), \quad f(0) = g_0,
$$
and that $(H_\ell)$ is an orthogonal family of eigenfunctions
associated to the adjoint operator $\LL^*$ ($H_\ell$ is associated to
the eigenvalue $|\ell| = \ell_1 + \dots + \ell_d$ for any $\ell \in
\N^d$).  An immediate application of our method implies
$$
 \left\| f_t   \right\|_{ L^p(m)} \le C_{d,p,n} \, e^{-nt}
    \, \left\| g_{0}  \right\|_{L^p(m)} \quad \forall \, t \ge 0, 
$$
which improves \eqref{eq:theovFP} (which holds in that context with
$a=-\lambda_P = -1$) whenever $n \ge 2$.  Coming back to the function
$g$ we obtain the optimal intermediate asymptotic estimate
$$
 \left\|  g_t   \right\|_{ L^p(m)} \le  { C_{d,p,n}  \over (1+t)^{n/2 + d/(2p')} } 
    \, \left\| g_{0}  \right\|_{L^p(m)} \quad \forall \, t > 0 . 
$$
That last estimate improves \cite[Corollary 4]{MR2727993} because the
range of initial data is larger and the rate in time is better (it is
in fact optimal).
\end{enumerate}
\end{rems}

\begin{rems} We now list the remarks specific to the spatially
  periodic case.
\begin{enumerate} 
\item The value of $\lambda$ in our quantitative estimate is related
  to the hypocoercivity estimate in $L^2(\mu^{-1})$ setting. However
  the best rate in general is the real part of the second eigenvalue
  defined by
\begin{equation}\label{def:FPhomo:lambda}
  \lambda :=
  \sup_{\| \cdot \| \sim \| \cdot
    \|_{W^{\sigma,p}(m)} } \inf_{ f \in C^\infty _c(\R^{d}) } 
 \left(  -  \frac{ \langle \LL f, f^* \rangle}{\| \Pi_1^\perp f \|} \right)
\end{equation}
where $C^\infty_c(\R^d)$ denotes the smooth compactly supported
functions, and where the supremum is taken over all norms
$\| \cdot \|$ on $W^{\sigma,p}(m)$ equivalent to the ambiant norm, and
where $f^* \in W^{{\sigma',p'}}(m)$ is the unique element in
$W^{{\sigma',p'}}(m)$ such that
$\| f \|^2 = \| f^*\|_*^2 = \langle f^*,f \rangle$, where
$\| \cdot \|_*$ is the corresponding dual norm. Let us mention that
similar results have been proved for diffusion processes in
\cite{MR2661206}.
 
\item Our result partially generalize to a spatially inhomogeneous
  setting the estimate on the Monge-Kantorovich-Wasserstein distance
  obtained recently in \cite{BGG}.
  
  % (4) Still for $\gamma \ge 2$, it is likely that our proof and
  % result extend to the case of the dual space of Lipschitz functions
  % (which vanishe in $v=0$) and therefore that \eqref{eq:theovFPW2}
  % holds for the Monge-Kantorovish-Wasserstein distance $W_{1}$ for
  % any $\alpha > \max(a_{-1}(1,m),-\lambda)$ and $C = C_{\alpha}$.

% \item There are many possible variants and extensions. We refer to
%   \cite{GMM} which also deal with the case $F = \nabla \Phi$, $\Phi
%   \not= \langle v\rangle^\gamma/\gamma$ and some cases where $F$ is
%   not the gradient of a potential.  In both cases, the fundamental
%   property used is that $F$ satisfies a Poincar\'e inequality and that
%   $F(v) \cdot v \ge C \, |v|^\gamma$, $C > 0$, for large values of
%   $|v|$ (which is a usual criterion for the Poincar\'e inequality).
%   Let us emphasize however that in Lemma~\ref{lem:FP-RegBB} the
%   function $B(v):= F - 2 \nabla m_0/m_0$, $m_0 := \exp(\Phi/2)$, does
%   not necessarily vanishes when $F \not= \nabla \Phi$. Anyway, we
%   can exhibit a weight function such that $B(v)$ is strictly smaller
%   than the leading negative term (at least for small perturbation of a
%   force field $\nabla \Phi$).  \textcolor{red}{cm~: pour cette
%     derniere remarque il faut mettre en avant le fait que l'on se
%     focalise ici sur les espaces de derivees positives et negatives}
%     {\Blue sm: on peut eventuellement enlever la fin de la remarque 
%     a partir de     ``Let us  emphasize however ... " qui est tres technique
%     et plus tres claire pour moi}

\item Our proof is based on the semigroup spectral gap estimate in
  $H^1(\mu^{-1/2})$ established in \cite{MNeu,MR2562709} and on the
  abstract extension Theorem. As a consequence, it gives an
  alternative proof for the semigroup spectral gap estimate obtained
  in \cite{DMScras,DMS} for the Lebesgue space $L^2(\mu^{-1/2})$.

\item Again, the proof holds for $F := \nabla \Phi + U$ which fulfills
  the conditions of \cite[section 3]{GMM}. In particular the
  associated Fokker-Planck operator does not take the $A A^* + B$
  structure of \cite{MR2562709} (where the term $\nabla_v ( U f)$ is
  included in the ``$B$'' part). 
\end{enumerate}
\end{rems}

The proof of Theorem~\ref{theo:vFP} is split into several steps:

\begin{enumerate}
\item We recall existing results for proving \eqref{eq:theovFP} in the
  space of $E= H^1(\mu^{-1/2})$: 
\begin{lem}\label{theo:KFPtorus-gapH1} {\bf (\cite[Theorem
    1.1]{MNeu})} The result in Theorem~\ref{theo:vFP} is true in
  the Hilbert space $H^1(\mu^{-1/2})$ associated to the norm
  $$
  \| f \|_{H^1(\mu^{-1/2})} := \left( \|f \|_{L^2(\mu^{-1/2})}^2 +
    \|\nabla_x f \|_{L^2( \mu^{-1/2})}^2 + \|\nabla_v f \|_{L^2(
      \mu^{-1/2})}^2 \right)^{1/2}.
  $$
\end{lem}
Such a result has been proved in \cite{MNeu}, see also
\cite{MR2034753,Herau2007,MR2130405,DMScras,DMS}. 

\item We devise an appropriate decomposition $\LL = \AA + \BB$ with
  $\BB = \LL - M \chi_R$ where $\chi_R$ is a smooth characteristic
  function of the set $|v| \le R$ with $M |\nabla_v \chi_R|$ small.

\item We need then to establish the dissipativity of $\BB$ in the
  spaces $W^{\sigma,p}(m)$ and of $B := \BB_{|E}$ in $E$. The
  coercivity of $\BB$ in these spaces is established in
  Lemma~\ref{lem:FPhomo-BLp}, \ref{lem:FPhomo-BW1p} and
  \ref{lem:FPhomo-BW-1}. The coercivity of $B$ in $E$ follows also
  from the same Lemma since the weight $m=\mu^{-1/2}$ is allowed. The
  latter could be proved by adapting the proof of \cite[Theorem
  1.1]{MNeu}. Or finally it could be checked more generally that the
  coercivity of $B$ in $E$ follows from that of $L$ combined with the
  strengthened Poincar\'e inequality as described below.
 % for the time-derivatives of
 %  $\|f \|_{L^2(\mu^{-1/2})}^2$ and
 %  $\|\nabla_x f \|_{L^2( \mu^{-1/2})}^2$ are unchanged, and one has
 %  additional terms $\int M |\nabla_v \chi_R| f |\nabla_v f| \mu^{-1}$
 %  in the inequality on the time-derivative of
 %  $\|\nabla_v f \|_{L^2( \mu^{-1/2})}^2$ and
 %  $\int M |\nabla_v \chi_R| f |\nabla_x f| \mu^{-1}$ in the inequality
 %  on the time-derivative of $\int \nabla_x f \nabla_v f \mu^{-1}$.
 %  These additional terms come with a small constant
 %  $M |\nabla_v \chi_R|$ is small, and the energy estimates remain true
 %  along this small perturbation.  \textcolor{red}{CM : j'ai
 %    bri\`evement expliqu\'e comment \'etendre les dissipativit\'e
 %    \'etablies sur $\mathcal L$ \`a $\mathcal B$, v\'erifie que la
 %    concision te va.}

\item We prove that the semigroup $\mathscr S_\BB(t)$ is regularizing
  in $L^2(\mu^{-1/2})$.

\item We conclude by applying Theorem~\ref{theo:Extension}.
\end{enumerate}

\begin{rem}
  Observe that since we need only applying regularization estimates
  for the semigroup of $\mathcal B$ after a composition by the
  operator $\mathcal A$, it is enough to prove these regularisation
  estimates with the usual weight $\mu^{-1/2}$.
\end{rem}

\subsection{Simplifications in the spatially homogeneous case} 
Let us start by pointing out the simplifications in the spatially
homogeneous case. First the decay~\eqref{eq:theovFP} in the space $E=
L^2(\mu^{-1/2})$ follows from the Poincar\'e inequality:

\begin{lem}\label{lem:PoincareInegalite} There exists a constant
  $\lambda_P >0$ so that for $f \in \DD(\R^d)$ with $\langle f \rangle = 0$
  \begin{equation}\label{eq:Poincare}
  \int_{\R^d} \left| \nabla_v \left( \frac{f}{\mu} \right) \right|^2
  \, \mu(v) \dd v  
  \ge  \lambda_P \, \int_{\R^d} f^2 \,  \, \mu^{-1}(v) \dd v  
\end{equation}
and moreover for $\lambda < \lambda_{P}$, there is $\eps(\lambda) > 0$
so that
\begin{eqnarray*}\label{eq:Poincare-fort}
  \int_{\R^d} \left| \nabla_v \left( \frac{f}{\mu} \right) \right|^2
  \, \mu(v) \dd v  
  &\ge&  \lambda \, \int_{\R^d} f^2   \, \mu^{-1}(v) \dd v 
  \\
  \nonumber
  &&+  \eps \,  \int_{\R^d}  \left(  f^2 \, \left|\nabla_v \Phi\right|^2 +
    |\nabla_v f |^2 \right) \, \mu^{-1}(v) \dd v.
\end{eqnarray*}
\end{lem}

\begin{proof}[Proof of Lemma~\ref{lem:PoincareInegalite}.]
  The proof of Lemma~\ref{lem:PoincareInegalite} is classical. We
  refer to \cite{BBCG} for a comprehensive proof of
  \eqref{eq:Poincare}. For the sake of completeness, we present a
  quantitative proof of \eqref{eq:Poincare-fort} as a consequence of
  \eqref{eq:Poincare} in the spirit of \cite{MRS}.

  On the one hand, by developing the LHS term, we find
$$
T  := \int_{\R^d} \left| \nabla_v \left( \frac{f}{\mu} \right)
\right|^2 \, \mu(v) \dd v  
=   \int_{\R^d} \left| \nabla_v f \right|^2 \, \mu^{-1} \dd v 
-  \int_{\R^d} f^{2}\, ( \Delta_v \Phi) \, \mu^{-1} \dd v .
$$

On the other hand, a similar computation leads to the following identity 
\begin{align*}
  T &= \int_{\R^d} \left| \nabla_v (f \mu^{-1/2}) \, \mu^{1/2} + (f
    \mu^{-1/2}) \, \nabla_v \mu^{1/2} \right|^2 \, \mu(v) \dd v
  \\
  &= \int_{\R^d} \left| \nabla_v (f \, \mu^{-1/2} ) \right|^2 \dd v +
  \int_{\R^d} f^{2}\, \left( \frac 14 |\nabla_v\Phi|^2 - \frac12 \Delta_v
    \Phi \right) \, \mu^{-1}\dd v.
\end{align*}

The two above identities together with \eqref{eq:Poincare} imply that
for any $\theta \in (0,1)$ 
\begin{align*} 
T &\ge (1-\theta) \lambda_P \,
\int_{\R^d} f^2 \, \mu^{-1} \dd v + \theta \int_{\R^d} f^{2}\, \left(
\frac1{16} |\nabla_v\Phi|^2 - \frac34 \Delta_v \Phi \right) \, \mu^{-1}
\dd v
\\
&+ \frac\theta{16} \int_{\R^d} f^{2}\, |\nabla_v\Phi|^2 \, \mu^{-1} \dd v+
\frac\theta2 \int_{\R^d} \left| \nabla_v f \right|^2 \, \mu^{-1} \dd v.
\end{align*}
Observe that $ |\nabla\Phi|^2 - 12 \Delta \Phi \ge 0$ for $v$ large
enough, and we can choose $\theta > 0$ small enough to conclude the
proof. \end{proof}

 We define 
\begin{equation}\label{eq:FPdefBhom}
\AA f := M \chi_R f, \qquad \BB f := \LL f - M \chi_R f
\end{equation}
where $M > 0$, $\chi_R(v) = \chi(v/R)$, $R > 1$, and
$0 \le \chi \in \DD(\R^d)$ is such that $\chi (v) = 1$ for any
$|v| \le 1$. The dissipativity estimates are proved as in the
spatially periodic case in
Lemmata~\ref{lem:FPhomo-BLp}-\ref{lem:FPhomo-BW1p}-\ref{lem:FPhomo-BW-1}-\ref{lem:FPhomo-BLipPrime}. Finally
the regularisation estimates are proved by using Nash's inequality:
\begin{lem}\label{lem:FP-RegBB}
  For any $1 \le p \le q \le \infty$ and for any $R,M$ as in the
  definition \eqref{eq:FPdefBhom} of $\BB$, there exists $b = b(R,M) >0$
  so that for any $\sigma \in \{-1,0, 1 \}$
\begin{equation}\label{eq:lemFP-RegBB1}
  \forall \, t \in [0,1], \quad \|\mathscr S_\BB(t) f \|_{W^{\sigma,q}(m)} \lesssim
  {e^{bt} \over t^{{d\over2} ({1 \over p} - {1 \over q}) }} \,  \|  f \|_{W^{\sigma,p}(m)}
\end{equation}
and for any $-1 \le \sigma < s \le 1$ 
\begin{equation}\label{eq:lemFP-RegBB2}
  \forall \, t \in [0,1], \quad \|\mathscr S_\BB(t) f \|_{H^s(m)} \lesssim
  {e^{bt} \over t^{s-\sigma}} \,  \|  f \|_{H^\sigma(m)}.
\end{equation}
\end{lem}

\begin{proof}[Proof of Lemma~\ref{lem:FP-RegBB}.]  
  The proof is classical and is a variation around Nash's inequality,
  together with Riesz-Thorin interpolation Theorem.  We refer for
  instance to \cite[Lemma~3.9]{GMM} for some similar results.
\end{proof}
 
\subsection{Dissipativity property of  $\BB$} 
\label{subsec:FPhomo:dissipB}
  We define 
\begin{equation}\label{eq:FPdefB}
\AA f := M \chi_R f, \qquad \BB f := \LL f - M \chi_R f
\end{equation}
where $M > 0$, $\chi_R(v) = \chi(v/R)$, $R > 1$,  and $0 \le \chi \in \DD(\R^d)$ is such that $\chi (v) = 1$ for any $|v| \le 1$. 

\begin{lem}\label{lem:FPhomo-BLp} 
  For any exponents $\gamma \ge 1$, $p \in [1,\infty]$, for any weight
  function $m$ given by \eqref{eq:def-wmpoly} or \eqref{eq:def-wmexpo}
  and for any $a > a_0(m,p)$, we can choose $R,M$ large enough in the
  definition \eqref{eq:FPdefB} of $\BB$ such that the operator $\BB-a$
  is dissipative in $L^p(m)$.
\end{lem}

\begin{proof}[Proof of Lemma~\ref{lem:FPhomo-BLp}.]
  We start by establishing an identity satisfied by the operator
  $\LL$.  For any smooth, rapidly decaying and positive function $f$,
  we make the splitting
  \begin{align*} 
    \int_{\T^d \times \R^d} (\LL \, f) \, f^{p-1} \, m^p \dd x \dd v
    = \int_{\T^d \times \R^d} f^{p-1} \, m^p \, (\Delta_v f + \hbox{\rm div}_v (F \,
    f)) \dd x \dd v =: T_1 + T_2.  
  \end{align*} 
  For the second term $T_2$, we use integration by part in $v$:
  \begin{eqnarray*} 
    T_2 &=& \int_{\T^d \times \R^d} f^{p-1} \, m^p \, \hbox{\rm div}_v (F \, f)
    \dd x \dd v \\
    &=& \int_{\T^d \times \R^d} f^{p-1} \, m^p \, ( \hbox{div}_v F \, f + F \cdot
    \nabla_v f) \dd x \dd v
    \\
    &=& \int_{\T^d \times \R^d} f^{p} \, ( \hbox{div}_v F) \, m^p \dd v - {1 \over
      p} \int_{\T^d \times \R^d} f^p \, \hbox{\rm div}_v (F \, m^p) \dd x \dd v
    \\
    &=& \int_{\T^d \times \R^d} f^{p} \, \left[ \left(1-{1 \over
          p}\right)  \hbox{div}_v F - F \cdot {\nabla_v m \over m}\right]\, m^p \dd x \dd v.  
\end{eqnarray*} 

For the first term $T_1$, we use integrations by part in $v$ and the
identity $m \nabla m^{-1}+ m^{-1} \nabla m = 0$ in order to get with
the notation $h= fm$
  \begin{eqnarray*} 
    T_1 %&=& \int_{\T^d \times \R^d} f^{p-1} \, m^p \, \Delta_v f \dd x \dd v \\ 
    &=& %1
    \int_{\T^d \times \R^d} h^{p-1} \, m \, \Delta_v \left( h m^{-1}\right) \dd x \dd v \\ 
    &=& %2
     - \int_{\T^d \times \R^d} \nabla_v \left( h^{p-1} \right) \cdot
    \left( \nabla_v h   + h m \nabla_v m^{-1}\right) \dd x \dd v 
    \\ && 
     - \int_{\T^d \times \R^d} h^{p-1} \nabla_v m \cdot  \left( \nabla_v h \, m^{-1} 
     + h \nabla_v m^{-1}\right)   \dd x \dd v   
      \\ &=& %3
    -   \int_{\T^d \times \R^d}  \nabla_v h^{p-1} \cdot \nabla_v h  \dd x \dd v  
  + \left(1 - \frac2p \right) \int_{\T^d \times \R^d} \left( \nabla_v h^p  \cdot
      \nabla_v m \right)   m ^{-1} \dd x \dd v  
      \\&&
            -  \int_{\T^d \times \R^d}  h^p \left( \nabla_v m  \cdot
      \nabla_v m^{-1} \right)   \dd x \dd v \\
    &=&  - (p-1) \int_{\T^d \times \R^d} |\nabla_v h|^2  \, h^{p-2} \dd x \dd v 
    \\&&
    + \int_{\T^d \times \R^d}h^p \left[  
      \left(\frac2p -1  \right) \nabla_v  \left( \frac{\nabla_v m}m \right) +   \frac{|\nabla_v m|^2}{m^2} \right] \dd x \dd v 
    \\
    &=& 
    - (p-1) \int_{\T^d \times \R^d} |\nabla_v (fm)|^2  \, (fm)^{p-2} \dd x \dd v 
    \\
    &&  + \int_{\T^d \times \R^d}(fm)^p \left[ 
      \left(\frac2p -1 \right)  \frac{\Delta_v m}m   +  2
      \left(1-\frac1p \right)    \frac{|\nabla_v m|^2}{m^2} \right] \dd
    x \dd v. 
     \end{eqnarray*} 
  All together, we then have established
\begin{multline}\label{eq:FPhomo-BLp}
 \int_{\T^d \times \R^d} (\BB \, f) \, f^{p-1}  \,  m^p  \dd x \dd v 
 \\ = 
 - (p-1) \int_{\T^d \times \R^d}  |\nabla_v (mf)|^2 \, (mf)^{p-2}\dd x \dd v  +
 \int_{\T^d \times \R^d}  f^p \, m^p \, \psi^0_{m,p} \dd x \dd v,
\end{multline} 
with
$$
\psi^0_{m,p} := \left( \frac2p -1 \right) \frac{\Delta_v m}m +
2 \left( 1- \frac1p \right) \frac{|\nabla_v m|^2}{m^2} + \left(1-{1 \over
    p}\right) \, \hbox{div}_v F - F \cdot {\nabla_v m \over m} - M \,
\chi_R.
$$
 Introducing the notation $s := 0$, $\kappa := 1$ when $m = \langle v \rangle^k$
and $k := s$ when $m := e^{\kappa \, \langle v \rangle^s}$, we have    
\begin{eqnarray*} 
  \psi^0_{m,p}  && = %1
  \quad \left( \frac2p -1 \right) \left(  \kappa k d \langle v \rangle^{s-2} + \kappa k (k-2) |v|^2 \langle v \rangle^{s-4} +
  \kappa^2 s^2 |v|^2 \langle v \rangle^{2s-4} \right)
  \\
  && \quad  + \left( 2- \frac2p \right)  \kappa^2 k^2 \, |v|^2 \, \langle v \rangle^{2s-4}
 + \left(1-{1 \over p}\right) \left( d \langle v \rangle^{\gamma-2} + (\gamma-2) \,|v|^2
  \,\langle v \rangle^{\gamma-4}\right)
  \\
  && \quad - \kappa k \, |v|^2 \, \langle v \rangle^{\gamma+s-4} - M
  \, \chi_R
\end{eqnarray*}
which gives the asymptotic behaviors
\begin{align*}
  \psi^0 _{m,p} (v) & \mathop{\sim}_{|v|\to \infty}  \kappa^2 s^2 |v|^{2s-2}
  - \kappa s \, |v|^{\gamma+s-2} & \hbox{if } s > 0 \\[0.2cm]
  \psi^0 _{m,p} (v)& \mathop{\sim}_{|v|\to \infty}  \left[ \left(1-{1 \over p}\right) (d +
    \gamma-2) - k \, \right] |v|^{\gamma-2} & \hbox{if } s = 0.
\end{align*} 

\smallskip As a consequence, when $m = e^{\kappa \, \langle v \rangle^s}$, $\gamma \ge s
> 0$, $\gamma+s \ge 2$, $\kappa > 0$ (with $\kappa < 1/\gamma$ if
$s=\gamma$) we obtain
 \begin{align*} 
   \psi^0_{m,p}&\xrightarrow[v \to \infty]{} \kappa^2 - \kappa &
\hbox{if } \gamma = s= 1, \\[0.2cm] 
\psi^0_{m,p}&\xrightarrow[v \to \infty]{} - \kappa s&
\hbox{if } \gamma + s= 2, \ s < \gamma, \\[0.2cm] 
\psi^0_{m,p}&\xrightarrow[v \to \infty]{} -\infty & \hbox{in the other cases.}  
\end{align*}

\smallskip When $\gamma \ge 2$ and $m = \langle v \rangle^k$, we get
\begin{align*}
\psi^0_{m,p}&\xrightarrow[v \to \infty]{} \left(1 - {1 \over p}\right) \, d - k & \hbox{if }
\gamma = 2,
\\
\psi^0_{m,p}&\xrightarrow[v \to \infty]{} -\infty & \hbox{if } \gamma > 2 \ \mbox{ and
  } \ k > \left(1-{1 \over p}\right) (d + \gamma-2).  
\end{align*} 
Observe that in all cases when $\gamma +s > 2$, we have
\begin{equation}\label{eq:psi0casgamma>2}
\psi^0_{m,p}  \,\, \mathop{\sim}_{|v|\to \infty} \,\, - \theta \,
\langle v \rangle ^{\gamma+s-2}, \quad \mbox{for some constant } \theta > 0.
\end{equation}

\Black We have then proved the following estimate: for any $a >
a_{p,m}$, $\theta' \in (0,a-a_0(p,m))$ small enough and $p \in
[1,\infty)$, we then can choose $R,M$ large enough in such a way that
$\psi^0_{m,p}(v) \le a-\theta'$ for any $v \in \R^d$, and
\begin{multline}\label{eq:BdissipLp}
\int_{\T^d \times \R^d} (\BB f) \, f^{p-1} \, m^p \dd x \dd v \le a\int_{\T^d \times \R^d} |f|^p \,
m^p \dd x \dd v \\ - \theta' \int_{\T^d \times \R^d} |f|^p \, m^p \, \langle v
\rangle^{\gamma+s-2}\dd x \dd v -
(1-p) \int_{\T^d \times \R^d} |\nabla_v(fm)|^2 \, (fm)^{p-2} \dd x \dd v.
\end{multline}
As a consequence and in particular, throwing out the two last terms, we have 
$$
\forall \, f \in L^p(m), \quad \|\mathscr S_{\BB}(t) f \|_{L^p(m)} \le e^{at} \,
\|f \|_{L^p(m)}.
$$
Since $p \mapsto a_0(p,m)$ is increasing, we may pass to the limit as
$p\to\infty$ in the above inequality and we thus conclude that $\BB-a$
is dissipative in $L^p(m)$ for any $p \in [1,\infty]$ and any $a >
a_0(p,m)$. \end{proof}

\begin{lem}\label{lem:FPhomo-BW1p} 
  For any exponents $\gamma \ge 1$, $p \in [1,\infty]$, for any weight
  function $m$ given by \eqref{eq:def-wmpoly} or \eqref{eq:def-wmexpo}
  and for any $a > a_1(m,p)$, we can choose $R,M$ large enough in the
  definition \eqref{eq:FPdefB} of $\BB$ such that the operator $\BB-a$
  is hypodissipative in $W^{1,p}(m)$.
\end{lem}

\begin{proof}[Proof of Lemma~\ref{lem:FPhomo-BW1p}.] 
  The decay of $\nabla_x \mathscr S_{\BB}(t) f = \mathscr S_{\BB}(t)
  \nabla_x f$ is proved as in Lemma~\ref{lem:FPhomo-BLp} since
  $x$-derivatives commute with the equation. We hence have
\begin{align*}
  \int_{\T^d \times \R^d} (\BB f) \, f^{p-1} \, m^p \dd x \dd v & \le a\int_{\T^d \times \R^d} |f|^p \,
m^p \dd x \dd v, \\ 
\int_{\T^d \times \R^d} \partial_{x_i} (\BB f) \, \partial_{x_i} f
|\partial_{x_i} f|^{p-2} \, m^p \dd x \dd v & \le a\int_{\T^d \times
  \R^d} |\partial_{x_i} f|^p \,
m^p \dd x \dd v.
\end{align*}

  For any $i \in \{ 1, \dots, d \}$, we compute
\begin{multline*}
  \int_{\T^d \times \R^d} (\partial_{v_i} \LL \, f) \, \partial_{v_i}
  f |\partial_{v_i} f|^{p-2} \, m^p \dd x \dd v \\ = \int_{\T^d \times
    \R^d} \partial_{v_i} f |\partial_{v_i} f|^{p-2} \, m^p \,
  \left(\Delta_v \partial_{v_i} f
    + \sum_{j=1} ^d \partial_{v_i} \partial_{v_j} \left( F_j \,  f\right)\right) \dd x \dd v \\
  - \int_{\T^d \times \R^d} \partial_{x_i} f \partial_{v_i} f
  |\partial_{v_i} f|^{p-1} m^p \dd x \dd v =: T_1 + T_2 + T_3.
\end{multline*}
For the first term $T_1$, proceeding exactly as in the proof of
Lemma~\ref{lem:FPhomo-BLp}, we find
\begin{multline*}
  T_1 = - (p-1) \int_{\T^d \times \R^d} |\nabla_v (m \partial_{v_i}
  f)|^2 \, \left|m \partial_{v_i} f\right|^{p-2} \dd x \dd v
  \\
  + \int_{\T^d \times \R^d} |\partial_{v_i} f|^p \, m^p \, \left\{
    \left(\frac2p -1 \right) \frac{\Delta_v m}m + 2 \left(1-\frac1p
    \right) \frac{|\nabla_v m|^2}{m^2} \right\} \dd x \dd v.
\end{multline*}
For the second term $T_2$,  we have
\begin{multline*}
  T_2 =   \int_{\T^d \times \R^d} \sum_{j=1} ^d \left( \partial_{v_i} \partial_{v_j} F_j \, f
    +  \partial_{v_i} F_j \, \partial_{v_j} f  \right)  \, \partial_{v_i} f
  |\partial_{v_i} f |^{p-2} m^p \dd x \dd v \\ +
  \int_{\T^d \times \R^d} |\partial_{v_i} f|^{p} \, \left[ ( \hbox{div}_v F) \,  \left(1-{1
        \over p}\right) - F \cdot {\nabla_v m \over m}\right]\, m^p \dd x \dd v.
\end{multline*}
For the third term $T_3$, we use Young inequality to split it as 
\begin{equation*}
  T_3 \le \var^{-1} \int_{\T^d \times \R^d} |\partial_{x_i} f|^p \,
  m^p \dd x \dd v + \var \int_{\T^d \times \R^d} |\partial_{v_i} f|^p
  \, m^p \dd x \dd v
\end{equation*}
where $\var$ will later be chosen small. 

Using the Young inequality, we get
\begin{eqnarray*}
  &&\sum_i  \int_{\T^d \times \R^d} (\partial_{v_i} \BB \, f) \, \partial_{v_i} f
  |\partial_{v_i} f|^{p-2}  \, m^p  \dd x \dd v \\
  && = \sum_i  \left\{T_1 + T_2 +T_3 -  \int_{\T^d \times \R^d} \partial_{v_i} (M \, \chi_R \,
    f) \, \partial_{v_i} f |\partial_{v_i} f|^{p-2}  \,  m^p \dd x \dd v \right\} 
  \\
  &&\le \int_{\T^d \times \R^d}     {|f|^p \over p'}   \,  Z \, m^p 
  +  \int_{\T^d \times \R^d} \psi_{m,p}^1 \, \left( \sum_{i=1} ^d |\partial_{v_i} f|^p
  \right) \, m^p \dd x \dd v \\
  && \quad \quad  + \var^{-1} \int_{\T^d \times \R^d}
    \left( \sum_{i=1} ^d
    |\partial_{x_i} f|^p \right) \, m^p \dd x \dd v,
 \end{eqnarray*}
 with 
\[
Z := \sum_{i,j=1} ^d |\partial_{v_i} \partial_{v_j} F_j| + (M/R) \, |(\hbox{div}_v \chi)_R|
\]
and
 $$
 \psi_{m,p}^1 := {1 \over p}\, Z + \frac1{p'}\sup_i \sum_{j}
 |\partial_{v_i} F_j|+ \frac1p \sup_j \sum_i |\partial_{v_i} F_j| + 
 \psi^0_{m,p} + \var.
$$

On the one hand, the function $Z$ is always negligible with respect to
the dominant term in $\psi^0_{m,p}$ (which is $F \cdot \nabla_v \ln
m$). On the other hand, we compute
\begin{align*}
\sup_i \sum_{j} |\partial_{v_i} F_j| & \le \left(1 + \sqrt{d} \, (\gamma-2) \right) \,
\langle v \rangle^{\gamma-2}, \\ 
\sup_j \sum_{i} |\partial_{v_i} F_j| & \le \left(1 + \sqrt{d} \, (\gamma-2) \right) \,
\langle v \rangle^{\gamma-2}.
\end{align*}
We deduce 
$$
\limsup \psi_{m,p}^1 \le \limsup \tilde \psi_{m,p}^1
$$
with 
$$
\tilde \psi_{m,p}^1 :=   \left(1 + \sqrt{d} \,  (\gamma-2)  \right) \,
\langle v \rangle^{\gamma-2} + \psi_{m,p}^0 + \eps.
$$

When $m = e^{\kappa \langle v \rangle^s}$, $\gamma \ge s > 0$,
$\gamma+s \ge 2$, $\gamma \ge 1$, $\kappa > 0$, we observe that
$\tilde \psi_{m,p}^1 \sim_{v \to \infty} \psi^0_{m,p}$, and when $m =
\langle v \rangle^k$, $\gamma \ge 2$, we observe that
\begin{align*} 
  \limsup_{v \to \infty} \tilde \psi_{m,p}^1 &\le 1 + \left(1- {1 \over
      p}\right) \, d - k + \var  \quad \hbox{if } \gamma = 2, 
  \\
  \limsup_{v \to \infty} \tilde \psi_{m,p}^1 & = -\infty \quad \hbox{if } \gamma > 2 \ \mbox{ and
  } \ k > 1 + \sqrt{d} (\gamma-2) + \left(1-{1 \over p}\right) (d + \gamma-2).
\end{align*}

Summing up, for any $a > a_1(p,m)$, $\eta \in (0,a-a_1(p,m))$ and $p
\in [1,\infty)$, we can choose $R,M$ large enough and $\var$ small
enough in such a way that $\psi^1_{m,p}(v) \le a-\eta$ for any $v \in
\R^d$. We then have established the following estimate
\begin{multline*}\label{eq:BdissipW1p}
 \sum_{i=1} ^d \int_{\T^d \times \R^d} (\partial_{v_i} \BB f)
  \, \partial_{v_i} f |\partial_{v_i} f|^{p-2} \, m^p \dd x \dd v \le 
  \\
\le  C % \left(1 + \frac M R\right)
  \int_{\T^d \times \R^d} |f|^p \, m^p \, \langle v \rangle^{\gamma-2}
  \dd x \dd v + C \int_{\T^d \times \R^d} \left( \sum_{i=1} ^d
    |\partial_{x_i} f|^p
  \right) \, m^p \dd x \dd v \\
  + a\int_{\T^d \times \R^d} \left( \sum_{i=1} ^d |\partial_{v_i} f|^p
  \right) \, m^p \dd x \dd v \\ - \frac{\theta}{2}\int_{\T^d \times \R^d} \left(
    \sum_{i=1} ^d |\partial_{v_i} f|^p \right) \, m^p \, \langle v
  \rangle^{\gamma+s-2} \dd x \dd v \\ - (p-1)\, \sum_{i,j=1} ^d
  \int_{\T^d \times \R^d} |\partial_{v_i} ((\partial_j
  f)m)|^2 \partial_{v_i} f |\partial_{v_i} f|^{p-2} m^p \dd x \dd v
\end{multline*}
where $C$ depends on $M$ and $R$. 

As a consequence, any solution $f$ to the linear evolution equation
 $$
 \partial_t f = \BB \, f, \quad f(0) = f_0 \in W^{1,p}(m)
 $$
 satisfies 
 \begin{multline*}
   {{\rm d} \over {\rm d}t}\int_{\T^d \times \R^d} \left( \sum_{i=1}
     ^d
     |\partial_{v_i} f |^p \right)\, {m^p\over p} \dd x \dd v  \le C
   \int_{\T^d \times \R^d} |f|^p \, m^p \, \langle v
   \rangle^{\gamma-2} \dd x \dd v \\ + C \int_{\T^d \times \R^d}
   \left( \sum_{i=1} ^d
    |\partial_{x_i} f|^p
  \right) \, m^p \dd x \dd v + a \int_{\T^d \times \R^d}
   \left( \sum_{i=1} ^d |\partial_{v_i} f |^p \right) \, m^p \dd x \dd v.
 \end{multline*}
 Defining the equivalent norm $\| \cdot  \|_{\tilde W^{1,p}(m)}$ thanks to 
\[
\|f \|_{\tilde W^{1,p}(m)}^p:= \| f \|_{L^p(m)}^p + \sum_{i=1} ^d \|\partial_{x_i} f
\|_{L^p(m)}^p + \zeta \sum_{i=1} ^d \|\partial_{v_i} f
\|_{L^p(m)}^p 
\]
and choosing $\zeta> 0$ small enough, we conclude thanks to
Lemma~\ref{lem:FPhomo-BLp} and the estimate \eqref{eq:psi0casgamma>2}
that $(\BB-a)$ is dissipative in $\tilde W^{1,p}(m)$ for any $a >
a_1(p,m)$ and $p \in (1,\infty)$, and therefore in $W^{1,p}(m)$ for
any $a > a_1(p,m)$ and $p \in [1,\infty]$. \end{proof}

%************************************************************************************************************************

\begin{lem}\label{lem:FPhomo-BW-1} 
  For any $p \in [1,\infty]$, for any force $F$ given by
  \eqref{def:forceF}, any weight function $m$ given by
  \eqref{eq:def-wmpoly} or \eqref{eq:def-wmexpo}, and for any $a >
  a_{-1}(m,p)$, we can choose $R,M$ large enough in the definition
  \eqref{eq:FPdefB} of $\BB$ such that the operator $\BB-a$ is
  hypodissipative in $W^{-1,p}(m)$.
\end{lem}

\begin{proof}[Proof of Lemma~\ref{lem:FPhomo-BW-1}.] We split the
  proof into three steps.

\smallskip\noindent{\sl Step 1. }
%We start by making explicit the expression of $\BB^*_m$. 
We first observe that if 
$$
\CC f := A \, f + B \cdot \nabla_v f + \Delta_v \, f - v \cdot
\nabla_x f,
$$
and we make the change of unknown $h := f m$ with $m=m(v)$, then the
corresponding operator $\CC_m h = m \, \CC(m^{-1} h)$ writes
\begin{equation*}
  \CC_m h := A_m \, h + B_m \cdot \nabla_v h + \Delta_v \, h - v \cdot
  \nabla_x h
\end{equation*}
with 
\begin{equation*}
A_m := \left[- {\Delta_v m \over m} + 2 \, {|\nabla_v m|^2 \over m^2} 
    + A - B \cdot {\nabla_v m \over m} \right], \qquad 
   B_m := \left[B - 2 \, {\nabla_v m \over m} \right].
\end{equation*}

We also observe that the dual operator $\CC^*$ writes
$$
\CC^* \phi := A^* \, \phi + B^* \cdot \nabla_v \phi + \Delta_v \, \phi + v\cdot
\nabla_x \phi
$$ 
with
$$
\quad A^* :=  (A -  \hbox{div}_v \, B), \ B^*:= - B. 
$$

Defining 
$$
\BB f := (\hbox{div}_v F - M \chi_R) \, f + F \cdot \nabla_v f +
\Delta_v f - v \cdot \nabla_x f
$$
and using the two above identities, we get 
\begin{equation}\label{eq:defB*m}
  \BB^*_m \phi = 
  \left[ {\Delta_v m \over m}   - M \chi_R - F \cdot {\nabla_v m \over m}
  \right] \, \phi 
  - \left[F - 2 \, {\nabla_v m \over m} \right]\cdot \nabla_v \phi  +
  \Delta_v \phi + v \cdot \nabla_x \phi.
\end{equation}

Besides, any solution $g$ to the equation 
$$
\partial_t g = \alpha \, g + \beta \cdot \nabla_v g + \Delta_v g \pm v
\cdot \nabla_x g
$$
satisfies at least formally (by performing two integrations by parts) the identity
\begin{multline*} 
  {{\rm d} \over {\rm d}t} \int_{\T^d \times \R^d} {|g|^p \over p}
  \dd x \dd v = - (p-1) \int_{\T^d \times \R^d} |\nabla_v g|^2 \,
  |g|^{p-2} \dd x \dd v \\ +
  \int_{\T^d \times \R^d} \left(\alpha - {\hbox{div}_v \beta \over p} \right) \,
  |g|^p \dd x \dd v.
\end{multline*}
As a consequence, for $\phi$ solution to the equation
\begin{equation}\label{eq:eqSBm*}
\partial_t \phi = \BB^*_m \phi, 
\end{equation}
we have
\begin{equation*}
{{\rm d} \over {\rm d}t}  \int_{\T^d \times \R^d} {|\phi|^p \over p} \dd x \dd v 
\le   \int_{\T^d \times \R^d} |\phi|^p \, \psi^2_{p,m} \dd x \dd v,
\end{equation*}
with 
\begin{equation}\label{eq:def-phipm}
  \quad   \psi^2_{p,m} :=  \left(1 - {2 \over p}\right) {\Delta_v m \over
    m}  
  +   {1 \over p}\, \hbox{div}_v F + {2 \over p} \,  {|\nabla_v m|^2 \over
    m^2} 
  - F \cdot {\nabla_v m \over m}  - M \chi_R  . 
\end{equation}
Recalling that 
$$
{\Delta_v m\over m} \mathop{\sim}_{v \to \infty} k\kappa \, (d+s-2)
|v|^{s-2} + k^2 \kappa^2 \, |v|^{2s-2}, \qquad \hbox{div}_v F \mathop{\sim}_{v \to \infty}
(d+\gamma-2) \, |v|^{\gamma-2},
$$
$$
{|\nabla_v m|^2\over m^2} \mathop{\sim}_{v \to \infty} \kappa^2 k^2 \,
|v|^{2s-2}, \qquad F \cdot {\nabla_v m \over m} \mathop{\sim}_{v \to
  \infty} k\kappa |v|^{\gamma+s-2},
$$
we have for an exponential weight function (so that $s > 0$ and
$k=s$)
$$
 \psi^2_{p,m}\mathop{\sim}_{v \to \infty}  \kappa^2  s^2  \, |v|^{2s-2} - s\kappa
|v|^{\gamma+s-2},
$$
and for a polynomial weight function (so that $s=0$), we have 
$$
 \psi^2_{p,m} \mathop{\sim}_{v \to \infty} \left( { d+\gamma-2 \over p} - k  \right)   \,
|v|^{\gamma-2},
$$
with again $\psi^2_{p,m}(v) \sim - \theta \langle v
\rangle^{\gamma+s-2}$ for large $v$ when $\gamma + s >2$.  In both
case, we conclude that for any $a > a_{p,m}$
\begin{equation}\label{eq:phiLp} 
 \frac1p \dt \|\phi
    \|^p_{L^p(\R^d)} \le a \|\phi \|^p_{L^p(\R^d)} - \theta' \left\| \phi \langle \cdot
  \rangle^{(\gamma+s-2)/p} \right\|^p_{L^p(\R^d)}
\end{equation}
for some small $\theta'$, uniformly when $p \to \infty$.

\medskip
\noindent {\sl Step 2. }
Now, we write
\begin{eqnarray*}
\partial_t (\partial_{v_i} \phi)
&=& \partial_{v_i} \BB^*_m \phi \\ 
&=&  \Delta_v (\partial_{v_i} \phi)   
- \sum_{j=1} ^d \partial_{v_i} \left(F_j - 2 \, {\partial_{v_j} m \over m}
\right)  \, \partial_{v_j} \phi
 - \left[F - 2 \, {\nabla_v m \over m} \right]\cdot \nabla_v (\partial_{v_i} \phi)
\\
&&+  \left[ {\Delta_v m \over m}  - F \cdot {\nabla_v m \over m} - M \,
  \chi_R  \right]\, \partial_{v_i} \phi
 +  
 \partial_{v_i} \left[ {\Delta_v m \over m}  - F \cdot {\nabla_v m \over m} - M \, \chi_R  \right] \, \phi
 \\
&&- v\cdot \nabla_x \left( \partial_{v_i} \phi \right)
- \partial_{x_i} \phi\\
 &=:&  \Delta_v (\partial_{v_i} \phi)   - \sum_{j=1} ^d \partial_{v_i}B^*_{m,j} \, (\partial_{v_j}
 \phi)  - B^*_m \cdot \nabla_v (\partial_{v_i} \phi) + A^*_m \, (\partial_{v_i} \phi)
 +  (\partial_{v_i}  A_m ^*)\, \phi \\
&&- v\cdot \nabla_x \left( \partial_{v_i} \phi \right)
- \partial_{x_i} \phi.
\end{eqnarray*}
By integration by parts, we deduce 
\begin{multline*}
  \frac1p {{\rm d} \over {\rm d}t} \int_{\T^d \times \R^d} \left(
    \sum_{i=1} ^d |\partial_{v_i} \phi|^p\right) \dd x \dd v =
  \sum_{i=1} ^d \int_{\T^d \times \R^d} (\partial_t \partial_{v_i}
  \phi) \, \partial_{v_i} \phi |\partial_{v_i} \phi|^{p-2} \dd x \dd v
  \\
  \le \int_{\T^d \times \R^d} \left(A_m ^* + {1 \over p} \, \hbox{div}_v
    B_m ^*\right) \left( \sum_{i=1} ^d |\partial_{v_i} \phi|^p \right) \dd
  x \dd v\\
  + \sum_{i=1} ^d \int_{\T^d \times \R^d} \Big[(\partial_{v_i} A_m ^* )\, \phi -
  ( \partial_{v_i}B_{m,j}^*) \, \partial_{v_j} \phi
  \Big]\, \partial_{v_i} \phi |\partial_{v_i} \phi|^{p-2} \\ +
  \sum_{i=1} ^d \int_{\T^d \times \R^d} \partial_{x_i} \phi
  \, \partial_{v_i} \phi |\partial_{v_i} \phi|^{p-2} \dd x \dd v \\
  \le \int_{\T^d \times \R^d} \left( \frac{\var^p}{p} + \psi^3_{p,m} + { 1 \over
      p}\, \sup_{i=1, ..., d} |\partial_{v_i}A_m ^* | \right) \, \left( \sum_{i=1} ^d
    |\partial_{v_i} \phi|^p \right) \dd x \dd v \\ + \frac{1}{p' \var^{p'}}
  \int_{\T^d \times \R^d} \left( \sum_{i=1} ^d |\partial_{x_i} \phi|^p
  \right) \dd x \dd v + {1 \over p' } \int_{\T^d \times \R^d} \left(
    \sum_i |\partial_{v_i}A_m ^* | \right) \, |\phi|^p \dd x \dd v
\end{multline*}
where
$$
\psi^3_{p,m}:= \sup_{i=1,\dots, d} \sum_{j=1} ^d | \partial_{v_j} B_{m,i}
^*| + A_m ^*+ {1 \over p} \, \hbox{div}_v B_m ^*.
$$

\smallskip We have
$$
\sup_{i=1,\dots,d} \sum_{j=1} ^d | \partial_j B_{m,i} ^*| \le  \left(1 + (\gamma-2) \sqrt{d}\right) \, |v|^{\gamma-2}
+ 2 k \kappa \left(1 + (s-2) \sqrt{d}\right) \, |v|^{s-2}, 
$$
as well as 
\begin{align*}
  A_m ^*(v) +  {1 \over p} \, \hbox{div}_v  B_m ^*(v)
  \sim \left(1 - \frac2p\right) {\Delta m \over m} 
  +  \frac2p  { |\nabla  m|^2 \over m^2} +  \frac1p \, \hbox{div}F - F \cdot {\nabla m  \over m}
  \sim \psi^0_{p,m} (v)
\end{align*}
and 
\[
\partial_{v_i} A_m^* \sim \left\{ \begin{array}{l} 
k (\gamma-2) (k+d-3) \\[0.2cm]
\qquad \mbox{ when } \gamma \ge 2 \mbox{ and } m(v) =
\langle v \rangle^k, \\[0.3cm]
\left[ 2 \kappa^2 s^2 + \kappa^2 s^2 (2s - 4) \right] v_i |v|^{2s-4} -
\left[ 2 \kappa s + \kappa s (\gamma+s-4) \right] v_i |v|^{\gamma+s-4}
\\[0.2cm] \qquad \mbox{ when } \gamma \ge 1 \mbox{ and } m(v) =
e^{\kappa \langle v \rangle^s}, 
\end{array}
\right.
\]
which yields
$$ 
|\partial_{v_i} A_m ^*| \lesssim W(v), \quad W(v) := \langle v \rangle^{\max\{2s,s+\gamma\} -   3}.
$$

\smallskip
For an exponential weight function (so that $s > 0$), we have thus
$$
\psi^3_{p,m}(v) \sim  \psi^2_{p,m}(v) \sim  \kappa^2  k^2  \, |v|^{2s-2} - k\kappa |v|^{\gamma+s-2}
$$
and for a polynomial weight function (so that $s=0$), we have 
$$
\limsup \psi^3_{p,m}\le  \Bigl( 1 + (\gamma-2) \sqrt{d} + { d+\gamma-2 \over p} - k  \Bigr)   \, |v|^{\gamma-2}.  
$$

In both case, we conclude that for any $a > a_{1}(p,m)$ and for $M,R$
large enough
$$
\frac1p \dt \left( \sum_{i=1} ^d \|\partial_{v_i} \phi \|^p_{L^p}
\right) \le a \, \left( \sum_{i=1} ^d \|\partial_{v_i} \phi \|^p_{L^p}
\right) + C \left( \sum_{i=1} ^d \|\partial_{x_i} \phi \|^p_{L^p}
\right) + C \|\phi \|^p_{L^p(W)}
$$
for some $C$ depending on $a$, uniformly when $p \to \infty$. Defining
again the norm
\[
\|\phi \|_{\tilde W^{1,p}(m)}:= \| \phi \|_{L^p(m)} + \sum_{i=1} ^d \|\partial_{x_i} \phi
\|_{L^p(m)} + \zeta \sum_{i=1} ^d \|\partial_{v_i} \phi
\|_{L^p(m)}
\]
for $\zeta$ small enough, equivalent to $W^{1,p}(m)$, and using that $W \le
C \langle v \rangle^{\gamma+s-2}$, we obtain the following
differential inequality
$$
\frac1p \dt \| \phi \|_{\tilde W^{1,p}(m)} ^p \le a \, \| \phi
\|_{\tilde W^{1,p}(m)} ^p
$$
uniformly as $p \to \infty$.  We have thus proved
$$
\forall \, t \ge 0, \,\, \forall \, \phi \in W^{1,p}, \quad \|
\mathscr S_{\BB^*_{m}}(t) \phi \|_{W^{1,p}} \le C \, e^{at} \| \phi \|_{W^{1,p}}
$$
for some $C>0$ (depending on $a$), uniformly as $p \to \infty$.

\medskip
\noindent
{\sl Step 3. }
 For any $h \in W^{-1,p}$ and $\phi \in W^{1,p'}$,  we have 
\begin{eqnarray*}
  \langle \mathscr S_{\BB_m}(t) h , \phi \rangle
  & =& \langle  h , \mathscr S_{\BB^*_m}(t) \phi \rangle
  \\
  &\le& \| h \|_{W^{-1,p}}  \, \| \mathscr S_{\BB^*_m}(t) \phi
  \|_{W^{1,p'}} \le 
  C \, e^{ at } \,   \| h \|_{W^{-1,p}} \, \|  \phi \|_{W^{1,p'}},
\end{eqnarray*} 
so that 
\begin{eqnarray*}
  \forall \, h \in W^{-1,p}, \quad  \| \mathscr S_{\BB_m}(t)  \, h \|_{W^{-1,p}}  \le C \, e^{at} \,   \| h \|_{W^{-1,p}}.
\end{eqnarray*}
Then, coming back to the operator $\BB$, we conclude with 
$$
\| \mathscr S_{\BB}(t) \, f \|_{W^{-1,p}(m)} \le C  e^{a \, t}\, \|  f \|_{W^{-1,p}(m)},
$$
so that $\BB-a$ is hypodissipative in $W^{-1,p}(m)$ for any $1 \le p \le \infty$.  \end{proof}

\medskip We introduce for $\zeta >0$ the norm
$$
\|\psi \|_{\FF_\infty} := \max\left\{ \|\psi \, \langle v \rangle^{-1}
  \|_{L^\infty} \, ; \ \sup_{i=1,\dots,d} \|\partial_{x_i} \psi
  \|_{L^\infty}\, ; \ \zeta \sup_{i=1,\dots,d} \|\partial_{v_i} \psi \|_{L^\infty} \right\},
$$
and the associated space 
$$
\FF_\infty:= \left\{\psi \in W^{1,\infty}_{\mbox{{\scriptsize loc}}} ; \,\, \|\psi \|_{\FF_\infty} < \infty \right\} 
$$
and its dual $(\FF_\infty)'$. Observe that
\begin{multline*}
\|f \|_{L^1(\langle v \rangle)}:= \sup_{\phi \in L^\infty, \, \|\phi
  \|_{L^\infty} \le 1} \int_{\T^d \times \R^d} \langle v \rangle \, f \, \phi \dd x
\dd v \\ 
= \sup_{\psi \in L^\infty _{\mbox{{\scriptsize loc}}}; \|\psi \langle v \rangle^{-1} \|_{L^\infty} \le 1 }
\int_{\T^d \times \R^d} f \, \psi \dd x \dd v,
\end{multline*}
so that $L^1(\langle v \rangle) \subset (\FF_\infty)'  $. 

\begin{lem}\label{lem:FPhomo-BLipPrime} 
  Assume that $\gamma \in [2,2 + 1/(d-1))$, then for any 
  \[
  a> \tilde a_\gamma := (d-1) (\gamma-2) - 1,
  \] 
  (observe that $\tilde a_\gamma <0$ from the assumptions), we can choose
  $R,M$ large enough in the definition \eqref{eq:FPdefB} of $\BB$ such
  that the operator $\BB-a$ is dissipative in $(\FF_\infty)'$.
\end{lem}

\begin{proof}[Proof of Lemma~\ref{lem:FPhomo-BLipPrime}.] 
  The proof is an adaptation of the proof of
  Lemma~\ref{lem:FPhomo-BW-1}, and we sketch it briefly, writing only
  the needed formal a priori estimates.
%for the limit case $p=\infty$ and we do not present the  necessary regularization process obtained 
%by considering first $p \in (1,\infty)$ and then letting $p \to \infty$. 

\medskip
\noindent {\sl Step 1. } For any $\psi \in
L^\infty_{\mbox{{\scriptsize loc}}}$, we denote by $\psi_t := \mathscr
S_{\BB^*}(t) \psi$ the solution (when it exists) to the dual evolution
equation
\begin{equation}\label{eq:psiB*}
\partial_t \psi_t = \BB^* \psi_t, \quad \psi_0 = \psi, 
\end{equation}
with 
$$
\BB^* \psi := \Delta_v \psi - F \cdot \nabla_v \psi - M \, \chi_R \psi
+ v \cdot \nabla_x \psi.
$$
Introducing the new unknown $\phi := \psi \langle v \rangle^{-1}$, we
observe that when $\psi_t$ is a solution to \eqref{eq:psiB*}, then the
associated function $\phi_t$ is a solution to the rescaled equation
\begin{equation}\label{eq:phiB*m}
  \partial_t \phi_t = \langle v \rangle^{-1} \partial_t \psi_t =
  \langle v \rangle^{-1}\, \BB^*(\langle v \rangle \phi_t) =:
  \BB^*_{\langle \cdot \rangle} \phi_t,
  \quad \psi_0 = \psi, 
\end{equation}
where $\BB^*_{\langle \cdot \rangle}$ is defined by \eqref{eq:defB*m}. 

\medskip
\noindent
{\sl Step 2. }  We calculate
\begin{multline*}
\partial_t \partial_{v_i} \psi 
= \partial_{v_i} \BB^* \psi =  \Delta_v \partial_{v_i} \psi  - \left(\partial_{v_i} F_j\right)
\, \partial_{v_j} \psi - F_j \, \partial_{v_i} \partial_{v_j} \psi \\
 - M \, \chi_R  \, \partial_{v_i} \psi
 +  M \, \left(\partial_{v_i}  \chi_R \right)\, \psi + v \cdot
 \nabla_x \left( \partial_{v_i} \psi \right) + \partial_{x_i} \psi,
\end{multline*}
with $F_j \sim v_j \langle v \rangle^{\gamma-2}$ and $\partial_{v_i}
F_j \sim \delta_{ij} \, \langle v \rangle^{\gamma-2} + (\gamma-2) \,
v_i v_j \, \langle v \rangle^{\gamma-4}$. We deduce
\begin{multline*}
  \frac1p \dt \int_{\T^d \times \R^d} \left( \sum_{i=1} ^d |\partial_{v_i}
    \psi|^p \right) \dd x \dd v 
  \\
  \le \int_{\T^d \times \R^d}  \Big[ - \left(\delta_{ij} \, \langle v
    \rangle^{\gamma-2} 
    + (\gamma-2) \, v_i v_j \, \langle v \rangle^{\gamma-4}\right)
  \, \partial_{v_j} \psi
  - v_j \langle v \rangle^{\gamma-2} \, \partial_{v_j} \partial_{v_i} \psi  
  \\
  - M \, \chi_R  \,  \partial_{v_i} \psi  +  M \, \left(\partial_{v_i}
    \chi_R \right)\, \psi + \partial_{x_i} \psi
  \Big] \, \partial_{v_i} \psi |\partial_{v_i} \psi|^{p-2} \dd x \dd v
  =: T_1 +\dots + T_5,
 \end{multline*}
 with the convention of summation of repeated indices. 
 We compute
 \begin{eqnarray*}
   T_1 
   &=& - \int_{\T^d \times \R^d} \langle v \rangle^{\gamma-2} \, \left(
     \sum_{i=1} ^d |\partial_{v_i} \psi |^p \right) \dd x \dd v
   \\
   T_2 
   &\le& - \sum_{i=1} ^d (\gamma-2) \int_{\T^d \times \R^d}  \langle v \rangle^{\gamma-4}
   \, |v_i|^2 \, |\partial_{v_i} \psi |^p \dd x \dd v \\
   && \quad 
   +  \sum_{i\not=j}(\gamma-2) \int_{\T^d \times \R^d} |v|^2 \langle v \rangle^{\gamma-4} \,
   \left( {1 \over p} |\partial_{v_j} \psi|^p + {1 \over p'} \,
     |\partial_{v_i} \psi|^p \right) \dd x \dd v
   \\
   &\le& (d-1) (\gamma-2)  \int_{\T^d \times \R^d} |v|^2 \langle v \rangle^{\gamma-4}
   \left( \sum_{i=1} ^d |\partial_{v_i} \psi|^p \right) \dd x \dd v
   \\
   T_3 
   &=& -  {1 \over p}  \sum_{i,j =1}^d   \int_{\T^d \times \R^d}  v_j \langle v
   \rangle^{\gamma-2} \, \partial_{v_j} |\partial_{v_i} \psi|^p \dd x
   \dd v \\
   &=&  {1 \over p}  \sum_{i,j=1} ^d \int_{\T^d \times \R^d} \left( 
     \langle v \rangle^{\gamma-2} + (\gamma-2) \, v_i^2 \, \langle v
     \rangle^{\gamma-4} \right) |\partial_{v_i} \psi|^p \dd x \dd v
   \\
   &\le&  {d \over p}  (\gamma-1)   \int_{\T^d \times \R^d}   \langle v
   \rangle^{\gamma-2} \left( \sum_{i=1} ^d |\partial_{v_i} \psi|^p
   \right) \dd x \dd v
\end{eqnarray*}
and
\begin{eqnarray*}
  T_4
  &=& \int_{\T^d \times R^d}  \left[
    - M \, \chi_R  \,  \partial_{v_i} \psi  +   M
    \left[ (\partial_{v_i} \chi)(v/R)\right] {\langle v \rangle \over R}
    \, {\psi \over \langle v \rangle}
  \right] \partial_{v_i} \psi |\partial_{v_i} \psi|^{p-2} \dd x \dd v
  \\\
  &\le& 
  C \, M \, \|\psi \langle v \rangle^{-1} \|_{L^p} \left( \sum_{i=1}
    ^d \|\partial_{v_i} \psi \|_{L^p}^{p-1} \right)
\end{eqnarray*}
and 
\begin{equation*}
  T_5 \le \frac{\var^p}{p} \int_{\T^d \times \R^d} \left( \sum_{i=1} ^d
    |\partial_{v_i} \psi|^p \right) \dd x \dd v + \frac{1}{p' \var^{p'}} \int_{\T^d \times R^d} \left( \sum_{i=1} ^d
    |\partial_{x_i} \psi|^p \right) \dd x \dd v.
\end{equation*}
All in all, we have proved
\begin{multline*}
\frac1p \dt \int_{\T^d \times \R^d} \left( \sum_{i=1} ^d
  |\partial_{v_i} \psi|^p \right) \dd x \dd v
 \\ \le \left[ {d (\gamma-1) \over p} + \frac{\var^p}{p} + (d-1) (\gamma-2) - 1\right]
 \int_{\T^d \times \R^d} \left( \sum_{i=1} ^d |\partial_{v_i} \psi|^p
 \right) \dd x \dd v \\
 + \frac{1}{p' \var^{p'}} \int_{\T^d \times \R^d} \left( \sum_{i=1} ^d
    |\partial_{x_i} \psi|^p \right) \dd x \dd v  + C \, M \, \left\|\psi \langle v \rangle^{-1} \right\|_{L^p} \,
  \left( \sum_{i=1} ^d \|\partial_{v_i} \psi \|_{L^p}^{p-1} \right).
 \end{multline*}

 We recall that any solution $\phi_t$ of \eqref{eq:phiB*m} satisfies
 \eqref{eq:phiLp}.  Fixing $a > (d-1) (\gamma-2) - 1$, next $\zeta_0>
 0$ so that $a-\zeta_0 > (d-1) (\gamma-2) - 1$, and then fixing $M$
 and $R$ so that \eqref{eq:phiLp} holds with the choice $a-\zeta_0$,
 $M$, $R$, we have for any $\zeta \in (0,\zeta_0)$ and $K \ge 1$ the
 differential inequality
  \begin{multline*}
  \frac1p \dt \left[ \|\phi \|_{L^p}^p + \sum_{i=1}^d
    \| \partial_{x_i} \phi \|_{L^p} ^p + \zeta  \sum_{i=1} ^d
    \|\partial_{v_i} \psi \|_{L^p} ^p\right]
  \le  (a-\zeta_0) \, \left( \|\phi \|_{L^p}^p  + \sum_{i=1} ^d \|
    \partial_{x_i} \phi \|_{L^p} ^p \right) 
  \\
+ \zeta \,  \Bigg[  \left( {d (\gamma-1) \over p} + (d-1) (\gamma-2) - 1\right)
   \left( \sum_{i=1} ^d  \|\partial_{v_i} \psi \|_{L^p}^p \right) + C \, M \, {K^p \over p} \, \|\phi
   \|_{L^p}^p \\
   + {C \, M \over K} \left( \sum_{i=1} ^d \|\partial_{x_i} \psi
     \|_{L^p}^p \right) \Bigg].
    \end{multline*}

    Taking $K$, $p$ large enough and then $\zeta$ small enough, we
    deduce
    \begin{multline*}
      \frac1p \dt \left[ \|\phi \|_{L^p}^p + \sum_{i=1}^d
        \| \partial_{x_i} \phi \|_{L^p} ^p + \zeta \sum_{i=1} ^d
        \|\partial_{v_i} \psi \|_{L^p} ^p\right] \\ \le a \left[ \|\phi
        \|_{L^p}^p + \sum_{i=1}^d \| \partial_{x_i} \phi \|_{L^p} ^p +
        \zeta \sum_{i=1} ^d \|\partial_{v_i} \psi \|_{L^p} ^p\right]
    \end{multline*}
    uniformly for $p$ large. As a consequence, we get by Gronwall
    lemma and then passing to the limit $p\to\infty$
    $$
    \|\mathscr S_{\BB^*}(t) \psi \|_{\FF_\infty} = \|\psi_t
    \|_{\FF_\infty} \le e^{at} \, \|\psi \|_{\FF_\infty}.
    $$    
     We conclude the proof by duality. \end{proof}
     
%      
%(7) We notice that 
%$$
%\sup_{\phi \in \FF_0}|\langle f-g,\phi \rangle|= 
%\sup_{\phi \in \FF_1}|\langle f-g,\phi \rangle|= 
%\sup_{\phi \in \FF_2}|\langle f-g,\phi \rangle|,
%$$
%where
%$$
%\FF_0 := \{ \phi \in W^{1,\infty}_{loc}(\R^d); \,\, \|\nabla \phi \|_{L^\infty} \le 1 \}
%$$
%$$
%\FF_2 := \{ \phi \in W^{1,\infty}_{loc}(\R^d); \,\, \|\nabla \phi \|_{L^\infty} \le 1, \quad \phi(0) = 0\}
%$$

\subsection{Regularisation in the spatially periodic case}
 
We prove a regularization property of the kinetic Fokker-Planck
equation related to the theory of hypoellipticity. It can be
considered well-known and ``folklore'', but we include a sketch of
proof for clarity and in order to make explicit the estimate.  The
argument follows closely the methods and discussions in
\cite{Herau2007} and \cite[Section~A.21]{MR2562709}.

\begin{lem}\label{lem:reg-kfp}
  The semigroup $\mathscr S_{\BB}$ satisfies (with no claim of optimality on
  the exponents) first (gain of derivative in $L^2$ spaces)
\[
{\bf (1)} \quad \forall \, t \in [0,1], \ \forall \, k \in \N^* \quad 
\left\{ 
\begin{array}{l}\displaystyle
\| \mathscr S_{\BB}(t) f \|_{H^k(\mu^{-1/2})} \lesssim \frac{1}{t^{3k/2}} \, \| f\|_{L^2(\mu^{-1/2})},
\vspace{0.3cm} \\ \ds  
\| \mathscr S_{\BB}(t) f \|_{L^2(\mu^{-1/2})} \lesssim \frac{1}{t^{3k/2}} \, \| f\|_{H^{-k}(\mu^{-1/2})}.
\end{array}
\right.
\]
second (gain of integrability at order zero)
\[
{\bf (2)} \quad \forall \, t \in [0,1], \quad 
\left\{ 
\begin{array}{l}\displaystyle
\| \mathscr S_{\BB}(t) f \|_{L^2(\mu^{-1/2})} \lesssim \frac{1}{t^{(5d+1)/2}} \, \|
f\|_{L^1(\mu^{-1/2})}, \vspace{0.3cm} \\ \ds
\| \mathscr S_{\BB}(t) f \|_{L^\infty(\mu^{-1/2})} \lesssim \frac{1}{t^{(5d+1)/2}} \, \|
f\|_{L^2(\mu^{-1/2})}
\end{array}
\right.
\]
third (gain of integrability at order one)
\[
{\bf (3)} \quad \forall \, t \in [0,1], \quad 
\left\{ 
\begin{array}{l}\displaystyle
\| \nabla \mathscr S_{\BB}(t) f \|_{L^2(\mu^{-1/2})} \lesssim \frac{1}{t^{(5d+1)/2}} \, \|
\nabla f\|_{L^1(\mu^{-1/2})}, \vspace{0.3cm} \\ \ds
\| \nabla \mathscr S_{\BB}(t) f \|_{L^\infty(\mu^{-1/2})} \lesssim \frac{1}{t^{(5d+1)/2}} \, \|
\nabla f\|_{L^2(\mu^{-1/2})}
\end{array}
\right.
\]
fourth (gain of integrability at ordre minus one)
\[
{\bf (4)} \quad \forall \, t \in [0,1], \quad 
\left\{ 
\begin{array}{l}\displaystyle
  \| \mathscr S_{\BB}(t) f \|_{W^{-1,\infty}(\mu^{-1/2})} \lesssim \frac{1}{t^{(5d+1)/2}} \, \|
  f\|_{W^{-1,2}(\mu^{-1/2})}, \vspace{0.3cm} \\ \ds
  \|  \mathscr S_{\BB}(t) f \|_{W^{-1,2}(\mu^{-1/2})} \lesssim \frac{1}{t^{(5d+1)/2}} \, \|
  f\|_{W^{-1,1}(\mu^{-1/2})}.
\end{array}
\right.
\]
\end{lem}
% \textcolor{red}{CM : probl\`eme avec le point (4) ci-dessus : ce n'est
%   pas ce qu'il nous faut mais \c ca semble bien \^etre ce qui est
%   donn\'e par dualit\'e}

\begin{rem}
  We have not been able to find the precise form of these
  regularisation estimates in the literature, however regularisation
  estimates for kinetic Fokker-Planck equations are well-known, see
  for instance \cite{Herau2007}, \cite[Appendix~A.21.2]{MR2562709} on
  the analysis side and \cite{MR2673982,MR2917400} on the probability
  side.
\end{rem}

\begin{proof}[Proof of Lemma~\ref{lem:reg-kfp}.] We only sketch the
  proof which is similar to the arguments developed in
  \cite{Herau2007}, see also \cite[A.21.2 Variants]{MR2562709}, and in
  Lemma~\ref{lem:FP-RegBB}.

  \smallskip\noindent {\sl Step 1. Proof of inequality {\bf (1)}.}  We
  only prove the case $k=1$, higher exponents $k$ are obtained by
  differentiating the equation and applying the same argument. We
  write down the energy estimates for the solution $f$, its first
  derivatives, and the product of the first derivatives
\begin{align*}
  \dt \n{f}_{L^2(\mu^{-1/2})} ^2 \le& - \int_{\T^d \times \R^d} \left|
    \nabla_v (f/\mu) \right|^2 \mu \dd x \dd v \\
  \dt \n{\partial_{x_i} f}_{L^2(\mu^{-1/2})} ^2 \le& - \int_{\T^d \times \R^d}
  \left|
    \nabla_v (\partial_{x_i} f/\mu) \right|^2 \mu \dd x \dd v \\
  \dt \n{\partial_{v_i} f}_{L^2(\mu^{-1/2})} ^2 \le& - \int_{\T^d \times \R^d}
  \left| \nabla_v (\partial_{v_i} f/\mu) \right|^2 \mu \dd x \dd v \\ 
  & -  \int_{\T^d \times \R^d} \partial_{v_i} f \partial_{x_i} f \mu^{-1} \dd x \dd
  v + \int_{\T^d \times \R^d} \left| \partial_{v_i} f\right|^2 \mu^{-1}
  \dd x \dd v \\ 
& + \frac{M}2 \int_{\T^d \times \R^d} \left|\partial^2_{v_i} \chi_R \right|
|f|^2 \mu^{-1}   \dd x \dd v\\
  \dt \int_{\T^d \times \R^d} \partial_{x_i} f \partial_{v_i} f \mu^{-1} \dd x
  \dd v \le& - \int_{\T^d \times \R^d}
  \left| \nabla_x f \right|^2 \mu^{-1} \dd x \dd v \\ 
  & - 2 \int_{\T^d \times \R^d} \nabla_v \left( \partial_{v_i} f /\mu
  \right) \cdot \nabla_v \left( \partial_{x_i} f / \mu \right) \mu \dd
  x \dd v \\ 
  & + 2 M  \int_{\T^d \times \R^d} \chi_R \partial_x f \partial_v f
  \mu^{-1} \dd x \dd v \\ & + M  \int_{\T^d \times \R^d}
  \left| \partial_{v_i} \chi_R \right| |f| |\partial_x f|  \mu^{-1} \dd x \dd v.
\end{align*}
Observe also that 
\begin{equation*}
  \int_{\T^d \times \R^d}
  \left| \nabla_v (g/\mu) \right|^2 \mu \dd x \dd v = \int_{\T^d \times \R^d}
  \left| \nabla_v g \right|^2 \mu^{-1} \dd x \dd v + \int_{\T^d \times \R^d}
  |g|^2 \left( \frac{|v|^2}{2} - d \right) \mu^{-1} \dd x \dd v.
\end{equation*}
Define the energy functional 
\begin{multline*}
  \FF(t,f_t) := A \| f_t \|_{L^2(\mu^{-1/2})}^2 + a t \| \nabla_{v} f_t  \|_{L^2(\mu^{-1/2})}^2 
   \\  + 2 c t^{2} \langle \nabla_{v} f_t,  \nabla_{x} f_t
  \rangle_{L^2(\mu^{-1/2})} +  b t^3 \| \nabla_{x} f_t  \|_{L^2(\mu^{-1/2})}^2
\end{multline*}
with $a,b,c >0$, $c < \sqrt{ab}$ (positive definite) and $A$ large enough, and compute from above 
\begin{multline*}
  \dt \FF(t,f_t) \le - A\int_{\T^d \times \R^d} \left| \nabla_v
    (f_t/\mu) \right|^2 \mu \dd x \dd v +
  a  \| \nabla_{v} f_t  \|_{L^2(\mu^{-1/2})}^2 
   \\ + 4 c t \langle \nabla_{v} f_t,  \nabla_{x} f_t
  \rangle_{L^2(\mu^{-1/2})} +  3 b t^2 \| \nabla_{x} f_t  \|_{L^2(\mu^{-1/2})}^2
 \\ - b t^3 \sum_{i=1}^d
  \int_{\T^d \times \R^d}
  \left|\nabla_v (\partial_{x_i} f/\mu) \right|^2 \mu \dd x \dd v
  - a t  \sum_{i=1}^d \int_{\T^d \times \R^d} \left|\nabla_v
    (\partial_{v_i} f/\mu) \right|^2 \mu \dd x \dd v \\
  - at \sum_{i=1} ^d \int_{\T^d \times \R^d} \partial_{v_i} f \partial_{x_i} f \mu^{-1} \dd x \dd
  v + at \int_{\T^d \times \R^d} \left| \nabla_{v} f\right|^2 \mu^{-1}
  \dd x \dd v \\ + \frac{a t M}2 \sum_{i=1} ^d \int_{\T^d \times \R^d} \left|\partial^2_{v_i} \chi_R \right|
|f|^2 \mu^{-1}
  \dd x \dd v  - 2 c t^2
  \int_{\T^d \times \R^d} \left| \nabla_x f \right|^2 \mu^{-1} \dd x
  \dd v \\ - 4 c t^2 \sum_{i=1} ^d \int_{\T^d \times \R^d} \nabla_v \left( \partial_{v_i} f /\mu
  \right) \cdot \nabla_v \left( \partial_{x_i} f / \mu \right) \mu \dd
  x \dd v \\ + 4 c M t^2 \sum_{i=1} ^d \int_{\T^d \times \R^d} \chi_R \partial_{x_i} f \partial_{v_i} f
  \mu^{-1} \dd x \dd v   + 2 c M t^2 \sum_{i=1} ^d \int_{\T^d \times \R^d}
  \left| \partial_{v_i} \chi_R \right| |f| |\partial_x f|  \mu^{-1} \dd x \dd v.
\end{multline*}
which implies when the compatible conditions $c < \sqrt{ab}$, $2c >
3b$ and $A >> a,b,c, M$ are satisfied: 
\begin{equation*}
  \dt \FF(t,f) \le - K \left( \| \nabla_{v} f_t  \|_{L^2(\mu^{-1/2})}^2 
  + t^2 \| \nabla_{x} f_t \|_{L^2(\mu^{-1/2})}^2\right)  + C \int_{\T^d \times \R^d} f^2 \mu^{-1} \dd x \dd v
\end{equation*}
for some constants $K,C>0$. Since the $L^2(\mu^{-1/2})$ norm is decreasing over $t
\in [0,1]$ we deduce that
\begin{equation*}
 \forall \, t \in [0,1], \quad  \FF(t,f_t) \le \FF(0,f_0) + C
 \n{f_0}_{L^2(\mu^{-1/2})} \lesssim \FF(0,f_0)
\end{equation*}
which yields the first part of {\bf (1)} by simple iteration of this
gain. 

For the second part of (1) we first establish in a similar manner as
above 
\begin{equation*}
  \left\| \mathscr S_{\BB^*}(t) f \right\|_{H^k(\mu^{-1/2})} \lesssim \frac{1}{t^{3k/2}} \, \| f\|_{L^2(\mu^{-1/2})}
\end{equation*}
which means 
\begin{equation*}
  \left\| \mathscr S_{\mu^{-1/2} \BB^* (\mu^{1/2} \cdot)}(t) h \right\|_{H^k} \lesssim \frac{1}{t^{3k/2}} \, \| h\|_{L^2}
\end{equation*}
and by duality 
\begin{equation*}
  \left\| \mathscr S_{\mu^{-1/2} \BB (\mu^{1/2} \cdot)}(t) h \right\|_{L^2} \lesssim \frac{1}{t^{3k/2}} \, \| h\|_{H^{-k}}
\end{equation*}
which means (according to our definition of weighted dual spaces)
\begin{equation*}
  \left\| \mathscr S_{\BB}(t) f \right\|_{L^2(\mu^{-1/2})} \lesssim \frac{1}{t^{3k/2}} \, \| f\|_{H^{-k}(\mu^{-1/2})}.
\end{equation*}

\smallskip\noindent {\sl Proof of inequality {\bf (2)}. }  Since the
norms we consider are propagated by the flow it is no loss of generality
to reduce to $t \in [0,\eta]$, $0<\eta <<1$. We introduce the quantity
\begin{eqnarray*}
\GG(t,f) &:=& 
B \| f \|_{L^1(\mu^{-1/2})}^2  + t^Z \bar \FF(t,f_t) \\ 
\bar \FF(t,f_t) &:=& \Big( A \| f \|_{L^2(\mu^{-1/2})} ^2 
+ a t^2 \| \nabla_{v} f  \|_{L^2(\mu^{-1/2})}^2 \\
&&\qquad
+ 2 c t^4 \langle \nabla_{x} f,  \nabla_{v} f
\rangle_{L^2(\mu^{-1/2})} + b t^6 \| \nabla_{x} f
\|_{L^2(\mu^{-1/2})}^2 \Big)
\end{eqnarray*}
with $B >> A >>a,b,c$ and $c < \sqrt{ab}$ and $Z = (d+3)/2$.

A similar calculation as above yields, for well-chosen $A,a,b,c>0$: 
\begin{equation*}
  \dt \bar \FF(t,f_t) \le - K \left( \| \nabla_{v} f_t  \|_{L^2(\mu^{-1/2})}^2 
  + t^4 \| \nabla_{x} f_t \|_{L^2(\mu^{-1/2})}^2\right)  + C \int_{\T^d \times \R^d} f^2 \mu^{-1} \dd x \dd v
\end{equation*}
and we deduce 
\begin{multline*}
  \dt \GG(t,f) \le \frac{d B}2 \| f \|_{L^1(\mu^{-1/2})}^2 + Z t^{Z-1}
  \bar \FF(t,f_t) \\ - K t^Z \left( \| \nabla_{v} f_t
    \|_{L^2(\mu^{-1/2})}^2 
  + t^4 \| \nabla_{x} f_t \|_{L^2(\mu^{-1/2})}^2\right) + C t^Z \int_{\T^d \times \R^d} f^2 \mu^{-1} \dd x \dd v.
\end{multline*}
We choose $\eta$ small enough so that $Z t^{Z+1} << K t^Z$, and deduce 
\begin{multline*}
  \dt \GG(t,f) \le \frac{d B}2 \| f \|_{L^1(\mu^{-1/2})}^2 - \frac{K}2 t^Z \left( \| \nabla_{v} f_t
    \|_{L^2(\mu^{-1/2})}^2 
  + t^4 \| \nabla_{x} f_t \|_{L^2(\mu^{-1/2})}^2\right) \\ + C' t^{Z-1} \int_{\T^d \times \R^d} f^2 \mu^{-1} \dd x \dd v.
\end{multline*}
for some other constant $C'>0$. 

 The Nash inequality implies
\begin{equation}\label{eq:Nash}
  \int_{\T^d \times \R^d} f^2 \mu^{-1} \dd x \dd v \lesssim_d
  \Bigl( \int_{\T^d \times \R^d} |f| \mu^{-1/2} \dd x \dd v
  \Bigr)^{4 \over 2d + 2}\, \Bigl( \int_{\T^d \times \R^d} |
  \nabla_{x,v}(f \mu^{-1/2}) |^2  \dd x \dd v 
  \Bigr)^{2d \over 2d + 2}
\end{equation}
and using the Young inequality we have 
$$
\| f \|^2_{L^2(\mu^{-1/2})}  \le C_{\var,d} t^{-5d} \,  \| f
\|^2_{L^1(\mu^{-1/2})} +  \var t^5 \, \| \nabla_{x,v} f \|^2_{L^2(\mu^{-1/2})} ,
$$
for $\var$ small and $C_{\var,d}$ depending on $\var$ and the
dimension $d$. Taking $\var$ small we deduce 
\begin{equation*}
  \dt \GG(t,f) \le \frac{d B}2 \| f \|_{L^1(\mu^{-1/2})}^2 + C''
  t^{Z-1 - 5d} \| f \|_{L^1(\mu^{-1/2})}^2
\end{equation*}
for some constant $C''>0$. Finally choosing $Z = 5d+1$ we conclude
that 
\begin{equation*}
 \forall \, t \in [0,\eta], \quad  \GG(t,f_t) \le \GG(0,f_0) + C
 \n{f_0}_{L^1(\mu^{-1/2})} ^2 \lesssim \GG(0,f_0)
\end{equation*}
which yields the first part of {\bf (2)}. The second part can be
proved either by duality, or by using the inequality {\bf (1)} with
$k=d$ and Sobolev embedding (the constant is then slightly better:
$t^{-3d/2}$ which has no consequence for the rest of the paper). 

\smallskip\noindent {\sl Proof of inequality {\bf (3)}.} The proof of
the first part is similar to the proof of the first part of inequality
{\bf (2)} after differentiating the equation to get 
\begin{align*}
  & \partial_t \partial_{x_i} f + v \cdot \nabla_x \partial_{x_i} f =
  \nabla_v \cdot \left( \nabla_v \partial_{x_i} f + v \partial_{x_i} f
  \right) \\ 
  & \partial_t \partial_{v_i} f + v \cdot \nabla_x \partial_{v_i} f =
  \nabla_v \cdot \left( \nabla_v \partial_{v_i} f + v \partial_{v_i} f
  \right) - \partial_{x_i} f + \partial_{v_i} f
\end{align*}
(observe that it involves no term of order zero derivative). The
second part is proved by applying inequality {\bf (1)} to the
differentiated equation for $k=d$ together with Sobolev embedding.
 
\smallskip\noindent {\sl Proof of inequality {\bf (4)}.}  It follows
from {\bf (3)} by duality. 
\end{proof}

\begin{cor}\label{cor:Tn:FPtorus}
  For any $a > a_0$, there exist $n \ge 1$ and a constant such that
  for any spaces $E$ and $\EE$ of the type $ W^{\sigma,p}(m)$ as
  defined above, there holds
\begin{equation}\label{eq:TnEtoEE}
  \forall \, t \ge 0, \quad \| T_{n}(t) f \|_{E} \lesssim  e^{at}  \, \| f \|_{\EE} .
\end{equation}
\end{cor}
\begin{proof}[Proof of Corollary~\ref{cor:Tn:FPtorus}]
  The proof follows from the application of Lemma~\ref{lem:Tn} and
  Lemma~\ref{lem:reg-kfp} that implies that $\mathcal A \mathscr
  S_{\mathcal B}(t) \mathcal A$ maps any $W^{\sigma,p}(m)$ to
  $H^d(\mu^{-1/2})$ with some constant $C t^{-\Theta}$ with some
  $\Theta >0$. 
\end{proof}

\subsection{End of the proof of Theorem~\ref{theo:vFP}}

In the cases $\sigma \ge 0$, $1\le p < \infty$ and $\sigma = -1$, $1<
p \le \infty$ estimate \eqref{eq:theovFP} is an immediate consequence
of Theorem~\ref{theo:Extension} together with
Lemma~\ref{lem:PoincareInegalite}, Lemma~\ref{theo:KFPtorus-gapH1},
Lemma~\ref{lem:FPhomo-BLp}, Lemma~\ref{lem:FPhomo-BW1p},
Lemma~\ref{lem:FPhomo-BW-1}, Lemma~\ref{lem:FP-RegBB} and
Lemma~\ref{lem:Tn}.

\smallskip In the case $\sigma = 0$, $p= \infty$ so that $L^\infty(m)$
is not dense in $L^2(\mu^{-1/2})$ (for any choice of the weight $m$), we
remark that for any $\eps >0$ (small enough) there exists $p_{\eps}$
and $m_\var$ so that $L^\infty(m) \subset L^p(m_\eps)$ for any $p \ge
p_\eps$, so that estimate \eqref{eq:theovFP} holds in $L^p(m_\eps)$,
then in $L^\infty(m_\eps)$ by passing to the limit $p\to\infty$ and
finally in $L^\infty(m)$ by passing to the limit $\eps\to0$. We handle
the two last cases in \eqref{eq:theovFP} in a similar way.

\smallskip In order to prove \eqref{eq:theovFPW1}, we first observe
that combining Theorem~\ref{theo:Extension} together with
Lemma~\ref{lem:PoincareInegalite}, Lemma~\ref{theo:KFPtorus-gapH1},
Lemma~\ref{lem:FPhomo-BLipPrime} and Lemma~\ref{lem:Tn}, we have
established
$$
\| \mathscr S_{\LL}(t) f_{0} - \mathscr S_{\LL}(t) g_0 \|_{
  (\FF_\infty)' } \le C_{a} \, e^{at} \, \| f_{0} - g_0 \|_{
  (\FF_\infty)' }.
$$
Next, for any two probability measures $f,g$ with bounded first
moment, we have \begin{eqnarray*} W_1(f,g) &=& \sup_{\|\nabla \phi
    \|_{L^\infty} \le 1} \int_{\R^d} (f-g) \, \phi \dd v
  \\
  &=& \sup_{\|\nabla \phi \|_{L^\infty} \le 1} \int_{\R^d} (f-g) \,
  (\phi - \phi(0)) \dd v
  \\
  &=& \sup_{\max\{\|\langle v \rangle^{-1} \, \psi
    \|_{L^\infty} , \|\nabla \psi \|_{L^\infty} \} \le 1} \int_{\R^d}
  (f-g) \, \psi \dd v,
\end{eqnarray*} where we have used the Kantorovich-Rubinstein theorem (see for
instance \cite[Theorem 1.14]{VillaniTOT}) in the first line, the mass
condition in the second line and the change of test functions $\psi :=
\phi - \phi(0)$ on the last line. As a consequence the $W_1$ distance
and the distance associated to the duality norm $\|\cdot
\|_{(\FF_\infty)'}$ are equivalent, which ends the proof.

\section{The kinetic Fokker-Planck equation with potential confinement}
\label{sec:KFPrd}
\setcounter{equation}{0}
\setcounter{theo}{0}

\subsection{Main result }
\label{sec:KFPtorus-MainR}
Consider the kinetic Fokker-Planck equation in the whole space with a space confinement potential 
\begin{equation}\label{eq:KFPrd}
  \begin{cases}
    \partial_t f = \LL f := \CC f + \TT f,
    \\[2mm]
    \CC f := \nabla_v \cdot \left( \nabla_v f + v \, f\right), \\[2mm]
    \TT f := - v \cdot \nabla_x f + G \cdot \nabla_{v}f,
  \end{cases}
\end{equation}
on the density $f=f(t,x,v)$, $t \ge 0$, $x \in \R^d$,  $v \in \R^d$, where the  (exterior) force
field $G = G(x) \in \R^d$ is given by \
\begin{equation}\label{def:forceFx} 
  G(x) = \nabla_x \Psi(x) := x \, \langle x
  \rangle^{\beta-2} \quad \mbox{ with } \quad \fa |x| \ge R_1, \ 
  \Psi (x) := {1 \over \beta} \,  \langle x \rangle^\beta + \Psi_0
\end{equation}
for some constants $R_1>0$ and $\beta \ge 1$. 

The unique stationary state of the kinetic Fokker-Planck equation
\eqref{eq:KFPrd} is 
 $$
 \mu(x,v)= \exp ( - \Psi(x) - |v|^2/2 ) , 
$$
with the choice of the constant $\Psi_{0} \in \R$ so that $\mu$ is a
probability measure.

\smallskip 
We define the Hamiltonian function
 $$
  H(x,v) := 1+\Psi(x) + {|v|^2 \over 2},
 $$
 and we consider the following assumptions: 

\medskip

\begin{center}
{\bf Assumptions on the functional spaces} 
\end{center} 
\smallskip

\noindent {\em \underline{Polynomial weights}}: For $\beta \ge 2$ and
  $p \in [1,+\infty]$, we introduce the weight functions $m := H^k$ with $k >
  k(d,p)$ for some explicit $k(d,p)>d/p'$ from the proofs.
\smallskip

\noindent {\em \underline{Stretched exponential weights}}: For any
  $\beta \ge 1$ and $p \in [1,+\infty]$, we introduce the weight
  functions $m : = e^{\kappa \, H^s}$ with $s \in (0, 1)$ and
  $\kappa>0$.\smallskip

\noindent {\em \underline{Definition of the spaces}}: We then define on
  $\R^d \times \R^d$ the associated weighted Lebesgue spaces
  $\EE := L^p(m)$, %$\sigma \in \{-1,0,+1\}$,
  $p \in [1,+\infty]$.
% , in the same way as in the previous section in
% the case of $\T^d \times \R^d$.}

\smallskip
For any $f \in \EE$, the terms $\langle f\rangle$,
$\langle\!\langle f \rangle\!\rangle$ and $\Pi_1^\perp f$ are defined
as before. 

% \begin{rem}
%   Observe that this covers the space $L^1_{x,v}(H^k)$ with any
%   $k>0$. It is likely that higher Lebesgue spaces could be covered
%   with more technical efforts, we are however primarily interested on
%   $L^1$ and Wasserstein distances due to their physical relevances.
% \end{rem}
% we may define the local density and the total mass by 
% $$
% \langle f\rangle (x) := \int_{\R^d} f(x,v) \dd v \quad\hbox{and}\quad
% \langle\!\langle f \rangle\!\rangle := \int_{\T^d \times\R^d} f(x,v)
% \dd x \dd v.
% $$
% We then define  the projector $ \Pi_1^\perp$ on the orthogonal supplementary of the first eigenspace  thanks to 
% $$
% \forall \, f \in W^{\sigma,p}(m), \quad f^\perp = \Pi_1^\perp f := f -
% \langle\!\langle f \rangle\!\rangle \, \mu.
% $$ 

\begin{theo}\label{theo:KFPrdhyp1} 
  Consider one of the spaces $\EE$ defined above. Then there exists
  $a = a(\EE) < 0$ such that for any $f_0, g_0 \in \EE$ with same
  mass, the associated solutions $f_t$, $g_t$ of the kinetic
  Fokker-Planck equation \eqref{eq:KFPrd} satisfy
  \begin{equation*} \label{eq:theoKFPrdhyp1} 
    \|f_t - g_t \|_\EE \le
    C_{a} \, e^{at} \, \|f_0 - g_0 \|_\EE,
\end{equation*}
which implies the relaxation to equilibrium
\begin{equation*} \label{eq:theoKFPrdhyp1}
\|f_t - \langle\!\langle f_0 \rangle\!\rangle \,  \mu  \|_\EE
\le C_{a} \, e^{at} \, \|f_0 - \langle\!\langle f_0 \rangle\!\rangle \mu  \|_\EE, 
\end{equation*}
for some constructive constant $C_{a}>0$. 
% \textcolor{red}{CM : ici il me semble que l'on peut traiter
%   Wasserstein $W_1$ qui serait joli, sans beaucoup d'effort, j'attend
%   qu'on en discute d'abord. On peut aussi demontrer la decroissance
%   exponentielle entre deux solutions quelconques.}
\end{theo}

%\medskip\medskip
%\begin{theo}\label{theo:KFPrdhyp2} Assume $\beta = 2$. 
%Let us fix  a  weight function $m$  as in section~\ref{sec:FP-MainR}, so that $a_0(1,m) < 0$. There exists $\lambda_{PR} > 0$ so that for any $a > \max(a_0(1,m),-\lambda_{PR})$ there exists $C_a = C_a(m)$ such that for any  $f_0 \in L^1(m)$,  there holds
%\begin{eqnarray*} \label{eq:theoKFPrdhyp2}
%&&\|f_t - \langle\!\langle f_0 \rangle\!\rangle \,  \mu  \|_{ L^1(m)} 
%\le C_{a} \, e^{at} \, \|f_0 - \langle\!\langle f_0 \rangle\!\rangle \mu  \|_{L^1(m)}. 
%\end{eqnarray*}
%  for some constructive constant $C_{a}>0$ and a fixed constant $\lambda_{PR}$ which does not depend on the functional space   
%  $W^{\sigma,p}(m)$. 
%  \end{theo}

\begin{rems}
  \begin{enumerate}
  \item Such a semigroup spectral gap result for the kinetic
    Fokker-Planck equation in the whole space with a confining
    potential (and the same harmonic potential for the friction force
    acting on velocities) has been proved in the Sobolev spaces
    $H^\sigma(\mu^{-1/2})$, $\sigma \in \N^*$ in
    \cite{Herau_Boltz,MR2034753,MR2562709} and in the Lebesgue space
    $L^2(\mu^{-1/2})$ in \cite{DMScras,DMS} (inspiring from
    \cite{Herau_Boltz}). These last references provide also
    constructive estimates. 
  
%  \item It is likely that Theorem~\ref{theo:KFPrdhyp1} extends to
%    stretch exponential weight $e^{\kappa H^s}$, $\kappa > 0$, $s \in
%    (0,1)$. 
%\textcolor{red}{cm~: voir si on peut inclure ca}
  
  \item We did not include it in the statement for the sake of clarity
    but our method of proof can recover the semigroup decay estimate
    in $L^2(\mu^{-1/2})$ as a consequence of the known decay estimate
    in $H^1(\mu^{-1/2})$ proved in
    \cite{Herau_Boltz,MR2034753,MR2562709}, which provides an
    alternative argument to those in \cite{DMScras,DMS}.

  \item In the case of a polynomial weight, the sufficient condition
    $k > k(d,p)$ can be more precise: in our calculation we find
    $k(d,p) > d/p' + 3/2$. 

  \item We believe that, by combining the new estimates in this
    section with the strategy of the previous section,
    Theorem~\ref{theo:KFPrdhyp1} can be extended to Sobolev space
    $W^{\sigma,p}(m)$ for $\sigma = \pm 1$ and a polynomial weight
    with a condition $k>k(d,p)$ for some $k(d,p) > d/p' + 3/2$ greater
    or equal to the one in (3).

  \item We are not able at now to prove the exponential decay in
    Wasserstein distance, as for the torus case. The reason is that
    the minimal polynomial confinement (the condition on $k$) we can
    afford in the theorem is too strong for working in the space of
    probability measures with first moment bounded. More precisely,
    the condition on $k$ in the $(L^\infty_k \cap \dot W^{1,\infty})'$
    decay estimate is no better than the one discussed above for
    $p=1$, i.e. $k > 3/2$, and $k=1$ is required for this dual norm to
    be equivalent to the $W_1$ distance. 

  %Theorem~\ref{theo:KFPrdhyp1} extends the exponenial decay estimate to Lebesgue spaces with polynomial decay at least in the case $\beta=2$.  
%  \item  Such a semigroup spectral gap result for the  kinetic Fokker-Planck equation 
% (in the whole space with a confining potential) 
%  has been proved  in \cite{MR2034753} and \cite{MR2562709} in the Sobolev space 
%  $H^\sigma(\mu^{-1/2})$, $\sigma \in \N^*$, and in the recent works
%  \cite{DMScras,DMS} in the Lebesgue space $L^2(\mu^{-1})$. 
 
  \end{enumerate}
\end{rems}

The strategy of the proof follows the same structure as in the
previous section, and we start from the following $H^1$ spectral gap
estimate that has been established in \cite{MR2562709} for potentials
$\Psi$ under our assumptions, with constructive proof.  See also
\cite{MR1787105,MR2034753,MR2130405} for previous results in that
direction.
\begin{theo} {\bf (\cite[Theorem
    35]{MR2562709})} \label{lem:KFPpotential-gapH1} The result in
  Theorem~\ref{theo:KFPrdhyp1} is true in the Hilbert space
  $H^1(\mu^{-1/2})$, and satisfy quantitative hypodissipativity
  estimate for the equivalent norm
$$
\left( \|f \|_{L^2( \mu^{-1/2})}^2 + a \|\nabla_x f
  \|_{L^2( \mu^{-1/2})}^2 + b\|\nabla_v f \|_{L^2( \mu^{-1/2})}^2
 + 2c \langle \nabla_x f, \nabla_v f \rangle_{L^2( \mu^{-1/2})} \right)^{1/2}
$$
for appropriate choice of $a,b,c>0$ with $c < \sqrt{ab}$. 
\end{theo}

\subsection{Dissipativity property of $\BB$}
\label{subsec:FPKrdB}

We define $\AA$ and $\BB$ as follows:
\begin{equation}\label{eqKFPrddefB}
\AA f := M \chi_R f, \qquad \BB f :=  \CC f + \TT  f - M \chi_R f
\end{equation}
where $M > 0$, $\chi_R(x,v) = \chi(H(x,v)/R)$, $R > 1$, and $0 \le
\chi \in C^\infty_c(\R^d \times \R^d)$ is such that $\chi (x,v) = 1$
for any $|x|^2 + |v|^2 \le 1$.

We start with Lebesgue spaces:
\begin{lem}\label{lem:FPKrd-BdissipLp} 
We have:
\begin{itemize}
\item \emph{(Polynomial weights)} For any $\beta \ge 2$,
  $p \in [1,+\infty]$ and $k > k(d,p)$ for some $k(d,p)>d/p'$ from the
  proof, there is $a <0$ such that the operator $\BB-a$ is dissipative
  in the space $L^p(H^{k})$.

\item \emph{(Exponential weights)} For any $\beta \ge 1$,
  $p \in [1,\infty]$ and $s \in (0,1]$ (with the extra condition
  $\kappa <1$ in the case $s=1$), there is $a <0$ such that the
  operator $\BB-a$ is dissipative in the space $L^p(e^{\kappa H^s})$.
\end{itemize}
\end{lem}

\begin{proof}[Proof of Lemma~\ref{lem:FPKrd-BdissipLp}.] 
  The proofs in the two cases will be similar: we give full details
  for the first case and less for the second case.  \mk

  \noindent
  \emph{Step 1: Polynomial weight.} Let us first consider $\beta \ge 
  2$ and $m(x,v) = H^k$, and the following weight multiplier:
\begin{equation*}
  W(x,v) := m \, w, \quad w := \left( 1 + \frac12 \frac{x \cdot v}{H_\alpha} 
  \right), \quad H_\alpha := 1+ \alpha \frac{\langle x
    \rangle^\beta}{\beta} + \frac{1}{\alpha} \frac{|v|^2}{2}.
\end{equation*}
Observe that $(x \cdot v) \le H_\alpha$ by Young's inequality (for any
$\alpha >0$ and $\beta \ge 2$), which proves that $w \in [1/2,3/2]$
and $(1/2)m \le W \le (3/2) m$.  We then consider a solution to the
equation $\partial_t f = \mathcal B f$ and compute
\begin{multline*}
  \frac1p \dt \int_{\R^d \times \R^d} f^p W^p \dd x \dd v = \int_{\R^d \times
    \R^d} f^{p-1} \mathcal C f W^p \dd x \dd v \\ + \int_{\R^d \times
    \R^d} f^{p-1} \mathcal T f W^p \dd x \dd v - M \int_{\R^d \times
    \R^d} f^{p} W^p \, \chi_R \dd x \dd v  . 
\end{multline*}
On the one hand, we recall the following computation picked up from
the proof of Lemma~\ref{lem:FPhomo-BLp} 
\begin{align*}
  & \int_{\R^d \times \R^d} f^{p-1} \mathcal C f W^p \dd x \dd v :=
  \int_{\R^d \times \R^d} f \, |f|^{p-2} \, (\CC f) \, W^p \dd x \dd v
  \\
 &  = - (p-1) \int_{\R^d \times \R^d} |\nabla_v (Wf)|^2 \, | Wf |^{p-2}
  \dd x \dd v \\
  & + \int_{\R^d \times \R^d} |f|^p \, W^p \, \left[
    \frac{2}{p'} {|\nabla_v
      W|^2 \over W^2} + \left( \frac{2}{p} - 1 \right)
    \frac{\Delta_v W}{W} + {d \over p'} - {v \cdot \nabla_v W \over W}
  \right] \dd x \dd v
\end{align*}
where $p'=p/(p-1)$, and we compute 
\begin{align*}
  & \frac{\nabla_v m}{m} = \frac{k v}{H}, \quad \frac{\Delta_v m}{m}
  = \frac{kd}{H} + \frac{k(k-1) |v|^2}{H^2}, \\
  % & \frac{|\nabla_v w|^2}{w^2} \la \var^2, \quad
  % \left| \frac{\Delta_v w}{w} \right| \la \var, \\
  & \frac{\nabla_v W}{W} = \frac{\nabla_v m}{m} + \frac{\nabla_v
    w}{w}, \quad \frac{v \cdot \nabla_v W}{W} = \frac{v \cdot \nabla_v
    m}{m} + \frac{v \cdot \nabla_v w}{w}, \\
  & \frac{|\nabla_v W|^2}{W^2} \le 2 \frac{|\nabla_v m|^2}{m^2}+ 2
  \frac{|\nabla_v w|^2}{w^2}.
\end{align*}
We deduce 
% \begin{align*}
%   \left[ {|\nabla_v m|^2 \over m^2} + {d \over p'} - v \cdot {\nabla_v m \over m}
%   \right] &\le \frac{d}{p'} + \frac{1}{H} \left[ -k |v|^2 + C_1 k^2 \langle v \rangle^2
%     \frac{1}{H} \right] \\
%  \left[ {|\nabla_v w|^2 \over w^2} + {d \over p'} - v \cdot {\nabla_v w \over w}
%   \right] &\le C_0 \quad \mbox{and therefore} \\ 
% \left[ {|\nabla_v W|^2 \over W^2} + {d \over p'} - v \cdot {\nabla_v W \over W}
%   \right] &\le \frac{d}{p'} + \var C_0 + \frac{1}{H} \left[ -k |v|^2 + C_1 k^2 \langle v \rangle^2
%     \frac{1}{H} \right]
% \end{align*}
\begin{align*}
  & \left[ \frac{2}{p'} {|\nabla_v W|^2 \over W^2} + \left( \frac{2}{p}
      - 1 \right) \frac{\Delta_v W}{W} + {d \over p'} - {v \cdot
      \nabla_v W \over W} \right] \\ 
  & \le \left[ \frac{4}{p'} {|\nabla_v m|^2 \over m^2} + \left( \frac{2}{p}
      - 1 \right) \frac{\Delta_v m}{m} + {d \over p'} - {v \cdot
      \nabla_v m \over m} \right] \\
& + \left[ \frac{4}{p'} {|\nabla_v w|^2 \over w^2} + \left( \frac{2}{p}
      - 1 \right) \frac{\Delta_v w}{w} - {v \cdot
      \nabla_v w \over w} \right] + 2 \left( \frac2p - 1 \right)
  \frac{\nabla_v m \cdot \nabla_v w}{m w}\\
  & \le \frac{d}{p'} + \frac{C_1}{\sqrt{H}} - \frac12 \frac{(x \cdot
    v)}{H_\alpha} + \frac{1}{2 \alpha} \frac{(x \cdot v)|v|^2}{H_\alpha
    ^2} - k \frac{|v|^2}{H}
\end{align*}
for some constant $C_1>0$. The RHS is not negative at infinity, where
infinity means $H >>1$. This explains the need for the additional
correction term $w$ in $W$. We compute
\begin{multline*}
  \int_{\R^d \times \R^d} f^{p-1} \mathcal T f W^p \dd x \dd v:=
  \frac1p \int_{\R^d \times \R^d} \mathcal T (f^p) W^p \dd x \dd v \\ := -
  \int_{\R^d \times \R^d} f^p W^{p-1} \mathcal T W \dd x \dd v 
  := - \int_{\R^d \times \R^d} f^p W^{p} \frac{\mathcal T  w}{w} \dd x
  \dd v 
\end{multline*}
where we have used $\mathcal T H = 0$. We have then
\begin{equation*}
   -  \frac{\mathcal T  w}{w} \le C_2  + \frac{C_3 }{\alpha^2} \frac{|v|^2}{H}  -
   \frac{1}{4} \frac{\langle x
     \rangle^{\beta}}{H_\alpha}  % \quad \mbox{and thus}
  % % \quad 
  % - \frac{\mathcal T w}{w} \le - \frac{\var}{2} \frac{\langle x
  %   \rangle^{\beta(1+1/q)-3} |x|^2}{\langle x
  %   \rangle^{2\beta/q} + \langle v \rangle^2} + \var C' \frac{\langle
  %   x \rangle^{\beta/q-1} |v|^2}{\langle x
  %   \rangle^{2\beta/q} + \langle v \rangle^2}
\end{equation*}
by differentiating and using Young inequality and the form of the
potential $\Psi(x)$ at infinity, for some constants $C_2,C_3>0$.  We
deduce by taking $R,M$ large enough that for any $\eta>0$ as small as
wanted
\begin{align*}
  & \frac1p \dt \int_{\R^d \times \R^d} f^p W^p \dd x \dd v \\ 
  & \le \int_{\R^d \times
    \R^d} f^p W^p \left[ C(d,p) +\left( \frac{C_3}{\alpha^2} - k\right) \frac{|v|^2}{H} -
    \frac{1}{4} \frac{\langle x
    \rangle^{\beta}}{H_\alpha} - M \chi_R \right] \dd x \dd v 
\end{align*}
where we have used $w \approx 1$. % Provided that $d/p'<1/4$ we can
% choose $R,M$ large enough so that $\eta$ is small enough so that
% $d/p' +\eta<1/4$, and then deduce easily that the bracket term above
% satisfies:
% \begin{align*}
%   \left[ \frac{d}{p'} + \eta - k \frac{|v|^2}{H} - \frac14
%     \frac{\langle x \rangle^{\beta-2} |x|^2}{H} - M \chi_R \right] \le
%   -\lambda H, \quad \mbox{with } \lambda:= \left( \frac14 -
%     \frac{d}{p'} - \eta \right) >0.
% \end{align*}
Finally we restrict to $H \ge R$ with $R$ large enough, and observe
that we have either $|v|^2/2 \ge H/3$ or $|v|^2/2 \le H/3$, together
with $|v|^2/2 + \langle x \rangle^\beta/\beta \ge 2H/3$. In the first
case
\begin{equation*}
  C(d,p) +\left( \frac{C_3}{\alpha^2} - k\right) \frac{|v|^2}{H} \le
  C(d,p) +\frac23\left( \frac{C_3}{\alpha^2} - k\right) \le - \frac{k}{2}
\end{equation*}
for $k$ large enough. In the second case we only have
$( C_3\alpha^{-2} - k) |v|^2/H \le 0$ but we can use the second
negative term since now $\langle x \rangle^\beta/\beta \ge H/3$ and
$H_\alpha^{-1} \ge 1/(\alpha H)$:
\begin{equation*}
  C(d,p) -  \frac{1}{4} \frac{\langle x
    \rangle^{\beta}}{H_\alpha} \le C(d,p) -
  \frac{\beta}{12\alpha} \le -
  \frac{\beta}{24\alpha} 
\end{equation*}
for $\alpha$ small enough. All in all we deduce finally, for $k$ large
enough and $\alpha$ small enough (depending on $p$ and $d$) and $M$
large enough:
\begin{align*}
   \frac1p \dt \int_{\R^d \times \R^d} f^p W^p \dd x \dd v  
   \le - K \int_{\R^d \times
    \R^d} f^p W^p \dd x \dd v 
\end{align*}
for some constant $K>0$, which concludes the proof.
\medskip

  \noindent
  \emph{Step 2: (Stretched) exponential weight.} Let us now consider
  $\beta \ge 1$, $s \in (0,1]$, $\kappa >0$ (with $\kappa <1$ in the
  case $s=1$) and $m(x,v) = e^{\kappa H^s}$, and the corrected weight:
\begin{equation*}
  W(x,v) := m \, w, \quad w := \left( 1 + \frac12 \frac{x \cdot v}{H_\alpha} 
  \right), \quad H_\alpha := 1+ \alpha \frac{\langle x
    \rangle^\beta}{\beta} + \frac{1}{\alpha} \frac{|v|^2}{2}
\end{equation*}
which satisfies again $W \approx m$.  We compute as before
\begin{align*}
  \int_{\R^d \times \R^d} f^{p-1} \mathcal C f W^p \dd x \dd v
  &= - (p-1) \int_{\R^d \times \R^d} |\nabla_v (Wf)|^2 \, | Wf |^{p-2}
  \dd x \dd v \\
  & \qquad + \int_{\R^d \times \R^d} |f|^p \, W^p \, \left[ {|\nabla_v
      W|^2 \over W^2} + {d \over p'} - v \cdot {\nabla_v W \over m}
  \right] \dd x \dd v
\end{align*}
where $p'=p/(p-1)$, and we compute again
\begin{align*}
  & \left[ \frac{2}{p'} {|\nabla_v W|^2 \over W^2} + \left( \frac{2}{p}
      - 1 \right) \frac{\Delta_v W}{W} + {d \over p'} - {v \cdot
      \nabla_v W \over W} \right] \\ 
  & \le \left[ \frac{4}{p'} {|\nabla_v m|^2 \over m^2} + \left( \frac{2}{p}
      - 1 \right) \frac{\Delta_v m}{m} + {d \over p'} - {v \cdot
      \nabla_v m \over m} \right] \\
& + \left[ \frac{4}{p'} {|\nabla_v w|^2 \over w^2} + \left( \frac{2}{p}
      - 1 \right) \frac{\Delta_v w}{w}  - {v \cdot
      \nabla_v w \over w} \right] + 2 \left( \frac2p - 1 \right)
  \frac{\nabla_v m \cdot \nabla_v w}{m w}
\end{align*}
with
\begin{align*}
  & \left[ \frac{4}{p'} {|\nabla_v m|^2 \over m^2} + \left( \frac{2}{p}
    - 1 \right) \frac{\Delta_v m}{m} + {d \over p'} - {v \cdot \nabla_v m \over m}
    \right] \le \frac{d}{p'} + \frac{1}{H} \left[ -\kappa s |v|^2 H^s + \kappa^2
    s^2 |v|^2 H^{2s-1} \right] \\
  & \left[ \frac{4}{p'}  {|\nabla_v w|^2 \over w^2} + \left( \frac{2}{p}
    - 1 \right) \frac{\Delta_v w}{w} - {v \cdot \nabla_v w \over w}
    \right]  + 2 \left( \frac2p - 1 \right)
    \frac{\nabla_v m \cdot \nabla_v w}{m w} \\ 
  & \hspace{8cm} \le 
    \frac{C_1}{\sqrt{H}} - \frac12 \frac{(x \cdot
    v)}{H_\alpha} + \frac{1}{2 \alpha} \frac{(x \cdot v)|v|^2}{H_\alpha
    ^2}
\end{align*}
for some constant $C_1>0$. 
Since $s \in (0,1)$ we have $2s-1 <s$ and the term $\kappa^2 s^2 |v|^2
H^{2s-1}$ is dominated by the previous negative term for $M$ large
enough.  Then we have 
\begin{equation*}
   -  \frac{\mathcal T  w}{w} \le C_2 - C_3 \frac{\langle x \rangle^{\beta-2}|x|^2}{H_\alpha} % \quad \mbox{and thus}
\end{equation*}
for two other constants $C_2, C_3>0$.  We deduce that
\begin{align*}
  & \frac1p \dt \int_{\R^d \times \R^d} f^p W^p \dd x \dd v \\ 
& \le \int_{\R^d \times
    \R^d} f^p W^p \left[ C(d,p) - \kappa s \frac{|v|^2
      H^s}{H} -
    C_3 \frac{\langle x
      \rangle^{\beta-2} |x|^2}{H_\alpha}  - M \chi_R \right] \dd x \dd v
\end{align*}
for some constant $C(d,p)>0$. 
Finally we restrict to $H \ge R$ with $R$ large enough, and observe
that we have either $|v|^2/2 \ge H/3$ or $|v|^2/2 \le H/3$, together
with $|v|^2/2 + \langle x \rangle^\beta/\beta \ge 2H/3$. In the first
case
\begin{equation*}
  C(d,p) - \kappa s \frac{|v|^2 H^s}{H} \le
  C(d,p) - \kappa s R^s \frac{|v|^2}{H} \le C(d,p) - \frac{\kappa s
    R^s}{3} \le - \frac{\kappa s R^s}{6} 
\end{equation*}
for $R$ large enough. In the second case we use the second negative
term since now $\langle x \rangle^\beta/\beta \ge H/3$ and
$H_\alpha^{-1} \ge 1/(\alpha H)$ and $|x| \ge (H/3)^{1/\beta} \ge
(R/3)^{1/\beta}$ is non-zero:
\begin{equation*}
  C(d,p) -  C_3 \frac{\langle x
    \rangle^{\beta-2} |x|^2}{H_\alpha}  \le  C(d,p) - \frac{C_3}{2} \frac{\langle x
    \rangle^{\beta}}{H_\alpha}  \le C(d,p) - \frac{C_3  \beta}{6
    \alpha} \le - \frac{C_3  \beta}{12 \alpha} 
\end{equation*}
for $\alpha$ small enough. All in all we deduce finally, for $R$ large
enough and $\alpha$ small enough (depending on $p$ and $d$) and $M$
large enough:
\begin{align*}
   \frac1p \dt \int_{\R^d \times \R^d} f^p W^p \dd x \dd v  
   \le - K \int_{\R^d \times
    \R^d} f^p W^p \dd x \dd v 
\end{align*}
for some constant $K>0$, which concludes the proof.
% Consider first the case where $\langle x \rangle^\beta \le
% |v|^2/2$. Then $\langle x \rangle^{\beta/2-1} \langle v \rangle \le
% \langle v \rangle^{2-2/\beta}$ and for $R$ large enough it is
% destroyed by the the term $- \kappa s|v|^2 H^s$. Moreover since $-
% \kappa s |v|^2 H^{s-1}$ goes to minus infinity as $R$ goes to
% infinity we conclude that, for $R,M$ large enough, the bracket term is
% uniformly strictly negative.
% Consider now the case where $\langle x \rangle^\beta \ge |v|^2/2$. It
% implies that $x$ goes to infinity as $R$ goes to infinity.  Then
% observe that $\langle x \rangle^{\beta/2-1} \langle v \rangle \la
% \langle x \rangle^{3\beta/2-3} \langle v \rangle^{-1} + \langle v
% \rangle^2$ by Young's inequality. Therefore we get
% \begin{align*}
%   & \frac1p \dt \int_{\R^d \times \R^d} f^p W^p \dd x \dd v \\
%   & \le \int_{\R^d \times \R^d} f^p W^p \left[ \frac{d}{p'} + \eta -
%     \kappa s \frac{|v|^2 H^s- C' \langle v \rangle^2}{H} - \frac14 \frac{\langle x
%       \rangle^{3\beta/2-3} (|x|^2-C')\langle v \rangle^{-1} }{H} - M
%     \chi_R \right] \dd x \dd v.
% \end{align*}
% Now consider the subcase $\langle v \rangle \le \langle x
% \rangle^{\beta/(2(1+s))}$. And thus for $R$ large enough
% \begin{align*}
%   \frac14 \frac{\langle x \rangle^{3\beta/2-3} ( |x|^2-C') \langle v
%     \rangle^{-1} }{H} \ge \frac18 \frac{\langle x \rangle^{3\beta/2-
%       \beta/(2(1+s))-1}}{H} \ga \langle x \rangle^{\frac{\beta}{2}
%     \frac{s}{1+s} -1}
% \end{align*}
% which goes to infinity for $\beta s /(2(1+s)) >1$ as $R$ goes to
% infinity, which concludes the proof. 
\end{proof}

Then, by arguing exactly similarly as in Section~\ref{sec:FPhomo} and
using the previous calculations for the differentiated equation and
the adjoint operators, we obtain the following lemma. We omit the
proof in order not the repeat closely related technical
estimates. 
\begin{lem}\label{lem:FPKrd-BdissipWp} 
We have:
\begin{itemize}
\item \emph{(Polynomial weights)} For any $\beta \ge 2$,
  $p \in [1,+\infty]$ and $k > k(d,p)$ for some $k(d,p)>d/p'$ from the
  proof, there is $a <0$ such that the operator $\BB-a$ is dissipative
  in the spaces $W^{1,p}(H^{k})$ and $W^{-1,p}(H^k)$.

\item \emph{(Exponential weights)} For any $\beta \ge 1$,
  $p \in [1,\infty]$ and $s \in (0,1]$ (with the extra condition
  $\kappa <1$ in the case $s=1$), there is $a <0$ such that the
  operator $\BB-a$ is dissipative in the spaces
  $W^{1,p}(e^{\kappa H^s})$ and $W^{-1,p}(e^{\kappa H^s})$.
\end{itemize}
\end{lem}
 
\begin{rem}
  Observe that the previous lemma implies the hypodissipativity of $B$
  in $H^1(\mu^{-1/2})$ (as needed in the application of our abstract
  theorem to this Fokker-Planck equation with confinement). This
  result could also be obtained easily by slightly modifying the proof
  of \cite[Theorem 35]{MR2562709}. It is also possible to deduce this
  hypodissipativity from that of $L$ together with estimates
  quantifying the gain of decay at infinity in $x$ and $v$. Since we
  could prove the latter estimates using the ideas developed in this
  paper, and they seem of independent interest and not available in
  the literature, we include them in a short appendix. 
\end{rem}

%%%%%---------------------------%%%%%---------------------------%%%%%---------------------------
%%%%%---------------------------%%%%%---------------------------%%%%%---------------------------

 \subsection{Regularisation estimates}
\label{subsec:FPKrdB}

We prove a regularization property of the kinetic Fokker-Planck
equation with a confining potential. It is again related to the theory
of hypoellipticity, but is slightly less well-known due to the use of
weighted norms defined in the whole space.  The argument follows
the same method as before. 

\begin{lem}\label{lem:reg-kfp-pot}
  The semigroup $\mathscr S_{\BB}$ satisfies similar inequalities as
  in Lemma~\ref{lem:reg-kfp}, where now $\mu = e^{-H}$ is the
  $(x,v)$-dependent equilibrium, and $\ell$ is a number large enough
  (note the additional $H^\ell$ weight in $L^2$ and $L^1$ norms)
\[
{\bf (1)} \quad \forall \, t \in [0,1], \ \forall \, k \in \N^* \quad 
\left\{ 
\begin{array}{l}\displaystyle
\| \mathscr S_{\BB}(t) f \|_{H^k(\mu^{-1/2})} \lesssim
\frac{1}{t^{3k/2}} \, \| f\|_{L^2(H^{\ell/2} \mu^{-1/2})},
\vspace{0.3cm} \\ \ds  
\| \mathscr S_{\BB}(t) f \|_{L^2(H^{-\ell/2} \mu^{-1/2})} \lesssim \frac{1}{t^{3k/2}} \, \| f\|_{H^{-k}(\mu^{-1/2})}.
\end{array}
\right.
\]
second (gain of integrability at order zero)
\[
{\bf (2)} \quad \forall \, t \in [0,1], \quad 
\left\{ 
\begin{array}{l}\displaystyle
\| \mathscr S_{\BB}(t) f \|_{L^2(\mu^{-1/2})} \lesssim \frac{1}{t^{(5d+1)/2}} \, \|
f\|_{L^1(H^\ell \mu^{-1/2})}, \vspace{0.3cm} \\ \ds
\| \mathscr S_{\BB}(t) f \|_{L^\infty(H^{-\ell} \mu^{-1/2})} \lesssim \frac{1}{t^{(5d+1)/2}} \, \|
f\|_{L^2(\mu^{-1/2})}
\end{array}
\right.
\]
third (gain of integrability at order one)
\[
{\bf (3)} \quad \forall \, t \in [0,1], \quad 
\left\{ 
\begin{array}{l}\displaystyle
\| \nabla \mathscr S_{\BB}(t) f \|_{L^2(\mu^{-1/2})} \lesssim \frac{1}{t^{(5d+1)/2}} \, \|
\nabla f\|_{L^1(H^\ell \mu^{-1/2})}, \vspace{0.3cm} \\ \ds
\| \nabla \mathscr S_{\BB}(t) f \|_{L^\infty(H^{-\ell} \mu^{-1/2})} \lesssim \frac{1}{t^{(5d+1)/2}} \, \|
\nabla f\|_{L^2(\mu^{-1/2})}
\end{array}
\right.
\]
fourth (gain of integrability at ordre minus one)
\[
{\bf (4)} \quad \forall \, t \in [0,1], \quad 
\left\{ 
\begin{array}{l}\displaystyle
  \| \mathscr S_{\BB}(t) f \|_{W^{-1,\infty}(H^{-\ell} \mu^{-1/2})} \lesssim \frac{1}{t^{(5d+1)/2}} \, \|
  f\|_{W^{-1,2}(\mu^{-1/2})}, \vspace{0.3cm} \\ \ds
  \|  \mathscr S_{\BB}(t) f \|_{W^{-1,2}(\mu^{-1/2})} \lesssim \frac{1}{t^{(5d+1)/2}} \, \|
  f\|_{W^{-1,1}(H^\ell\mu^{-1/2})}.
\end{array}
\right.
\]
\end{lem}
% \textcolor{red}{CM : probl\`eme avec le point (4) ci-dessus comme dans
%   la section pr\'ec\'edente}

\begin{proof}[Proof of Lemma~\ref{lem:reg-kfp-pot}.] The proof is
  similar to that of Lemma~\ref{lem:reg-kfp}. The only technical change
  is an additional weight in energy functional for proving the point
  {\bf (1)}; it reads for instance in the case $k=1$:
\begin{multline*}
  \FF(t,f_t) := A \| f_t \|_{L^2(H^{\ell/2} \mu^{-1/2})}^2 + a t \| \nabla_{v} f_t  \|_{L^2(\mu^{-1/2})}^2 
   \\  + 2 c t^{2} \langle \nabla_{v} f_t,  \nabla_{x} f_t
  \rangle_{L^2(\mu^{-1/2})} +  b t^3 \| \nabla_{x} f_t  \|_{L^2(\mu^{-1/2})}^2
\end{multline*}
with $a,b,c >0$, $c < \sqrt{ab}$ (positive definite) and $A>0$ and
$\ell \in \N$ large enough (observe the weight $H^\ell$ on the $L^2$
part of the norm).
\end{proof}
 
The proof of the decay estimate on $T_n(t)$ and the completion of the
proof of Theorem~\ref{theo:KFPrdhyp1} are then done as in
Corollary~\ref{cor:Tn:FPtorus} and Theorem~\ref{theo:vFP}. 

\appendix

\section{Quantitative compactness estimates on the resolvent}
\label{sec:quant-comp-estim}

In this appendix we amplify the ideas of this article in order to give
quantitative estimates of compactness on the resolvent of the kinetic
Fokker-Planck considered. More precisely: One way to understand the
compactness of resolvent is to split it into a local gain of
\emph{regularity} and a gain of \emph{decay at infinity}, and we focus
here on the gain of decay at infinity. The gain of regularity can then
be recovered by local hypoelliptic estimates along the theory of
H\"ormander. Note that another route for deriving estimates on the
gain of decay at infinity is to use the \emph{global} hypoellipticity
estimates as in \cite{Herau2007} and \cite[Section~A.21]{MR2562709}
with Gaussian weight and \emph{deduce} the gain of decay at infinity
by applying some forms of ``strengthened'' Poincar\'e inequality;
however the fractional derivatives involved would likely create
technical difficulties, whereas our estimates based on weight
multiplicators is elementary. Our estimates also do not require
regularity on the solution. 

Let us first say a word on the case of the periodic confinement. In
this case it is enough to use the strengthened Poincar\'e inequality
in velocity only: 
\begin{align*}
  & \dt \frac12 \int_{\T^d \times \R^d} f^2 \mu^{-1} \dd x \dd v \le - \int_{\T^d
  \times \R^d} \left| \nabla_v \left( \frac{f}{\mu} \right) \right|
  \mu \dd x \dd v \\ & \le - K \int_{\T^d \times \R^d} f^2 (1+|v|^2) \mu^{-1} + C
  \int_{\T^d \times \R^d} f^2 \mu^{-1} \dd x \dd v 
\end{align*}
for some constants $C,K>0$, and therefore we deduce
\begin{align*}
  - \langle \LL f, f \rangle_{L^2(\mu^{-1})} \ge K \int_{\T^d \times \R^d} f^2 (1+|v|^2) \mu^{-1} \dd x \dd v - C
  \int_{\T^d \times \R^d} f^2 \mu^{-1} \dd x \dd v 
\end{align*}
and finally 
\begin{align*}
  \int_{\T^d \times \R^d} f^2 (1+|v|^2) \mu^{-1} \dd x \dd v \la \n{\LL
  f}_{L^2(\mu^{-1})} ^2 + \n{f}_{L^2(\mu^{-1})} ^2 
\end{align*}
which gives the gain of decay at infinity on the resolvent. Combined
with hypocoercivity estimates that provide bounds
$\n{f}_{L^2(\mu^{-1})} \la \n{\LL f- \xi}_{L^2(\mu^{-1})}$ for certain
$\xi \in \C$, this allows to control $\int f^2 (1+|v|^2) \mu^{-1}$.

Let us now turn to the more interesting case of the potential
confinement. We now differentiate the following norm
\begin{align*}
  \int_{\R^d \times \R^d} f^2 W \mu^{-1} \dd x \dd v, \quad W(x,v):= \left[ a |x|^{\beta/3} + b
  |v|^2 + 2 c |x|^{\beta/6-1} (x \cdot v) \right], 
\end{align*}
for some appropriate choice of $a,b,c >0$ so that $c <
\sqrt{ab}$. Then
\begin{align*}
  & \dt \frac12 \int_{\T^d \times \R^d} f^2 |x|^{\beta/3} \mu^{-1} \dd
    x \dd v \\ 
  & \qquad \qquad \le - \int_{\T^d
    \times \R^d} \left| \nabla_v \left( \frac{f}{\mu} \right) \right| |x|^{\beta/3}
    \mu \dd x \dd v + C_0 \int_{\T^d \times \R^d} f^2 |x|^{-\beta/3} |v|
    \mu^{-1} \dd x \dd v \\ 
  & \dt \frac12 \int_{\T^d \times \R^d} f^2 |v|^{2} \mu^{-1} \dd
    x \dd v \\ 
  & \qquad \qquad \le - \int_{\T^d
    \times \R^d} \left| \nabla_v \left( \frac{f}{\mu} \right) \right| |v|^{2}
    \mu \dd x \dd v + C_0 \int_{\T^d \times \R^d} f^2 (1+ |v|^2)
    \mu^{-1} \dd x \dd v \\
  & \dt \frac12 \int_{\T^d \times \R^d} f^2 |x|^{\beta/6-1}(x \cdot v) \mu^{-1} \dd
    x \dd v \\ 
  & \qquad \qquad \le - \int_{\T^d
    \times \R^d} f^2 |x|^{2\beta/3} \mu \dd x \dd v + C_0 \int_{\T^d \times \R^d} f^2 \left(1+|x|^{-\beta/3}|v|^2\right)
    \mu^{-1} \dd x \dd v \\ 
  & \qquad \qquad + C_0 \int_{\T^d
    \times \R^d} \left| \nabla_v \left( \frac{f}{\mu} \right) \right|
    |x|^{\beta/6} |v|
    \mu \dd x \dd v 
\end{align*}
for some constant $C_0>0$, which implies by using Young's inequality
and adjusting the constants $a,b,c>0$ that
\begin{align*}
  \dt \int_{\R^d \times \R^d} f^2 W \mu^{-1} \dd x \dd v =
  \int_{\R^d \times \R^d} (\LL f) f W \mu^{-1} \dd  x \dd v \\
  \le - K
  \int_{\R^d \times \R^d} f^2 W^2 \mu^{-1} \dd x \dd v + C \int_{\R^d \times \R^d} f^2 \mu^{-1} \dd x \dd v
\end{align*}
for some constants $C,K>0$, and finally 
\begin{align*}
  \int_{\R^d \times \R^d} f^2 W^2 \mu^{-1} \dd x \dd v \la \n{\LL
  f}_{L^2(\mu^{-1})} ^2 + \n{f}_{L^2(\mu^{-1})} ^2,
\end{align*}
which is again a quantitative estimate of gain of decay at infinity
for the resolvent.

% {\Green ENLEVER TOUT CELA 
% \subsection{Open questions. }

% \smallskip
% (1)   Prove the  hypodissipativity of $\BB$ in any Sobolev spaces $H^\sigma(m)$, $\sigma \in \Z$, by combining Villani's result 
% of \cite[section A.21.3]{MR2562709} and a duality argument. 

% \smallskip
% (2) Are we able to prove the open problem in  \cite[Remark A.17]{MR2562709}? 

% \textcolor{red}{cm~: to discuss in the end}
% }
% \smallskip
% 
%% $\bullet$ Make the dissipativity estimate uniformly in $L^p$, $1 \le p < \infty$, so that we can pass to the limit $p\to\infty$.
%% We believe that there is no difficulty in doing this. 
% 
% $\bullet$ Prove the dissipativity estimate in $W^{-1,p}$ thanks to a duality argument. 
% $\gamma=2$ ok? 
% 
% $\bullet$ $H^{\pm1}(m)$ ok ? 

\bibliographystyle{acm}
%\bibliography{biblio,mybiblio}
\bibliography{./spectreFP}

\begin{thebibliography}{10}

\bibitem{MR946973}
{\sc Arkeryd, L.}
\newblock Stability in {$L^1$} for the spatially homogeneous {B}oltzmann
  equation.
\newblock {\em Arch. Rational Mech. Anal. 103}, 2 (1988), 151--167.

\bibitem{BBCG}
{\sc Bakry, D., Barthe, F., Cattiaux, P., and Guillin, A.}
\newblock A simple proof of the {P}oincar\'e inequality for a large class of
  probability measures including the log-concave case.
\newblock {\em Electron. Commun. Probab. 13\/} (2008), 60--66.

\bibitem{MR2381160}
{\sc Bakry, D., Cattiaux, P., and Guillin, A.}
\newblock Rate of convergence for ergodic continuous {M}arkov processes:
  {L}yapunov versus {P}oincar\'e.
\newblock {\em J. Funct. Anal. 254}, 3 (2008), 727--759.

\bibitem{MR2727993}
{\sc Bartier, J.-P., Blanchet, A., Dolbeault, J., and Escobedo, M.}
\newblock Improved intermediate asymptotics for the heat equation.
\newblock {\em Appl. Math. Lett. 24}, 1 (2011), 76--81.

\bibitem{BGG}
{\sc Bolley, F., Gentil, I., and Guillin, A.}
\newblock Convergence to equilibrium in {W}asserstein distance for
  {F}okker--{P}lanck equations.
\newblock {\em J. Funct. Anal. 263}, 8 (2012), 2430--2457.

\bibitem{MR2661206}
{\sc Cattiaux, P., Guillin, A., and Roberto, C.}
\newblock Poincar\'e inequality and the {$L^p$} convergence of semi-groups.
\newblock {\em Electron. Commun. Probab. 15\/} (2010), 270--280.

\bibitem{MR1787105}
{\sc Desvillettes, L., and Villani, C.}
\newblock On the trend to global equilibrium in spatially inhomogeneous
  entropy-dissipating systems: the linear {F}okker-{P}lanck equation.
\newblock {\em Comm. Pure Appl. Math. 54}, 1 (2001), 1--42.

\bibitem{DMScras}
{\sc Dolbeault, J., Mouhot, C., and Schmeiser, C.}
\newblock Hypocoercivity for kinetic equations with linear relaxation terms.
\newblock {\em C. R. Math. Acad. Sci. Paris 347}, 9-10 (2009), 511--516.

\bibitem{DMS}
{\sc Dolbeault, J., Mouhot, C., and Schmeiser, C.}
\newblock Hypocoercivity for linear kinetic equations conserving mass.
\newblock {\em Trans. Amer. Math. Soc. 367}, 6 (2015), 3807--3828.

\bibitem{GMM}
{\sc Gualdani, M.~P., Mischler, S., and Mouhot, C.}
\newblock Factorization of non-symmetric operators and exponential
  ${H}$-{T}heorem.
\newblock Preprint arXiv:1006.5523.

\bibitem{MR2917400}
{\sc Guillin, A., and Wang, F.-Y.}
\newblock Degenerate {F}okker-{P}lanck equations: {B}ismut formula, gradient
  estimate and {H}arnack inequality.
\newblock {\em J. Differential Equations 253}, 1 (2012), 20--40.

\bibitem{MR3262508}
{\sc Hairer, M., Stuart, A.~M., and Vollmer, S.~J.}
\newblock Spectral gaps for a {M}etropolis-{H}astings algorithm in infinite
  dimensions.
\newblock {\em Ann. Appl. Probab. 24}, 6 (2014), 2455--2490.

\bibitem{MR2130405}
{\sc Helffer, B., and Nier, F.}
\newblock {\em Hypoelliptic estimates and spectral theory for {F}okker-{P}lanck
  operators and {W}itten {L}aplacians}, vol.~1862 of {\em Lecture Notes in
  Mathematics}.
\newblock Springer-Verlag, Berlin, 2005.

\bibitem{Herau_Boltz}
{\sc H{\'e}rau, F.}
\newblock Hypocoercivity and exponential time decay for the linear
  inhomogeneous relaxation {B}oltzmann equation.
\newblock {\em Asymptot. Anal. 46}, 3-4 (2006), 349--359.

\bibitem{Herau2007}
{\sc H{\'e}rau, F.}
\newblock Short and long time behavior of the {F}okker-{P}lanck equation in a
  confining potential and applications.
\newblock {\em J. Funct. Anal. 244}, 1 (2007), 95--118.

\bibitem{MR2034753}
{\sc H{\'e}rau, F., and Nier, F.}
\newblock Isotropic hypoellipticity and trend to equilibrium for the
  {F}okker-{P}lanck equation with a high-degree potential.
\newblock {\em Arch. Ration. Mech. Anal. 171}, 2 (2004), 151--218.

\bibitem{MR0222474}
{\sc H{\"o}rmander, L.}
\newblock Hypoelliptic second order differential equations.
\newblock {\em Acta Math. 119\/} (1967), 147--171.

\bibitem{Kato}
{\sc Kato, T.}
\newblock {\em Perturbation theory for linear operators}.
\newblock Classics in Mathematics. Springer-Verlag, Berlin, 1995.
\newblock Reprint of the 1980 edition.

\bibitem{Mcmp}
{\sc Mouhot, C.}
\newblock Rate of convergence to equilibrium for the spatially homogeneous
  {B}oltzmann equation with hard potentials.
\newblock {\em Comm. Math. Phys. 261}, 3 (2006), 629--672.

\bibitem{MNeu}
{\sc Mouhot, C., and Neumann, L.}
\newblock Quantitative perturbative study of convergence to equilibrium for
  collisional kinetic models in the torus.
\newblock {\em Nonlinearity 19}, 4 (2006), 969--998.

\bibitem{MRS}
{\sc Mouhot, C., Russ, E., and Sire, Y.}
\newblock Fractional {P}oincar\'e inequalities for general measures.
\newblock {\em J. Math. Pures Appl. (9) 95}, 1 (2011), 72--84.

\bibitem{MR0100158}
{\sc Nash, J.}
\newblock Continuity of solutions of parabolic and elliptic equations.
\newblock {\em Amer. J. Math. 80\/} (1958), 931--954.

\bibitem{VillaniTOT}
{\sc Villani, C.}
\newblock {\em Topics in Optimal Transportation}, vol.~58 of {\em Graduate
  Studies in Mathematics series}.
\newblock American Mathematical Society, 2003.

\bibitem{MR2562709}
{\sc Villani, C.}
\newblock Hypocoercivity.
\newblock {\em Mem. Amer. Math. Soc. 202}, 950 (2009), iv+141.

\bibitem{MR1807683}
{\sc Wu, L.}
\newblock Large and moderate deviations and exponential convergence for
  stochastic damping {H}amiltonian systems.
\newblock {\em Stochastic Process. Appl. 91}, 2 (2001), 205--238.

\bibitem{MR2673982}
{\sc Zhang, X.}
\newblock Stochastic flows and {B}ismut formulas for stochastic {H}amiltonian
  systems.
\newblock {\em Stochastic Process. Appl. 120}, 10 (2010), 1929--1949.

\end{thebibliography}

%\begin{thebibliography}{100}
\bigskip
\bigskip

 \signsm \signcm 

\end{document}